\documentclass{amsart}
\usepackage[utf8]{inputenc}
\usepackage{comment}
\usepackage{amsmath}
\usepackage{mathtools}
\usepackage{}
\usepackage{graphicx}
\usepackage[colorlinks=true, allcolors=blue]{hyperref}
\usepackage{amsfonts}
\usepackage{amsthm}
\usepackage{newlfont}
\usepackage{amscd}
\usepackage{amsgen}
\usepackage{amssymb}
\usepackage{mathrsfs}	
\usepackage{longtable}
\usepackage{listings}
\usepackage{extarrows}
\usepackage{booktabs}
\usepackage{caption}
\usepackage{tikz}
\usetikzlibrary{calc}
 
\usepackage{young}
\usepackage{tikz-cd}
\usepackage{verbatim}
\numberwithin{equation}{section}
\usepackage[all]{xy}
\usepackage{color}
\usepackage{amssymb}
\usepackage[left=3cm,right=3cm]{geometry}
\usepackage{tikz}
\usepackage{tikz-cd}
\usepackage{mathtools}
\usepackage{bm} 
\usepackage{dynkin-diagrams}
\usetikzlibrary{matrix,shapes,arrows,decorations.pathmorphing}
\usepackage{calligra}
\usepackage{mathrsfs}
\usepackage{enumitem} 
\usepackage{float}
\restylefloat{table}
\usepackage{ytableau}
\usepackage{adjustbox}
\usetikzlibrary{calc,matrix,positioning}

\usepackage{placeins} 

\RequirePackage{hyperref}
    \hypersetup{
    colorlinks  =   true,
    linkcolor   =   blue,
    filecolor   =   blue,      
    urlcolor    =   blue,
    citecolor   =   blue,
    }

\tikzset{ 
    table/.style={
        matrix of nodes,
        row sep=-\pgflinewidth,
        column sep=-\pgflinewidth,
        nodes={rectangle, text width=2.5em, text height = 1.5em, align=center},
        text depth=1.25ex,
        text height=2.5ex,
        nodes in empty cells
    },
}

\usepackage{comment}
\usepackage{todonotes}
\newcommand{\enrico}[1]{\todo[inline, size=\tiny, author=Enrico, backgroundcolor=blue!20]{#1}}
\newcommand{\claudio}[1]{\todo[inline, size=\tiny, author=Claudio, backgroundcolor=red!20]{#1}}

\let\blb\mathbb

\def\CC{{\blb C}}

\def\PP{{\blb P}}
\def\QQ{{\blb Q}}

\def\ZZ{{\blb Z}}

\DeclareMathOperator{\Gr}{Gr}

\DeclareMathOperator{\oH}{H}
\DeclareMathOperator{\SL}{SL}

\DeclareMathOperator{\Sym}{Sym}

\newcommand{\mE}{\mathcal{E}}
\newcommand{\mF}{\mathcal{F}}

\newcommand{\mO}{\mathcal{O}}
\newcommand{\mQ}{\mathcal{Q}}

\newcommand{\mU}{\mathcal{U}}
\newcommand{\tmQ}{\tilde{\mathcal{Q}}}

\newcommand{\W}{\bigwedge}

\newtheorem{lemma}{Lemma}[section]
\newtheorem{proposition}[lemma]{Proposition}
\newtheorem{thm}[lemma]{Theorem}
\newtheorem{theorem}{Theorem}

\newtheorem{corollary}[lemma]{Corollary}

\newtheorem{definition}[lemma]{Definition}
\newtheorem{conjecture}{Conjecture}
\newtheorem*{conjA}{Conjecture A}
\newtheorem*{conjB}{Conjecture B}

\theoremstyle{remark}
\newtheorem{example}[lemma]{Example}
\newtheorem{remark}[lemma]{Remark}

\newtheorem*{notation*}{Notation}

\def\Hom{\operatorname{Hom}}

\def\End{\operatorname{End}}

\def\Ext{\operatorname{Ext}}
\def\rk{\operatorname{rk}}
\def\ch{\operatorname{ch}}

\def\grass{\operatorname{Gr}}

\DeclareMathOperator{\sHom}{\mathscr{H}\text{\kern -3pt {\calligra\large om}}\,}

\newcommand\quotient[2]{
        \mathchoice
            {
                \text{\raise1ex\hbox{$#1$}\Big/\lower1ex\hbox{$#2$}}%
            }
            {
                #1\,/\,#2
            }
            {
                #1\,/\,#2
            }
            {
                #1\,/\,#2
            }
    }

\newcounter{gensurf}
\title{Modular vector bundles with and without moduli}
\author{Enrico Fatighenti}
\address{\newline
Alma Mater studiorum Universit\`a di Bologna\hfill\newline
Dipartimento di Matematica\hfill\newline
Piazza di porta San Donato 5, 40126 Bologna, Italy}
\email[E.~Fatighenti]{enrico.fatighenti@unibo.it}

\author{Claudio Onorati}
\address{\newline
Alma Mater studiorum Universit\`a di Bologna\hfill\newline
Dipartimento di Matematica\hfill\newline
Piazza di porta San Donato 5, 40126 Bologna, Italy}
\email[C.~Onorati]{claudio.onorati@unibo.it}

\begin{document}
\dedicatory{A Kieran}

\maketitle
\begin{abstract}
    If $X\subset\Gr(2,6)$ is the Fano variety of lines of a smooth cubic fourfold, then we show that the restriction to $X$ of any Schur functor of the tautological quotient bundle is modular and slope polystable. Moreover it is atomic if and only if it is rigid, in which case it is also slope stable. We further compute the Ext-groups of such bundles in infinitely many cases, showing in particular the existence of new modular vector bundles on manifolds of type $\operatorname{K3}^{[2]}$ that are slope stable and whose $\Ext^1$-group is 40-dimensional.
\end{abstract}

\tableofcontents

\section*{Introduction}

In the recent years there has been a lot of interest around the study of special sheaves on irreducible holomorphic symplectic manifolds. Three different definitions are present in literature, each generalising some aspects of the moduli theory of sheaves on K3 surfaces:
\begin{itemize}
    \item in \cite{Verbitsky:JAG1996}, Verbitsky introduced \emph{projectively hyper-holomorphic} vector bundles, and proved that the smooth locus of (an irreducible component of) the moduli space of stable hyper-holomorphic vector bundles (with fixed numerical invariants) is hyper-K\"ahler. On the same lines, in \cite{MeaOno} it is shown that the infinitesimal deformations of a projectively hyper-holomorphic vector bundle are controlled by a formal DG Lie algebra;
    \item in \cite{O'Grady:Modular}, O'Grady introduced \emph{modular} torsion free sheaves, and proved that the (slope) stability problem of such sheaves decomposes the ample cone as in the surface case;
    \item finally in \cite{Beckmann} (see also \cite{Markman:Modular}), Beckmann defined \emph{atomic} objects in the derived category, and proved that they admits an \emph{extended Mukai vector}.
\end{itemize}
It is known that atomic torsion free sheaves and stable projectively hyper-holomorphic vector bundles are modular. When the irreducible holomorphic symplectic manifold $X$ is of type $\operatorname{K3}^{[2]}$, then a stable modular vector bundle is projectively hyper-holomorphic. Nevertheless, there are examples of stable modular sheaves that are not atomic already for this kind of manifolds (as the present paper contributes to show).
\medskip

The brief discussion above shows many important features shared by these special vector bundles, but leaves open a fundamental question, namely the existence of such objects. Very few examples exist in literature, and most of them are rigid. 

When $X$ is of type $\operatorname{K3}^{[2]}$, then \cite{O'Grady:Modular} exhibits many rigid modular vector bundles on $X=S^{[2]}$, all coming from the K3 surface $S$; with a similar technique, in \cite{O'Grady:Modular with many moduli} the author constructs infinitely many examples with $2(a^2+1)$ parameters for any integer $a\geq1$. In \cite{Bottini}, Bottini studies a stable atomic vector bundle on $X=S^{[2]}$ with $10$ parameters, and shows that an irreducible component of the corresponding moduli space is birational to a manifold of type OG10. In \cite{Markman:Modular,Beckmann} other examples of torsion sheaves are exhibited. Finally, in \cite{Fat24} two more examples of modular vector bundles on the Fano varieties of lines of a smooth cubic fourfold are produced, both with $\Ext^1$-group $20$-dimensional.
\medskip

Recall that if $F$ is a torsion free sheaf on a smooth and projective variety $X$, its \emph{discriminant} is defined as 
\[ \Delta(F)=c_1(F)^2-2\operatorname{rk}(F)\ch_2(F). \]
If $X$ is an irreducible holomorphic symplectic variety of type $\operatorname{K3}^{[2]}$, then a torsion free sheaf is modular if $\Delta(F)\in\QQ \mathsf{c}_2(X)$, where $\mathsf{c}_2(X)=c_2(T_X)$ is the second Chern class of $X$. 

The first aim of this work is to expand the list of examples above by proving the following result. 

\begin{theorem}\label{theorem A}
    Let $X\subset\Gr(2,V_6)$ be the Fano variety of lines of a very general cubic fourfold. For any partition $\lambda=(\lambda_1,\lambda_2,\lambda_3,\lambda_4)$, the restricted Schur functor $\Sigma_\lambda\mQ$ on $X$ is (slope) polystable and modular. More precisely,
    \[ \Delta(\Sigma_\lambda\mQ)=\frac{\delta(\lambda)}{4}\operatorname{rk}(\Sigma_\lambda\mQ)^2\mathsf{c}_2(X), \]
    where 
    \[ \delta(\lambda)=\frac{3(\sum_{i=1}^4\lambda_i^2)-2(\sum_{1\leq i<j\leq4}\lambda_i\lambda_j)+12\lambda_1+4\lambda_2-4\lambda_3-12\lambda_4}{60}. \]
\end{theorem}


As remarked before, it is very interesting to understand the relation between modular and atomic sheaves. Roughly speaking atomic sheaves are those for which we can define a surface-like Mukai vector, and conjecturally they deform with $X$ over all its directions, included the non-commutative ones (see Section~\ref{section:atomic}). 

Since we have just seen that the Schur functors of the tautological quotient bundle are modular, it is natural to ask weather they are also atomic. The answer is negative in general, but it is interesting to notice that only those that are rigid are in fact atomic. Our second main result is the following.

\begin{theorem}\label{theorem B}
    Let $X\subset\Gr(2,V_6)$ be the Fano variety of lines of a very general cubic fourfold. For any partition $\lambda=(\lambda_1,\lambda_2,\lambda_3,\lambda_4)$, the associated vector bundle $\Sigma_\lambda\mQ$ on $X$ is atomic if and only if $\lambda_1=\lambda_2=\lambda_3$ or $\lambda_2=\lambda_3=\lambda_4$.

    Moreover, these bundles are all slope stable and rigid.
\end{theorem}
We wish to remark that up to taking dual and twisting with a line bundle, Theorem~\ref{theorem B} is saying that the only atomic bundles are the symmetric bundles $\Sym^m\mQ$, which we fully describe in Section~\ref{section:Sym Q}. Unless the case $m=1$, the rigid and stable atomic vector bundles $\Sym^m\mQ$ do not satisfy the hypothesis of \cite[Theorem~1.4]{O'Grady:Modular}, therefore they are essentially new.
\medskip

The proof of Theorem~\ref{theorem B} is based on Theorem~\ref{thm:Chern character}, where we explicitly compute the Chern character of the bundles $\Sigma_\lambda\mQ$ in function of $\lambda$ and the Chern classes of $\mQ$. To the best of our knowledge this is the first time that Chern classes of Schur functors of a bundle are shown to be related to Chern classes of the bundle itself via functions depending only on the partition. Even though all our computations take place on the Fano variety of lines, thus heavily relying on non-trivial relations in the cohomology ring, one can reproduce the argument in general for any bundle of rank $r$ on a smooth and projective variety. We will face this problem elsewhere.
\medskip


Next we move to study the Ext-algebra of the bundles $\Sigma_\lambda\mQ$. 
The strategy to compute the groups $\Ext^i(\Sigma_\lambda\mQ,\Sigma_\lambda\mQ)$ is to use the famous Borel--Weil--Bott formula combined with the Koszul complex of $X\subset\Gr(2,6)$. This machinery can be encoded in the spectral sequence (\ref{eqn:spectral sequence}), whose degeneracy is hard to control in general. 

Our third main result is an implementation of this strategy for infinitely many partitions.

\begin{theorem}\label{theorem C}
    Let $X\subset\Gr(2,V_6)$ be the Fano variety of lines of a very general cubic fourfold. 
    \begin{enumerate}
    \item $\Sigma_\lambda\mQ$ is simple (hence slope stable) for any $\lambda=(\lambda_1,\lambda_2,\lambda_3,\lambda_4)$ with $\lambda_2-\lambda_4\leq3$; moreover the following table holds:
    \FloatBarrier
    \begin{table}[!htbp]
    \begin{tabular}{lcc} \toprule
$\qquad\lambda$ & $\operatorname{ext}^1(\Sigma_\lambda\mQ,\Sigma_\lambda\mQ)$ & $\operatorname{ext}^2(\Sigma_\lambda\mQ,\Sigma_\lambda\mQ)$  \\
					\midrule
   $(m,0,0,0)$, $m\geq1$   & $0$ & \tiny{$3\left(\frac{3m^2+12m-20}{20} \right)^2 r_m^2-2$} \\  
   $(m,1,0,0)$, $m\geq1$	& $20$ & \tiny{$3\frac{27m^4+180m^3+26m^2-980m+347}{1200}r_{m,1}^2+38$} \\	
   $(m,1,1,0)$, $m\geq1$	& $20$ & \tiny{$3\frac{27m^4+144m^3-136m^2-808m+848}{1200}r_{m,1,1}^2+38$} \\
   $(m,2,0,0)$, $m\geq2$	& $20$ & \tiny{$3\frac{27m^4+144m^3+92m^2-500m-400}{1200}r_{m,2}^2+38$} \\
   $(2,2,1,0)$	& $20$ & $401$ \\
   $(3,2,1,0)$	& $40$ & $35406$ \\
   $(4,2,1,0)$	& $40$ & $554121$ \\
   $(m,2,2,0)$, $m\geq2$	& $20$ & \tiny{$3\frac{27m^4+72m^3-124m^2+4m+248}{1200}r_{m,2,2}^2+38$} \\
   $(3,3,0,0)$	& $20$ & $45338$ \\
   $(4,3,0,0)$	& $20$ & $864545$ \\
   $(5,3,0,0)$	& $20$ & $7136006$ \\
   $(3,3,1,0)$	& $20$ & $45065$ \\
   $(4,3,1,0)$	& $40$ & $996003$ \\
   $(3,3,2,0)$	& $20$ & $23771$ \\
   $(4,3,2,0)$	& $40$ & $554121$ \\
					\bottomrule
				\end{tabular}
\end{table}

\item If $\lambda=(m,t,s,0)$ is one of the partitions above, then the same conclusions hold for partitions of the form
    \[ (m+k,t+k,s+k,k)\quad\mbox{ and }\quad (m+k,m-s+k,m-t+k,k). \]
    
\item For any $\lambda=(\lambda_1,\lambda_2,\lambda_3,\lambda_4)$, we have
\[ \Ext^1(\Sigma_\lambda\mQ,\Sigma_\lambda\mQ)\supset\left\{
    \begin{array}{ll}
      \W^3V_6^\vee   & \mbox{if}\quad\lambda_1=\lambda_2>\lambda_3\,\mbox{or}\,\lambda_1>\lambda_2=\lambda_3>\lambda_4\,\mbox{or}\,\lambda_2>\lambda_3=\lambda_4 \\
      (\W^3V_6^\vee)^{\oplus 2} & \lambda_1>\lambda_2>\lambda_3>\lambda_4\, ,
    \end{array}
    \right.\]
\end{enumerate}
with equality in the cases listed in items $(1)$ and $(2)$.
\end{theorem}

The inclusion in item $(3)$ of Theorem~\ref{theorem C} follow from Proposition~\ref{prop:W3 in Ext1}. In particular, together with item $(1)$, it implies that there are infinitely many partitions $\lambda$ for which $\Sigma_\lambda\mQ$ is slope stable and whose $\Ext^1$-group is at least $40$-dimensional.

The listed partitions for which the dimensions of the $\Ext^1$-groups are determined
are explicitly described in Section~\ref{section:Evidences}: those are the ones for which we can have a relatively good control of the spectral sequence (\ref{eqn:spectral sequence}). In general though this becomes more and more complicated as the first entry of $\lambda$ becomes large. Section~\ref{section:example of non-degeneracy} exhibits an example where we can use explicit geometric arguments to control the differentials of (\ref{eqn:spectral sequence}) and determine the whole $\Ext$-algebra, but this approach seems to be very difficult (if not impossible) to push further. 

Finally, let us point out that the computation of the dimension of $\Ext^2(\Sigma_\lambda\mQ,\Sigma_\lambda\mQ)$ directly follows from the simplicity and the computation of the dimension of $\Ext^1(\Sigma_\lambda\mQ,\Sigma_\lambda\mQ)$ thanks to Proposition~\ref{prop:eulero}, which shows that 
\[ \chi(\Sigma_\lambda\mQ,\Sigma_\lambda\mQ)=3\mathsf{P}(\lambda)\rk(\Sigma_\lambda\mQ)^2, \]
where $\mathsf{P}\in\QQ[x_1,x_2,x_3,x_4]$ is an explicitly determined rational polynomial, see equality (\ref{eqn:chi(Sigma_lambdaQ)}).
\medskip

The proof of Theorem~\ref{theorem C} led us to believe that in general a rigid pattern underlines the structure of $\Ext^1(\Sigma_\lambda\mQ,\Sigma_\lambda\mQ)$. We set the following conjecture.

\begin{conjecture}\label{conjecture:simple+ext 1}
    Let $X\subset\Gr(2,V_6)$ be the Fano variety of lines of a smooth cubic fourfold, and let $\lambda=(\lambda_1,\lambda_2,\lambda_3,\lambda_4)$ be a partition. Then 
    \begin{enumerate}
        \item $\Sigma_\lambda\mQ$ is simple, i.e.\ $\Hom(\Sigma_\lambda\mQ,\Sigma_\lambda\mQ)=\CC$;
        \item if $\lambda$ has at most two different entries (i.e.\ $\Sigma_\lambda\mQ$ is rigid), then 
        \[ \Ext^1(\Sigma_\lambda\mQ,\Sigma_\lambda\mQ)=\left\{
    \begin{array}{ll}
      \W^3V_6^\vee   & \mbox{if}\quad\lambda_1=\lambda_2>\lambda_3\,\mbox{or}\,\lambda_1>\lambda_2=\lambda_3>\lambda_4\,\mbox{or}\,\lambda_2>\lambda_3=\lambda_4 \\
      (\W^3V_6^\vee)^{\oplus 2} & \mbox{otherwise.}
    \end{array}
    \right.\]
    \end{enumerate}    
\end{conjecture}
Notice also that the simplicity claim, together with Theorem~\ref{theorem A}, implies that the vector bundles $\Sigma_\lambda\mQ$ are slope stable.

Another interesting aspect to notice is the fact that $\Ext^1(\Sigma_\lambda\mQ,\Sigma_\lambda\mQ)$ is isomorphic to a $\SL(6)$-representation: this is the content of the Borel--Weil--Bott Theorem for Schur functors on $\Gr(2,6)$, but the fact that it descends to the variety $X$ is quite surprising. (We note that this is not true in general for the group $\Ext^2(\Sigma_\lambda\mQ,\Sigma_\lambda\mQ)$.)

We wish to remark once more that  Conjecture~\ref{conjecture:simple+ext 1} is confirmed in infinitely many cases by Theorem~\ref{theorem C}.
\medskip

Finally, an important aspect of moduli theory is to determine whether or not the Kuranishi space of a polystable vector bundle is smooth. The obstructions to the smoothness are contained in the $\Ext^2$-group, which in the cases considered by us grows indefinitely with $\lambda$ (see Proposition~\ref{prop:eulero}). 

We set an optimistic conjecture about smoothness of the Kuranishi space of the vector bundles $\Sigma_\lambda\mQ$ on $X$.

\begin{conjecture}\label{conjecture:Kuranishi smooth}
    Let $X\subset\Gr(2,V_6)$ be the Fano variety of lines of a smooth cubic fourfold. For any partition $\lambda=(\lambda_1,\lambda_2,\lambda_3,\lambda_4)$, let us consider the associated vector bundle $\Sigma_\lambda\mQ$ on $X$. Then the Kuranishi space 
    \[ \operatorname{Def}(\Sigma_\lambda\mQ) \]
    of infinitesimal deformations of $\Sigma_\lambda\mQ$ is smooth.
\end{conjecture}
We collect in Section~\ref{section:Kuranishi smooth} some little evidences for Conjecture~\ref{conjecture:Kuranishi smooth}.
\medskip

{\bf Plan of the paper.}
 In Section~\ref{section:preliminaries} we state the main definitions, notations and preliminary results we will need in the rest of the paper. In particular, Section~\ref{section:homogeneous} contains a quick recap of the homogeneous vector bundles on the Grassmannian $\Gr(2,6)$ we will consider and the recipes to perform Pieri and Littlewood--Richardson decompositions, and the algorithm for the Borel--Weil--Bott Theorem.
 Section~\ref{section:IHS} is instead a summary of results on irreducible holomorphic symplectic fourfolds, focusing on modular and atomic sheaves. Here we also explicitly determine relations on the cohomology ring of a manifold of type $\operatorname{K3}^{[2]}$ which are probably known to experts, but for which a reference was missing in literature.

 Section~\ref{section:Sym Q} contains the first original results of the paper: we prove here that the symmetric vector bundles $\Sym^m\mQ$ are slope stable, atomic and rigid, and we compute the Ext-algebra explicitly. This will be the starting point of our investigations of more general Schur functors.

 Before tackling this more general problem though, we collect in Section~\ref{section:two worked out example} two worked-out ``easy" examples. We hope that this section will help the reader to navigate the rest of the paper. We prove here that the vector bundles $\Sigma_{(2,1,0,0)}\mQ$ and $\Sigma_{(3,2,1,0)}\mQ$ are slope stable and modular, and we determine their Ext-algebra. In particular it is important to notice that $\dim \Ext^1(\Sigma_{(3,2,1,0)}\mQ,\Sigma_{(3,2,1,0)}\mQ)=40$, providing the first example of a stable modular sheaf on a variety of type $\operatorname{K3}^{[2]}$ with $40$ moduli.

 Section~\ref{section:Chern} contains the long computation of the Chern character of $\Sigma_\lambda\mQ$ in terms of $\lambda$ and the Chern classes of $\mQ$. This is a technical, but not difficult, section based on several inductive steps on the rows and columns of $\lambda$.

 Section~\ref{section:proof of thm A} and Section~\ref{section:proof of thm B} are the first two main sections, in which we prove Theorem~\ref{theorem A} and Theorem~\ref{theorem B}, respectively. In particular, in Section~\ref{section:euler char} we compute the Euler characteristic of $\Sigma_\lambda\mQ$ as a function of $\lambda$.

 In Section~\ref{section:Ext groups} we state Conjecture~\ref{conjecture:simple+ext 1} and prove some general facts in its support; in Section~\ref{section:Evidences} we prove Theorem~\ref{theorem C} via a case-by-case analysis of the several cases.

 Finally, in Section~\ref{section:Kuranishi smooth} we make some speculations about the smoothness of the Kuranishi space of $\Sigma_\lambda\mQ$ and set Conjecture~\ref{conjecture:Kuranishi smooth}.

 To conclude the paper we also include two appendixes. In Appendix~\ref{section:U} we quickly prove that the vector bundles $\Sigma_\mu\mU$ are not modular, thus justifying our choice to concentrate on Schur functors of the tautological quotient bundle. In Appendix~\ref{section:congetture verdure} we collect several combinatorial questions and remarks about the Littlewood--Richardson decomposition of endomorphism bundles. These facts have been noticed while performing the computations in the paper, especially in Section~\ref{section:Evidences}, but a general proof is lacking, so that we leave them as open questions.

\subsection*{Acknowledgements}
The authors wish to thank Valeria Bertini, Sasha Kuznetsov, Laurent Manivel, Francesco Meazzini and Kieran O'Grady for useful conversations at several stages of this work.

This research has been partially funded by the European Union - NextGenerationEU under the National Recovery and Resilience Plan (PNRR) - Mission 4 Education and research - Component 2 From research to business - Investment 1.1 Notice Prin 2022 - DD N. 104 del 2/2/2022, from title ``Symplectic varieties: their interplay with Fano manifolds and derived categories", proposal code 2022PEKYBJ – CUP J53D23003840006.

Both authors are members of the INDAM-GNSAGA group.

\section{Preliminaries}\label{section:preliminaries}

\subsection{Homogeneous vector bundles on $\grass(2,6)$}\label{section:homogeneous}

The grassmannian $\grass(2,6)$ is the variety of two dimensional subspaces in $V_6$, a six dimensional complex vector space. It is a homogeneous variety, being the quotient of $\SL(6)$ by a parabolic subgroup (containing $\SL(4)\times\SL(2)$).
The homogeneous vector bundles we consider in this manuscript will come from representations of $\SL(4)\times\SL(2)$.

Recall that irreducible representations of $\SL(N)$ are indexed by \emph{dominant weights} $\lambda=(\lambda_1,\cdots,\lambda_N)\in\ZZ^N$ (recall that dominant means that $\lambda_1\geq\cdots\geq\lambda_N$). If $\lambda_N\geq0$, then $\lambda$ will be called a partition. The integer number $N$ is the \emph{length} of the weight. If $\lambda$ is a partition of length $N$, then its \emph{size} is 
\[ |\lambda|=\lambda_1+\cdots+\lambda_N. \]

Now, if $V_N$ is a vector space of rank $N$ where $\SL(N)$ acts naturally, then we denote by
\[ \Sigma_\lambda V_N \]
the \emph{Schur} module corresponding to the irreducible representation with maximal weight $\lambda$.

\begin{example}\label{example:sym e wedge}
   \begin{enumerate} 
    \item If $\lambda=(m,0,\cdots,0)$, then $\Sigma_\lambda V_N=\Sym^mV_N$ is the symmetric product representation;
    \item if $\lambda=(1,\cdots,1,0,\cdots,0)$, where the number of $1$'s is $n$, then $\Sigma_\lambda V_N=\W^n V_N$ is the wedge product representation.
    \end{enumerate}
\end{example}
    
With this notation, we denote by 
\[ \Sigma_{\lambda|\mu}:=\Sigma_\lambda V_4\otimes\Sigma_\mu V_2 \]
the irreducible representation of $\SL(4)\times\SL(2)$ with maximal weights $\lambda$ and $\mu$.

In the rest of the paper we use the following notation.
\begin{notation*}
Working on $\Gr(2,6)$, we have:
\begin{itemize}
    \item the partition $\lambda=(1,0,0,0)$, corresponding to the tautological quotient vector bundle, is denoted by $\widetilde{\mQ}$;
    \item for any dominant weight $\lambda$ of length $4$, we denote by 
    \[ \Sigma_\lambda\widetilde{\mQ} \] 
    the irreducible homogeneous vector bundles corresponding to the Schur module $\Sigma_{\lambda|(0,0)}$;
    \item the partition $\mu=(1,0)$, corresponding to the tautological subspace vector bundle, is denoted by $\widetilde{\mU}$;
    \item for any dominant weight $\mu$ of length $2$, we denote by 
    \[ \Sigma_\mu\widetilde{\mU} \] 
    the irreducible homogeneous vector bundles corresponding to the Schur module $\Sigma_{(0,0,0,0)|\mu}$;
    \item finally, with an abuse of notation we keep denoting by 
    \[ \Sigma_{\lambda|\mu}:=\Sigma_\lambda\widetilde{\mQ}\otimes\Sigma_\mu\widetilde{\mU} \]
    the irreducible homogeneous vector bundle associated to the corresponding Schur module.
\end{itemize}
\end{notation*}

If $\lambda=(\lambda_1,\cdots,\lambda_N)$ is a dominant weight of length $N$, then we define its \emph{dual} as
\[ \lambda^*:=(-\lambda_N,\cdots,-\lambda_1). \]
Then for any Schur module $\Sigma_{\lambda|\mu}$ on $\Gr(2,6)$ we have
\begin{equation}\label{eqn:duale}
    (\Sigma_{\lambda|\mu})^\vee=\Sigma_{\lambda^*|\mu^*}.
\end{equation}

If we put
$\lambda+n:=(\lambda_1+n,\cdots,\lambda_N+n)$
for an integer $n\in\ZZ$,
then 
\begin{equation}\label{eqn: x n} \Sigma_{\lambda+n}=\Sigma_\lambda V_N\otimes\left(\W^N V_N\right)^{\otimes n}. 
\end{equation}

Since 
\[ \Sigma_{(1,1,1,1)}\widetilde{\mQ}=\det\widetilde{\mQ}=\mO_{\Gr(2,6)}(1)\quad\mbox{ and }\quad \Sigma_{(1,1)}\widetilde{\mU}=\det\widetilde{\mU}=\mO_{\Gr(2,6)}(-1), \]
from (\ref{eqn: x n}) it follows that 
\begin{equation}\label{eqn: + n}
\Sigma_{\lambda+n|\mu+n}=\Sigma_{\lambda|\mu}
\end{equation}
for every $n\in\ZZ$.

\begin{example}\label{example:duale}
Let $\lambda$ be a partition of length $4$. Then  
\[ \Sigma_\lambda\widetilde{\mQ}^\vee=\Sigma_{\lambda'}\widetilde{\mQ}\otimes\mO_{\Gr(2,6)}(-\lambda_1), \]
where $\lambda'=(\lambda_1-\lambda_4,\lambda_1-\lambda_3,\lambda_1-\lambda_2,0)$.
In fact, combining (\ref{eqn:duale}) with (\ref{eqn: + n}) it follows that $\Sigma_\lambda\widetilde{\mQ}^\vee=\Sigma_{\lambda^*}\widetilde{\mQ}=\Sigma_{\lambda^*+\lambda_1|(\lambda_1,\lambda_1)}$.
\end{example}

\begin{remark}
    As a consequence of (\ref{eqn: + n}), up to tensoring with a suitable ample line bundle $\mO_{\Gr(2,6)}(n)$ we can restrict ourselves to considering only the bundles $\Sigma_\lambda\widetilde{\mQ}$ where $\lambda$ is a partition.
\end{remark}

Because of this remark, from now on we only work with homogeneous vector bundles $\Sigma_\lambda\widetilde{\mQ}$ where $\lambda$ is a partition. In particular we can picture $\lambda$ as a Young diagram.
\medskip

We state in the following the main technical results we will need to decompose plethysms and compute cohomology of irreducible homogeneous vector bundles on $\Gr(2,6)$. These are standard facts and can be found for example in \cite{Weyman}.
\medskip

{\bf Pieri's decomposition.}
Let $\lambda$ be a partition of length $N$, which we think as a Young diagram. Pieri's formula says that
\begin{equation}\label{eqn:Pieri} \Sigma_\lambda\widetilde{\mQ}\otimes\Sym^m\widetilde{\mQ}=\bigoplus_{\nu}\Sigma_\nu\widetilde{\mQ},
\end{equation}
where the direct sum runs over all the Young diagrams $\nu$ obtained from $\lambda$ by adding $m$ blocks subject to the following rule. 
\begin{center}
    $\bullet$ \emph{Pieri rule}: no two blocks are added on the same column.
\end{center}

The same formula holds if we replace $\widetilde{\mQ}$ with $\widetilde{\mU}$.
\medskip

{\bf Littlewood--Richardson's decomposition.}
As before, let $\lambda$ and $\mu$ be two partitions of length $N$, which we think as Young diagrams. Littlewood--Richardson's formula says that
\begin{equation}\label{eqn:LR} \Sigma_\lambda\widetilde{\mQ}\otimes\Sigma_\mu\widetilde{\mQ}=\bigoplus_{\nu}(\Sigma_\nu\widetilde{\mQ})^{\oplus m^\nu_{\lambda,\mu}}, 
\end{equation}
where the direct sum runs over all the Young diagrams $\nu$ such that $|\nu|=|\lambda|+|\mu|$, and the multiplicities $m^\nu_{\lambda,\mu}$ are determined according to the following algorithm. Let us write $\mu=(\mu_1,\mu_2,\mu_3,\mu_4)$ and $|\mu|=\mu_1+\mu_2+\mu_3+\mu_4$. Starting with the Young diagram of $\lambda$, we want to arrive to the Young diagram of $\nu$ using the following steps:
\begin{center}
    $\bullet$ \emph{Step~i}: add $\mu_i$ boxes labelled with the number $i$, following the Pieri rule;
\end{center}
for $i=1,\cdots,N$.

At the end of this process, let us list the numbers inserted starting from the top-right box and proceeding right-to-left and top-to-bottom. Let us call $S$ this sequence. Then the following rule must be satisfied.
\begin{center}
    $\bullet$ \emph{Littlewood--Richardson rule}: for every $0\leq t\leq|\mu|$, we look at the first $t$ elements of the sequence $S$ and for every $p=1,2,3,4$ we must have that the number of $p$'s is bigger or equal to the number of $(p+1)$'s. 
\end{center}

The multiplicity $m^\nu_{\lambda,\mu}$ is then the number of admissible sequences $S$ satisfying the Littlewood--Richardson rule.

\begin{remark}\label{remark:tutti 1}
    The rule in Step~5 in particular says that in the first row we must insert only boxes labelled with $1$.
\end{remark}

\begin{remark}\label{rmk:strictly increasing on colomns}
    Since we add boxes labelled progressively at each step following the Pieri rule, it follows that at end on each column we must read a (strictly) increasing sequence of integers.
\end{remark}

\subsubsection{The Borel--Weil--Bott Theorem}\label{section:BBW}
Let $\Sigma_{\lambda|\mu}$ be an irreducible homogeneous vector bundle on $\Gr(2,6)$, with $\lambda$ and $\mu$ two partitions. In the following we abuse notation and consider $\lambda|\mu$ as a $6$-tuple of positive integers. 

Let $\delta=(5,4,3,2,1,0)$ be the \emph{fundamental dominant weight} for $\SL(6)$. If $v$ is a $6$-tuple of integral numbers, then the $l$-function of $v$ is defined as  
\begin{equation}\label{eqn:l function}
    l(v):=\mbox{minimum number of transposition needed to bring $l(v)$ to a dominant partition}.
\end{equation}

Then the Borel--Weil--Bott Theorem is based on the following algorithm.
\begin{itemize}
    \item If $\lambda|\mu + \delta$ has repeated entries, then $\Sigma_{\lambda|\mu}$ is acyclic, i.e.\ $\oH^p(\Sigma_{\lambda|\mu})=0$ for every $p$.
    \item If $\lambda|\mu + \delta$ has no repeated entries, then
    \[ \oH^{p}(\Sigma_{\lambda|\mu})=\left\{
    \begin{array}{cc}
      \Sigma_{\widetilde{\lambda|\mu}} V_6   &  \mbox{if}\; p=l(\lambda|\mu+\delta) \\
       0  & \mbox{otherwise}
    \end{array}\right.\]
    where $\widetilde{\lambda|\mu}:=\operatorname{sort}(\lambda|\mu+\delta)-\delta$.
\end{itemize}

\subsubsection{The reduction trick}\label{section:reduction}
In this paper we will work with homogeneous vector bundles of the form $\Sigma_\lambda\widetilde{\mQ}$, where $\lambda=(\lambda_1,\lambda_2,\lambda_3,\lambda_4)$ is a partition. 

In we apply Example~\ref{example:duale} twice, we get the following identity
\begin{equation}\label{eqn:reduction}
    \Sigma_\lambda\widetilde{\mQ}=\Sigma_{\lambda'}\widetilde{\mQ}\otimes\mO_{\Gr(2,6)}(\lambda_4),
\end{equation}
where $\lambda'=(\lambda_1-\lambda_4,\lambda_2-\lambda_4,\lambda_3-\lambda_4,0)$. 

This simple observation will make us able to reduce most of our computations (for example of discriminants or Euler characteristic) to partitions of the form $\lambda=(\lambda_1,\lambda_2,\lambda_3,0)$, thus shortening or simplifying the proofs.


\subsection{Irreducible holomorphic symplectic fourfolds}\label{section:IHS}
The main character of this work is the Fano variety of lines of a smooth cubic fourfold. As it is well known, such a variety is an irreducible holomorphic symplectic manifold of dimension $4$, i.e.\ it is a compact (more precisely, projective) simply connected K\"ahler manifold with a unique-up-to-scalar non-degenerate holomorphic $2$-form. We refer to \cite{GHJ} for a general account about irreducible holomorphic symplectic manifolds, but we recall in the following the main facts and properties we will need later. We further restrict ourselves to fourfolds: this restriction is not important, but it simplifies a bit the discussion and it is the only case we will consider later.

\subsubsection{The Fano variety of lines of a cubic fourfold}\label{section:F(Y)}
Let $Y=\{f_3(x_0,\cdots,x_5)=0\}\subset\PP^5$ be a smooth cubic fourfold. It is well known (\cite{BD}) that the Fano variety of lines $X=F(Y)$ is an irreducible holomorphic symplectic fourfold deformation equivalent to an Hilbert scheme of $2$ points on a K3 surface. 

From now on, as it is customary, we refer to such irreducible holomorphic symplectic fourfolds as \emph{manifolds of type} $\operatorname{K3}^{[2]}$.

Let us recall (see \cite{BD}) that $\oH^4(F(Y),\QQ)=\Sym^2\oH^2(F(Y),\QQ)$ and that the Hodge diamond of $F(Y)$ is 
\[
\begin{array}{ccccccccc}
      &   &    &   &  1  &   &    &   &  \\
      &   &    & 0 &     & 0 &    &   & \\
      &   & 1  &   &  21 &   & 1  &   & \\ 
      & 0 &    & 0 &     & 0 &    & 0 & \\
    1 &   & 21 &   & 232 &   & 21 &   & 1 
\end{array}
\]

The variety $F(Y)$ can also be described as the zero locus $X\subset\Gr(2,6)$ of a general section of the homogeneous vector bundle $\Sym^3\tilde{\mU}^\vee$. Here and in the following we denote by $\tilde{\mU}$ and $\tilde{\mQ}$ the tautological bundles on $\Gr(2,6)$ of the subspaces and the quotients, respectively. 

One of the main objects of our investigations is the restriction $\mQ$ to $X$ of the tautological quotient bundle $\tilde{\mQ}$. 
We collect here the following identities, which will be very useful later. 

We denote by $\mathsf{c}_i(X)=c_i(T_X)\in\oH^{2i}(X,\ZZ)$ the Chern classes of $X$ and by $\mathsf{p}\in\oH^8(X,\ZZ)$ the class of a point.

\begin{lemma}\label{lemma:c_2(X)^2 and ch_4(Q)}
    The following relations hold.
    \begin{enumerate}
        \item $\int_X \mathsf{c}_2(X)^2=828$;
        \item $\operatorname{td}_X=\left(1,0,\frac{1}{12}\mathsf{c}_2(X),0,3\mathsf{p}\right)$;
        \item $\sqrt{\operatorname{td}_X}=\left( 1,0,\frac{1}{24}\operatorname{c}_2(X),0,\frac{25}{32}\mathsf{p}\right)$;
        \item $\Delta(\mQ)=c_1(\mQ)^2-8\ch_2(X)=\mathsf{c}_2(X)$.
    \end{enumerate}
\end{lemma}


\begin{proof}
    Item $(1)$ is \cite[Formula~5.3.4]{O'Grady:Modular}. Item $(2)$ follows from it and from the fact that the Euler characteristic of $X$ is $324$ (\cite[Proposition~3]{BD}; see also \cite[Corollary~2.10 (b)]{Gottsche}). Item $(3)$ is a direct computation. Finally, item $(4)$ is also a direct computation, see for example \cite[Section~2.1]{O'Grady:Modular}.
\end{proof}

\begin{lemma}\label{lemma:identities ch_3(Q) e ch_4(Q)}
    Assume that $X=F(Y)$ is very general, i.e.\ $\oH^{1,1}(X,\QQ)$ has rank $1$ and is generated by $c_1(\mQ)$.
    \begin{enumerate}
    \item The group $\oH^{3,3}(X,\QQ)$ has rank $1$ and the following relations hold:
    \begin{enumerate}
        \item $c_1(\mQ)^3=-24\ch_3(\mQ)$;
        \item $c_1(\mQ).\ch_2(\mQ)=2\ch_3(\mQ)$.
    \end{enumerate}
    \item The group $\oH^{8}(X,\QQ)$ has rank $1$ and the following relations hold:
    \begin{enumerate}
        \item $c_1(\mQ)^4=144\ch_4(\mQ)$;
        \item $c_1(\mQ).\ch_3(\mQ)=-6\ch_4(\mQ)$;
        \item $c_1(\mQ)^2.\ch_2(\mQ)=-12\ch_4(\mQ)$;
        \item $\ch_2(\mQ)^2=12\ch_4(\mQ)$;
        \item $\mathsf{c}_2(X)^2=1104\ch_4(\mQ)$;
        \item $\int_X\ch_4(\mQ)=\frac{3}{4}$.
    \end{enumerate}
    \end{enumerate}
\end{lemma}
\begin{proof}
    Let us consider the tangent-normal bundle sequence
    \[ 0\to T_X\to (\tilde{\mQ}\otimes\tilde{\mU})|_X \to \Sym^3\mU^\vee\to 0. \]
    By a direct computation we have
    \begin{align*} 
    \ch((\tilde{\mQ}\otimes\tilde{\mU})|_X) = & \left( 8, 6 c_1(\mQ), c_1(\mQ)-2\ch_2(\mQ), -\frac{1}{6}c_1(\mQ)^3-c_1(\mQ)\ch_2(\mQ)+4\ch_3(\mQ),\right. \\
     & \left.-\frac{1}{24}c_1(\mQ)^4-\frac{1}{6}c_1(\mQ)^2\ch_2(\mQ)-\frac{5}{6}\ch_2(\mQ)^2+2c_1(\mQ)\ch_3(\mQ)\right)
    \end{align*}

and

\begin{align*} 
    \ch(\Sym^3 \mU^\vee) = & (4, 6 c_1(\mQ), 2c_1(\mQ)^2-10\ch_2(\mQ),-6c_1(\mQ)\ch_2(\mQ) +18 \ch_3(\mQ), \\
    & -\frac{3}{8}c_1(\mQ)^4+\frac{29}{12}c_1(\mQ)^2\ch_2(\mQ)+ \frac{1}{3}\ch_2(\mQ)^2-\frac{11}{2}c_1(\mQ)\ch_3(\mQ))
\end{align*}

from which we deduce
\begin{align*}
    \ch(T_X) = & (4,0,-c_1(\mQ)^2+8\ch_2(\mQ),-\frac{1}{6}c_1(\mQ)^3+5c_1(\mQ)\ch_2(\mQ)-14\ch_3(\mQ), \\
    & \frac{1}{3}c_1(\mQ)^4-\frac{31}{12}c_1(\mQ)^2\ch_2(\mQ)-\frac{7}{6}\ch_2(\mQ)^2+\frac{15}{2}c_1(\mQ)\ch_3(\mQ)).
\end{align*}

Since the odd Chern classes for $X$ vanish, we have $\ch_3(T_X)=0$, which already gives
\begin{equation}\label{eqn:rel on H 6} 
-\frac{1}{6}c_1(\mQ)^3+5c_1(\mQ)\ch_2(\mQ)-14\ch_3(\mQ)=0. 
\end{equation}
By

Now, by \cite[Proposition 2.10]{Fat24} we know that $c_1(\mQ)c_3(\mQ)=18$; using the well-known fact that $\int_X c_1(\mQ)^4=108$, this equality can be rewritten as
\begin{equation}\label{eqn:Enrico}
c_1(\mQ)^2\ch_2(\mQ)=2 c_1(\mQ)\ch_3(\mQ),
\end{equation}
which gives part $(1.b)$ of the statement. Substituting in \ref{eqn:rel on H 6} we obtain $(1.a)$.

Let us move to item $(2)$. By items $(1)$ and $(2)$ of Lemma~\ref{lemma:c_2(X)^2 and ch_4(Q)} we have $\int_X c_2(T_X)^2 = 828$ and $\int_X c_4(T_X)=324$. These imply that $\int_X \ch_4(T_X)=15$. Hence, 
\begin{equation}\label{eqn:rel on H 8}
 \frac{1}{3}\int_X c_1(\mQ)^4-\frac{31}{12}\int_X c_1(\mQ)^2\ch_2(\mQ)-\frac{7}{6}\int_X \ch_2(\mQ)^2+\frac{15}{2}\int_X c_1(\mQ)\ch_3(\mQ)=15.
\end{equation}

By $(1.a)$ we get $\int_X c_1(\mQ)\ch_3(\mQ)=-\frac{9}{2}$ and, equivalently, that $\int_X c_1(\mQ)^2\ch_2(\mQ)=-9$. Putting all together, this gives
\[ \int_X \ch_2(\mQ)^2=9. \]

Finally, by Hirzebruch--Riemann--Roch, we have
\[
6=\chi(\mQ)= \int_X \ch(\mQ)\operatorname{td}_X=12+\frac{1}{12}\mathsf{c}
_2(X)\ch_2(\mQ)+\ch_4(\mQ). \]
By item $(4)$ of Lemma~\ref{lemma:c_2(X)^2 and ch_4(Q)} $$\mathsf{c}_2(X)\ch_2(\mQ)=c_1(\mQ)^2\ch_2(\mQ)-8 \ch_2(\mQ)^2=-81,$$

from which the equality $\ch_4(\mQ)=\frac{3}{4}$ follows. All the relations in part (2) are therefore proven.
\end{proof}

\subsubsection{LLV algebra and Mukai lattice}\label{section:LLV}
Let $X$ be an irreducible holomorphic symplectic fourfold.

Let us start by recalling the definition of the Looijenga--Lunts--Verbitsky (LLV) algebra. First of all, we denote by $h\in\operatorname{End}(\oH^\ast(X,\QQ))$ the grading endomorphism, i.e.\ $h(y)=(k-4)\operatorname{id}$ for any $y\in\oH^k(X,\QQ)$. Also, for any element $x\in\oH^2(X,\QQ)$ we denote by $e_x\in\operatorname{End}(\oH^\ast(X,\QQ))$ the operator $e_x(y)=x\cup y$. Then one says that $x$ satisfies the Hard Lefschetz property if there exists another operator $f_x\in\operatorname{End}(\oH^\ast(X,\QQ))$ such that $(e_x,h,f_x)$ is a basis for the Lie algebra $\mathfrak{sl}_2$. One usually refers to a such triple $(e_x,h,f_x)$ as an $\mathfrak{sl}_2$-triple.

The \emph{LLV algebra} is the Lie algebra $\mathfrak{g}(X)\subset\operatorname{End}(\oH^\ast(X,\QQ))$ generated by all the $\mathfrak{sl}_2$-triples associated to elements of $\oH^2(X,\QQ)$ with the strong Lefschetz property (see \cite{LL}).

Let us denote by $\operatorname{SH}(X,\QQ)\subset\oH^\ast(X,\QQ)$ the image of $\Sym\oH^2(X,\QQ)$ via cup product. By \cite{Bogomolov,LLVerbitsky} we have that $\operatorname{SH}(X,\QQ)$ is an irreducible $\mathfrak{g}(X)$-module of $\oH^\ast(X,\QQ)$, called the \emph{Verbistky component}, appearing with multiplicity $1$.

Finally, let us also recall the other useful description of the LLV algebra. The \emph{extended Mukai (rational) lattice} is 
\[ \widetilde{\oH}(X,\QQ)=\left(\QQ\alpha\oplus\oH^2(X,\ZZ)\oplus\QQ\beta, \;\tilde{q}\right), \]
where $\tilde{q}(\alpha)=\tilde{q}(\beta)=0$, $\tilde{q}(\alpha,\beta)=-1$, $\tilde{q}|_{\oH^2(X,\ZZ)}$ is the Beauville--Bogomolov--Fujiki form and $\oH^2(X,\QQ)$ is $\tilde{q}$-orthogonal to $\QQ\alpha\oplus\QQ\beta$. 

For any $x\in\oH^2(X,\QQ)$ with the strong Lefschetz property, putting
\begin{equation}\label{eqn:def of e on H tilde} 
e_x(\alpha)=x,\quad e_x(y)=q_X(x,y)\beta\quad\mbox{ and }\quad e_x(\beta)=0 
\end{equation}
defines a structure of $\mathfrak{g}(X)$-module on $\widetilde{\oH}(X,\QQ)$ and moreover (\cite{LLVerbitsky})
\[ \mathfrak{g}(X)\cong\mathfrak{so}(\widetilde{\oH}(X,\QQ)). \]

\subsubsection{Modular sheaves}\label{section:Modular and atomic}
Recall that the Verbitsky component $\operatorname{SH}(X,\QQ)$ is the irreducible $\mathfrak{g}(X)$-submodule of $\oH^\ast(X,\QQ)$ generated by $\oH^2(X,\QQ)$. 

\begin{definition}
    Let $X$ be an irreducible holomorphic symplectic manifold.
    We denote by $\mathsf{q}_X\in\operatorname{SH}^4(X,\QQ)$ the BBF-class of $X$, i.e.\ $\mathsf{q}_X$ is uniquely determined by the property that 
        \[ \int_X\gamma^2\mathsf{q}_X=q_X(\gamma) \]
        for every $\gamma\in\oH^2(X,\ZZ)$.
\end{definition}

\begin{definition}[O'Grady]\label{def:modular}
    Let $F$ be a torsion free coherent sheaf on $X$.
    Then $F$ is \emph{modular} if 
        \[ \Delta(F)=c_1(F)^2-2\rk(F)\ch_2(F)\in\QQ\,\mathsf{q}_X. \]
\end{definition}
\begin{remark}
    In \cite{O'Grady:Modular}, O'Grady defines a modular sheaf as a torsion free sheaf $F$ such that the projection of $\Delta(F)$ into the Verbitsky component is a multiple of $\mathsf{q}_X$. Our definition implies his definition, but the vice versa is not true in general. On the other hand,if $X$ is of type $\operatorname{K3}^{[2]}$, then the Verbitsky component coincides with the rational cohomology ring. Since we will only work with fourfolds of type $\operatorname{K3}^{[2]}$, we choose to state the definition in this weaker form.
\end{remark}

Let us recall the following known fact.
\begin{lemma}\label{lemma:c_2=30q}
    If $X$ is of type $\operatorname{K3}^{[2]}$, then $\mathsf{c}_2(X)=30\mathsf{q}_X$. 

    In particular a torsion free sheaf $F$ on $X$ is modular if and only if $\Delta(F)\in\QQ\mathsf{c}_2(X)$.
\end{lemma}
\begin{proof}
    First of all, by \cite[Lemma~1.5]{Markman:BBF} we know that $\mathsf{c}_2(X)$ is proportional to $\mathsf{q}_X$. Now, by a result of Fujiki (\cite{Fujiki}) there is a constant $C(\mathsf{c}_2(X))$ such that
    \[ \int_X\gamma^2\mathsf{c}_2(X)=C(\mathsf{c}_2(X))q_X(\alpha) \]
    for every $\gamma\in\oH^2(X,\ZZ)$. From \cite[Corollary~2.7, Example~2.12]{BeckmannSong} it follows that $C(\mathsf{c}_2(X))=30$. The claim follows.
\end{proof}

\begin{example}[Tautological bundles on $F(Y)$]\label{example:Q modular}
    Let $X=F(Y)$ be the Fano variety of lines of a cubic fourfold (see Section~\ref{section:F(Y)}) and $\mQ$ the restriction to $X$ of the tautological quotient bundle on $\Gr(2,6)$. We have already seen in item $(4)$ of Lemma~\ref{lemma:c_2(X)^2 and ch_4(Q)} that $\Delta(\mQ)=\mathsf{c}_2(X)$, so that $\mQ$ is modular.

    In a similar way, if $\mU$ is the restriction to $X$ of the tautological bundle of subspaces of $\Gr(2,6)$, then it is easy to see that 
    \[ \Delta(\mU)=\frac{3}{2}c_1(\mU)^2-\frac{1}{2}\mathsf{c}_2(X), \]
    so that $\mU$ is not modular.
\end{example}

\subsubsection{Atomic sheaves}\label{section:atomic}
Recall from Section~\ref{section:LLV} the definition of the LLV algebra $\mathfrak{g}(X)$ and its action on the cohomology ring $\oH^\ast(X,\QQ)$, the Verbitsky component $\operatorname{SH}(X,\QQ)$ and the extended Mukai lattice $\widetilde{\oH}(X,\QQ))$.

To these actions we also add the action of $\mathfrak{g}(X)$ on $\Sym^2\widetilde{\oH}(X,\QQ)$ by derivation. Then by \cite{Bogomolov,LLVerbitsky} 
there is a short exact sequence of $\mathfrak{g}(X)$-modules
\[ 0\longrightarrow\operatorname{SH}(X,\QQ)\stackrel{\psi}{\longrightarrow}\Sym^2\widetilde{\oH}(X,\QQ)\stackrel{\tilde{q}}{\longrightarrow}\QQ\longrightarrow 0, \]
where $\psi$ is normalised as $\psi(1)=\alpha^{(2)}/2$.

The $\QQ$-vector space $\operatorname{SH}(X,\QQ)$ has the natural bilinear product defined by cup-product and the property that $\operatorname{SH}^{2p}(X,\QQ)$ is orthogonal to $\operatorname{SH}^{2q}(X,\QQ)$ if $p+q\neq2$. On the other hand
one can define a bilinear non-degenerate pairing on $\Sym^2\widetilde{\oH}(X,\QQ)$ by the rule
\[ \tilde{q}^{(2)}((v_1,v_2),(w_1,w_2))=\tilde{q}(v_1,w_1)\tilde{q}(v_2,w_2)+\tilde{q}(v_1,w_2)\tilde{q}(v_2,w_1). \]
The morphism $\psi$ is then an isometry with respect to these products (\cite{Taelman,Beckmann:Mukai}) and we denote by 
\[ \mathsf{T}\colon\Sym^2\widetilde{\oH}(X,\QQ)\longrightarrow\operatorname{SH}(X,\QQ) \]
the corresponding orthogonal projection.
The morphism $\mathsf{T}$ is a morphism of $\mathfrak{g}(X)$-modules.

Recall that the Mukai vector of a coherent sheaf $F$ is 
\[ v(F)=\ch(F)\sqrt{\operatorname{td}_X}\in\oH^\ast(X,\QQ). \]
\begin{definition}[Beckmann, Markman]
    Let $F$ be a coherent sheaf on $X$. 
    \begin{enumerate}
        \item (\cite[Definition~4.15]{Beckmann:Mukai}) An \emph{extended Mukai vector} of $F$ is an element $\tilde{v}\in\widetilde{\oH}(X,\QQ)$ such that 
        \[ \mathsf{T}(\tilde{v}^{(2)})\in\QQ\, \bar{v}(F), \]
        where $\bar{v}(F)$ is the projection of the Mukai vector $v(F)$ to the irreducible component $\operatorname{SH}(X,\QQ)$.
        \item (\cite[Definition~1.1]{Beckmann}) $F$ is \emph{atomic} if there is a vector $0\neq\tilde{v}$ such that
        \[ \operatorname{Ann}(\tilde{v})=\operatorname{Ann}(v(F)). \]
    \end{enumerate}
\end{definition}


By \cite[Proposition~3.3]{Beckmann} (see also \cite[Theorem~1.7]{Markman:Modular}), if $F$ is an atomic sheaf and $\tilde{v}\in\widetilde{\oH}(X,\QQ)$ is an element such that $\operatorname{Ann}(\tilde{v})=\operatorname{Ann}(v(F))$, then $\tilde{v}$ is an extended Mukai vector for $F$. 

\begin{remark}
    By \cite[Theorem~1.2]{Beckmann}, a sheaf $F$ is atomic if and only if it has a rank $1$ cohomological obstruction map. Recall that the cohomological obstruction map is
    \[ \operatorname{HT}^2(X)\longrightarrow \oH^\ast(\Omega_X^\ast),\qquad \mu\mapsto \mu\lrcorner\, v(F). \]

    Conjecturally, if the sheaf $F$ is simple, this should be related to $F$ deforming with $X$ along any directions, included the gerby and non-commutative ones, see \cite[Theorem~1.4]{Beckmann}.
\end{remark}

Let us remark that the extended Mukai vector is only defined up to a (rational) constant. Nevertheless if the rank of $F$ does not vanish, then by \cite[Proposition~3.8]{Beckmann} (see also \cite[Theorem~6.13]{Markman:Modular}) there is a natural normalisation given by
\begin{equation}\label{eqn:v tilde normalizzato} 
\tilde{v}(F)=\left(\rk(F), c_1(F), s(F)\right), 
\end{equation}
for some $s(F)\in\QQ$. 

The following result is \cite[Proposition~3.1, Proposition~3.3]{Beckmann} in the particular case of an irreducible holomorphic symplectic fourfold of type $\operatorname{K3}^{[2]}$.
\begin{proposition}\label{prop:se ann=ann allora atomico}
    Let $X$ be an irreducible holomorphic symplectic fourfold of type $\operatorname{K3}^{[2]}$, and $F$ a torsion free coherent sheaf on $X$. If $F$ has an extended Mukai vector $\tilde{v}(F)$, then it is atomic.
\end{proposition}
\begin{proof}
    First of all, by \cite[Proposition~3.1, Proposition~3.3]{Beckmann} the following equivalence holds,
    \[ T(\tilde{v}^{(2)})\in\QQ\,\bar{v}(E)\quad\Longleftrightarrow\quad \operatorname{Ann}(\mathsf{T}(\tilde{v}^{(2)}))=\operatorname{Ann}(\bar{v}(E)), \]
    and $\operatorname{Ann}(\mathsf{T}(\tilde{v}^{(2)}))=\operatorname{Ann}(\tilde{v})$.

    Now, since $X$ is of type $\operatorname{K3}^{[2]}$, there is a natural isomorphism
    \[ \operatorname{SH}(X,\QQ)\cong\oH^\ast(X,\QQ), \]
    therefore $\bar{v}(E)=v(E)$ and the claim follows.
\end{proof}
\begin{remark}
    The hypothesis in the proposition above are too strong, for example arguing as in the proof of \cite[Lemma~3.2]{Beckmann} one would not need to assume that $F$ is torsion free. 
\end{remark}

\begin{remark}
    Proposition~\ref{prop:se ann=ann allora atomico} in particular says that for symplectic varieties of type $\operatorname{K3}^{[2]}$ being atomic is equivalent to having an extended Mukai vector. We wish to point out that this is not true in general for other irreducible holomorphic symplectic manifolds.
\end{remark}

\begin{remark}
    As for modular sheaves, being atomic is a property preserved under deformations of pairs $(\mbox{variety},\mbox{sheaf})$ and under the action of the group of derived auto-equivalences (see \cite[Proposition~3.10]{Beckmann}).
\end{remark}

Finally, we will need the following computations performed in \cite[Section~3]{Bottini}. We state them here only for symplectic fourfolds of type $\operatorname{K3}^{[2]}$, but they hold more generally for any ireducible symplectic fourfold (with the due changes - see also \cite[Section~4]{O'Grady:Modular with many moduli}).

\begin{lemma}\label{lemma:Bottini}
    Let $X$ be an irreducible holomorphic symplectic fourfold of type $\operatorname{K3}^{[2]}$. If $E$ is an atomic vector bundle of rank $r$ on $X$ with extended Mukai vector $\tilde{v}(E)=(r,\ell,s)$, then:
    \begin{enumerate}
        \item (\cite[Proposition~3.12]{Bottini}) $E$ is modular and 
        \[ \Delta(E)=\frac{\tilde{q}(\tilde{v}(E))+\frac{5}{2}r^2}{30}\mathsf{c}_2(X); \]
        \item (\cite[Theorem~3.17]{Bottini}) we have 
        \[ \chi(E,E)=3\left(\frac{\tilde{q}(\tilde{v}(E))}{\frac{5}{2}r^2}\right)^2 r^2; \]
        \item (\cite[proof of Corollary~3.9]{Bottini}) we have
        \[ \mathsf{T}(\tilde{v}(E)^{(2)})=\left(r,\ell,\frac{1}{2r}\left(\ell^2-\frac{\tilde{q}(\tilde{v}(E))}{30}\mathsf{c}_2(X)\right),\frac{s}{r}\ell^\vee,\frac{s^2}{2r}\right), \]
        where $\ell^\vee\in\oH^6(X,\QQ)$ is the class such that $\int_X\ell^\vee.\eta=q_X(\ell,\eta)$ for every $\eta\in\oH^2(X,\QQ)$.
    \end{enumerate}
\end{lemma}
\begin{proof}
    Items $(1)$ and $(3)$ differ from \cite{Bottini} only up to the equality $\mathsf{c}_2(X)=30\mathsf{q}_X$ (see Lemma~\ref{lemma:c_2=30q}).
\end{proof}

\begin{example}[Tautological bundles on $F(Y)$ again]
    Let $X=F(Y)$ be the Fano variety of lines of a cubic fourfold. We have already seen that $\mU$ is not modular, so it cannot be atomic either.

    On the other hand it follows that $\mQ$ is indeed atomic: in fact by \cite[Theorem~1.4]{O'Grady:Modular}, the pair $(X,\mQ)$ is deformation of the pair $(S^{[2]},F^{[2]})$, where $S$ is a K3 surface and $F$ is a (suitable) spherical vector bundle on $S$. By \cite[Section~10]{Markman:Modular}, up to the action of the derived auto-equivalence group, the pair $(S^{[2]},F^{[2]})$ is deformation of a pair $(T^{[2]},\mO_{T^{[2]}})$, where $T$ is another K3 surface. Since $\mO_{T^{[2]}}$ is atomic, also $\mQ$ is atomic.

    In this case, by Lemma~\ref{lemma:Bottini} above and the fact that $\mQ$ is modular with $\Delta(\mQ)=\mathsf{c}_2(X)$, it follows that the extended Mukai vector of $\mQ$ is
    \[ \tilde{v}(\mQ)=\left(4,c_1(\mQ),2\right). \]
\end{example}

\subsection{The Koszul resolution}\label{section:Koszul}

Let $X\subset\Gr(2,6)$ be the Fano variety of lines of a cubic fourfold. 

Let us first look at the grassmannian side. As usual, we denote by $\widetilde{\mU}$ and $\widetilde{\mQ}$ the tautological vector bundles on $\Gr(2,6)$. For any partition $\lambda=(\lambda_1,\lambda_2,\lambda_3,\lambda_4)$ we consider the Schur module $\Sigma_\lambda\widetilde{\mQ}$, and we denote by
\[ \widetilde{\mE}_\lambda:=\mE nd(\Sigma_\lambda\widetilde{\mQ})=\Sigma_\lambda\widetilde{\mQ}^\vee\otimes\Sigma_\lambda\widetilde{\mQ} \]
the associated endomorphism bundle.
Notice that by Example~\ref{example:duale} we can write
\begin{equation}\label{eqn: E tilde} \widetilde{\mE}_\lambda=\Sigma_{\lambda'}\widetilde{\mQ}\otimes\Sigma_\lambda\widetilde{\mQ}\otimes\mO_{\Gr(2,6)}(-\lambda_1), 
\end{equation}
where $\lambda'=(\lambda_1-\lambda_,\lambda_1-\lambda_3,\lambda_1-\lambda_2,0)$.

Let us now pass to the Fano variety of lines side, so that we put
\[ \Sigma_\lambda\mQ:=(\Sigma_\lambda\widetilde{\mQ})|_X\quad\mbox{ and }\quad \mE_\lambda:=(\mE_\lambda)|_X. \]

One of the main purposes of this work is to compute the Ext-groups of $\Sigma_\lambda\mQ$ on $X$, or equivalently the cohomology groups of $\mE_\lambda$. This will be achieved by using the Borel--Bott--Weil formula and the Koszul resolution of $X$. 
\medskip

Recalling that $X$ is described as the zero locus of a general section of $\Sym^3\widetilde{\mU}^\vee$, the Koszul complex of $X$ is
\begin{equation}\label{eqn:Koszul}
0\to\W^4\Sym^3\tilde{\mU}\to\W^3\Sym^3\tilde{\mU}\to\W^2\Sym^3\tilde{\mU}\to\Sym^3\tilde{\mU}\to\mO_{\Gr(2,6)}\to\mO_X\to0   .
\end{equation}
Since all the bundles appearing in the sequence above are homogeneous, we can decompose them in irreducible ones, see Section~\ref{section:homogeneous}. It turns out that:
\begin{itemize}
    \item $\Sym^3\tilde{\mU}=\Sigma_{(3,0)}\tilde{\mU}$
    \item $\W^2\Sym^3\tilde{\mU}=\Sigma_{(5,1)}\tilde{\mU}\oplus\Sigma_{(3,3)}\tilde{\mU}$;
    \item $\W^3\Sym^3\tilde{\mU}=\Sigma_{(6,3)}\tilde{\mU}$;
    \item $\W^4\Sym^3\tilde{\mU}=\Sigma_{(6,6)}\tilde{\mU}$.
\end{itemize}

Tensoring (\ref{eqn:Koszul}) with $\widetilde{\mE}_\lambda$ allows us to compute the cohomology of $\mE_\lambda$ via the induced spectral sequence
\begin{equation}\label{eqn:spectral sequence} 
\operatorname{E}^{-p,q}_1=\oH^q\left(\widetilde{\mE}_\lambda\otimes\W^p\Sym^3\widetilde{\mU}\right)\quad\Longrightarrow\quad \oH^{q-p}(\mE_\lambda). 
\end{equation}

Let us spell out the decomposition of (\ref{eqn:Koszul}) $\otimes\,\widetilde{\mE}_\lambda$ according to (\ref{eqn: E tilde}). First of all, by the Littlewood--Richardson formula we can write
\[ \Sigma_{\lambda'}\widetilde{\mQ}\otimes\Sigma_\lambda\widetilde{\mQ}=\bigoplus_{\nu}(\Sigma_\nu\widetilde{\mQ})^{\oplus m^\nu_{\lambda',\lambda}}, \]
from which we get
\[ \widetilde{\mE}_\lambda=\bigoplus_{\nu}\left(\Sigma_{\nu|(\lambda_1,\lambda_1)}\right)^{\oplus m^\nu_{\lambda',\lambda}}. \]
Eventually we have:
\begin{itemize}
    \item $\widetilde{\mE}_\lambda\otimes\Sym^3\tilde{\mU}=\bigoplus_{\nu}\left(\Sigma_{\nu|(\lambda_1+3,\lambda_1)}\right)^{\oplus m^\nu_{\lambda',\lambda}}$;
    \item $\widetilde{\mE}_\lambda\otimes\W^2\Sym^3\tilde{\mU}=\bigoplus_{\nu}\left(\Sigma_{\nu|(\lambda_1+5,\lambda_1+1)}\right)^{\oplus m^\nu_{\lambda',\lambda}}\oplus\bigoplus_{\nu}\left(\Sigma_{\nu|(\lambda_1+3,\lambda_1+3)}\right)^{\oplus m^\nu_{\lambda',\lambda}}$;
    \item $\widetilde{\mE}_\lambda\otimes\W^3\Sym^3\tilde{\mU}=\bigoplus_{\nu}\left(\Sigma_{\nu|(\lambda_1+6,\lambda_1+3)}\right)^{\oplus m^\nu_{\lambda',\lambda}}$;
    \item $\widetilde{\mE}_\lambda\otimes\W^4\Sym^3\tilde{\mU}=\bigoplus_{\nu}\left(\Sigma_{\nu|(\lambda_1+6,\lambda_1+6)}\right)^{\oplus m^\nu_{\lambda',\lambda}}$.
\end{itemize}

The $(p+i)$-th cohomology of each irreducible summand in $\widetilde{\mE}_\lambda\otimes\W^p\Sym^3\widetilde{\mU}$ contributes to the $i$-th cohomology of $\mE_\lambda$.
\medskip

More explicitly, the spectral sequence (\ref{eqn:spectral sequence}) can be read out by splitting the Koszul exact sequence (\ref{eqn:Koszul}) $\otimes\,\widetilde{\mE}_\lambda$ into short exact sequences:
\begin{align}\label{eqn:explicit Koszul}
 0\to\widetilde{\mE}_\lambda\otimes\W^4\Sym^3\widetilde{\mU}\to\widetilde{\mE}_\lambda\otimes\W^3\Sym^3\widetilde{\mU}\to K_1\to 0   \\
   0\to K_1\to \widetilde{\mE}_\lambda\otimes\W^2\Sym^3\widetilde{\mU}\to K_2\to 0  \nonumber   \\
   0\to K_2\to \widetilde{\mE}_\lambda\otimes\Sym^3\widetilde{\mU}\to K_3\to 0  \nonumber   \\
   0\to K_3\to \widetilde{\mE}_\lambda\to \mE_\lambda \to 0.  \nonumber
\end{align}

\section{The symmetric bundle $\Sym^m\mQ$}\label{section:Sym Q}

The aim of this section is to prove Theorems \ref{theorem A}, \ref{theorem B} and \ref{theorem C} in the special case when $\lambda=(m,0,0,0)$. Recall that in this case $\Sigma_\lambda\mQ=\Sym^m\mQ$.

More precisely, we will prove the following result. Let us put
\[ r_m=\binom{m+3}{3}. \]

\begin{thm}\label{thm:Sym Q}
    For any $m\geq0$ the vector bundle $\Sym^m\mQ$ is slope stable, atomic and rigid. 
    
    Moreover we have
    \[ \ch(\Sym^m\mQ)=r_m\left(1,\frac{m}{4}c_1(\mQ),\frac{m(m-1)}{40}c_1(\mQ)^2+\frac{m(m+4)}{20}\ch_2(\mQ),\frac{m^2}{4}\ch_3(\mQ),\frac{m(7m-2)}{20}\ch_4(\mQ)\right) \]
    from which we get:
    \begin{align*} 
    \Delta(\Sym^m\mQ) & =\frac{1}{4}\frac{m(m+4)}{20}r_m^2\mathsf{c}_2(X), \\
    \chi(\Sym^m\mQ,\Sym^m\mQ) & = 3\left(\frac{3m^2+12m-20}{20}\right)^2r_m^2\quad \mbox{ and } \\
    \tilde{v}(\Sym^m\mQ) & = r_m\left(1,\frac{m}{4} c_1(\mQ),-\frac{3m-5}{4}\right).
    \end{align*}
\end{thm}

In Proposition~\ref{prop: ext of sym} we give a deeper description of the Ext-groups of $\Sym^m\mQ$.

\begin{remark}
    Let us put 
    \[\tilde{\delta}(m):=\frac{1}{4}\frac{m(m+4)}{20}r_m^2\in\QQ[m]. \]
    The same proof of Theorem~\ref{thm:Sym Q} shows that if $E$ is a vector bundle of rank $4$ on a smooth and projective variety $Z$ (of dimension at leastt $2$), then
    \[ \Delta(\Sym^m E)=\tilde{\delta}(m)\Delta(E). \] 
    What is remarkable about the polynomial $\tilde{\delta}(m)$ is that it takes integral values when $m$ in an integer. This suggests that $\tilde{\delta}(m)$ may be counting some combinatorial quantity. 

    In fact this is the case: the function $\delta(m)$ appears as a special case of \cite[Lemma~8.14]{Mukai} (more precisely, the case $r=4$ in loc.\ cit.) and counts the number of standard tableaux of size $(m-1,m-1)$ for $\textrm{Mat}(2,6)$. These form a basis for the degree $m-1$ component of the homogeneous coordinate ring of the Grassmannian $\Gr(2,6)$ (see e.g.\ Remark~8.16 in loc. cit.), i.e.
    \[
    \tilde{\delta}(m)=h^0(\Gr(2,6), \mO(m-1)).    
    \]
    
    We point out that $\delta(m)$ has another interesting Lie-theoretical connection, being the dimension of $V^{(k)}$, the distinguished module in the Severi variety, see the case $a=4$ of \cite[Theorem~7.3]{lm06}.

    As we will see later, the function $\delta(m)$ extends to a function $\delta(\lambda)$ such that $\Delta(\Sigma_\lambda\mQ)=\delta(\lambda)\Delta(\mQ)$. The values of $\delta(\lambda)$ are still integral (when valuated on integral entries), but it is not clear what is that they count.
\end{remark}

The rest of this section is dedicated to the proof of Theorem~\ref{thm:Sym Q}.

Since the rank of $\mQ$ is $4$, it is well known that 
\[ \rk(\Sym^m\mQ)=r_m \qquad\mbox{and}\qquad c_1(\Sym^m\mQ)=\frac{m}{4}r_mc_1(\mQ). \]
The remaining Chern classes can be computed combinatorially via the splitting principle. The following formula can be found in \cite[Theorem~4.8]{Dragutin}:
\begin{equation}\label{eqn:ch of Sym}
\ch(\Sym^m\mQ)=\sum_\lambda\frac{1}{\lVert\lambda\rVert}\sum_{\mu\leq\lambda}\binom{m+3}{m-|\mu|}\overleftarrow{\mu}!S(\lambda,\mu)\ch_\lambda(\mQ). 
\end{equation}
Let us explain the formula above. The first summation is taken over all partitions $\lambda$ such that $\lambda_i\neq0$ for every $i$, and $\ch_\lambda(\tilde{\mQ})$ stands for the product $\prod_i\ch_{\lambda_i}(\tilde{\mQ})$. The norm $\lVert\lambda\rVert$ is the product $a_1!a_2!\cdots$, where $a_i=\sharp\{j\mid \lambda_j=i\}$. The second summation is taken over all the sub-partitions $\mu\leq\lambda$ such that $\ell(\mu)=\ell(\lambda)$ (in particular also $\mu_i\neq0$ for every $i$). The symbol $\overleftarrow{\mu}$ stands for the partition $(\mu_1-1,\mu_2-1,\cdots)$ and $\overleftarrow{\mu}!=\prod_i(\mu_i-1)!$. 
Finally, $S(\lambda,\mu)=\prod_iS(\lambda_i,\mu_i)$, where $S(x,y)$ is the Stirling number of the second kind.

It follows from the formula that $\ch_2(\Sym^m\mQ)$ has only contributes from $c_1(\mQ)^2$ (corresponding to $\lambda=(1,1)$) and $\ch_2(\mQ)$ (corresponding to $\lambda=(2)$). An explicit computation gives
\begin{equation*}
    \ch_2(\Sym^m\mQ)=\frac{m^2-m}{40}r_m c_1(Q)^2+\frac{m^2+4m}{20}r_m\ch_2(\mQ).
\end{equation*}
Using the definition, we get
\[ \Delta(\Sym^m\mQ)=c_1(\mQ)^2-2\rk(\Sym^m\mQ)\ch_2(\mQ)=\frac{1}{4}\frac{m^2+4m}{20}r_m^2\Delta(\mQ). \]
Since $\Delta(\mQ)=c_2(X)$ (item $(4)$ of Lemma~\ref{lemma:c_2(X)^2 and ch_4(Q)}), it follows that $\Sym^m\mQ$ is modular. 

Using again the formula (\ref{eqn:ch of Sym}) and the identities in Lemma~\ref{lemma:identities ch_3(Q) e ch_4(Q)}, we further compute
\begin{equation*}
    \ch_3(\Sym^m\mQ)=\frac{m^2}{4}r_m\ch_3(\mQ)
\end{equation*}
and
\begin{equation*}
    \ch_4(\Sym^m\mQ)=\frac{7m^2-2m}{20}r_m\ch_4(\mQ)
\end{equation*}


Now, to simplify the notation, for a vector bundle $F$ let us put
\[ \Xi(F):=\ch_2(F)^2-2c_1(F)\ch_3(F)+2\rk(F)\ch_4(F). \]
This is justified by the fact that in this way we get
\[ \ch(\mE nd(F))=\left(\rk(F)^2,0,-\Delta(F),0,\Xi(F)\right). \]

Using the computations of $\ch_j(\Sym^m\mQ)$ above,
Lemma~\ref{lemma:identities ch_3(Q) e ch_4(Q)} and the fact that $\ch_4(\mQ)=\frac{1}{1104}\mathsf{c}_2(X)^2$ (item $(3.e)$ of Lemma~\ref{lemma:identities ch_3(Q) e ch_4(Q)}), we get
\[ \Xi(\Sym^m\mQ)=\frac{m(m+4)}{110400}(9m^2+36m-5)r_m^2c_2(X)^2. \]

Since $\int_X\mathsf{c}_2(X)^2=828$ (item (1) of Lemma~\ref{lemma:c_2(X)^2 and ch_4(Q)}), we have that
\[ \int_X\Xi(\Sym^m\mQ)=\frac{3m(m+4)}{20}\frac{9m^2+36m-5}{20}r_m^2
\]

Now, by item (2) of Lemma~\ref{lemma:c_2(X)^2 and ch_4(Q)}
\[ \operatorname{td}_X=\left(1,0,\frac{1}{12}\mathsf{c}_2(X),0,3\mathsf{p}\right) \]
so that by Hirzebruch--Riemann--Roch formula we get
\begin{align} 
\chi(\Sym^m(Q),\Sym^m(Q)) & =\int_X\ch(\mE nd(\Sym^m\mQ))\operatorname{td}_X \nonumber \\
 & =\int_X\left(3r_m^2\mathsf{p}-\frac{1}{12}\Delta(\Sym^m\mQ).c_2(X)+\Xi(\Sym^m\mQ)\right) \nonumber \\
 & = 3r_m^2-\frac{69}{4}\frac{m(m+4)}{20}r_m^2+\frac{3}{20}\frac{m(m+4)}{20}(9m^2+36m-5)r_m^2 \nonumber \\
 & =\left( 3+9\frac{m(m+4)}{20}\frac{3m^2+12m-40}{20} \right)\binom{m+3}{3}^2 \nonumber \\
 & =\frac{3}{400}(9m^4+72m^3+24m^2-480m+400)r_m^2 \nonumber \\
& =3\left(\frac{3m^2+12m-20}{20} \right)^2 r_m^2. \label{eqn:chi of sym}
\end{align}

Let us now compute the dimensions of the Ext-groups of $\Sym^m\mQ$.

\begin{proposition}\label{prop: ext of sym}
    We have 
    \[ \operatorname{dim}\Ext^p(\Sym^m\mQ,\Sym^m\mQ)=\left\{
    \begin{array}{cc}
       1  &  \mbox{if } p=0,4 \\
       0  &  \mbox{if } p=1,3 \\
       3\left(\frac{3m^2+12m-20}{20} \right)^2 r_m^2-2 & \mbox{if } p=2.
    \end{array}\right. \]
    More precisely,
    \[ \Ext^2(\Sym^m\mQ,\Sym^m\mQ)=\left\{\begin{array}{ll}
       \CC  & m=0 \\
       \CC  &  m=1 \\
       \CC\oplus \Sigma_{(2,2,1,1,0,0)} & n=2 \\
       \CC\oplus \Sigma_{(2,2,1,1,0,0)}\oplus\bigoplus_{n=3}^m\left(\begin{array}{c}
         \Sigma_{(2n-2,n-2,n-2,n-3,0)} \\
         \oplus \\
         \Sigma_{(2n-2,n,n-1,n-1,n-2,0)} \\
         \oplus \\
         \Sigma_{(2n-2,n,n-1,n-1,n-2)}
         \end{array}\right) & n\geq3\, .
    \end{array}\right.
     \]
\end{proposition}

\begin{proof}
    We use the Koszul resolution (\ref{eqn:Koszul}) and the spectral sequence (\ref{eqn:spectral sequence}) applied to the vector bundle $\mE nd(\Sym^m\widetilde{\mQ})$.

    First of all, let us observe that by Example~\ref{example:duale} we have
    \[ \mE nd(\Sym^m\widetilde{\mQ})=\left(\Sigma_{(m,m,m,0)}\widetilde{\mQ}\otimes\Sym^m\widetilde{\mQ}\right)\otimes\mO_X(-m), \]
    and that, by Pieri's formula, we have
    \[ \Sigma_{(m,m,m,0)}\widetilde{\mQ}\otimes\Sym^m\widetilde{\mQ}=\bigoplus_{i=0}^m\Sigma_{(2m-i,m,m,i)}\widetilde{\mQ}. \]

    Therefore we get
    \begin{equation}\label{eqn:LR for Sym} 
    \mE nd(\Sym^m\widetilde{\mQ})=\bigoplus_{i=0}^m\Sigma_{(2m-i,m,m,i)|(m,m)}=\bigoplus_{n=0}^m\Sigma_{(2n,n,n,0|n,n)}, 
    \end{equation}
    where in the last equality we used (\ref{eqn: + n}) and the substitution $n=m-i$.

    In Lemma~\ref{lemma: (2n n n 0 | n n)} below we compute explicitly the cohomology groups of the bundles $\Sigma_{(2n,n,n,0|n,n)}$, from which the claim follows at once.
\end{proof}

\begin{remark}
    The claim about the dimensions of the Ext-groups of $\Sym^m\mQ$ does not need the whole analysis done in the proof of Lemma~\ref{lemma: (2n n n 0 | n n)}. In fact it would have been enough to remark that for odd $p$ the bundles $\Sigma_{(2n,n,n,0|n,n)}\otimes\W^p\Sym^3\widetilde{\mU}$ are acyclic and that for even $p$ the same bundles only contribute to $\Ext^2(\Sym^m\mQ,\Sym^m\mQ)$ (apart from the case $n=0$, which corresponds to the trivial line bundle). Then the claim would follow from the computation (\ref{eqn:chi of sym}).
\end{remark}

\begin{lemma}\label{lemma: (2n n n 0 | n n)}
    For any $n\geq0$ we have
    \[ \oH^0(\Sigma_{(2n,n,n,0|n,n)}=\left\{\begin{array}{ll}
      \CC  & \mbox{if } n=0 \\
       0  &  \mbox{if } n\neq 0
    \end{array}\right. \]
    \[ \oH^1(\Sigma_{(2n,n,n,0|n,n)}=0 \]
    \[ \oH^2(\Sigma_{(2n,n,n,0|n,n)})=\left\{\begin{array}{ll}
       \CC  &  \mbox{if } n=0 \\
        0   &  \mbox{if } n=1 \\
        \Sigma_{(2,2,1,1,0,0)} & \mbox{if } n=2 \\
        \left(\begin{array}{c}
         \Sigma_{(2n-2,n-2,n-2,n-3,n-3,0)} \\
         \oplus \\
         \Sigma_{(2n-2,n,n-1,n-1,n-2,0)} \\
         \oplus \\
         \Sigma_{(2n-2,n+1,n+1,n,n,0)}
         \end{array}\right)
         & \mbox{if } n\geq 3
    \end{array}\right.     
     \]
\end{lemma}
\begin{proof}
    First of all, we need to compute the cohomology groups of the homogeneous vector bundles on $\Gr(2,6)$ of the form 
    \[ \Sigma_{(2n,n,n,0|n,n)}\otimes\W^p\Sym^3\widetilde{\mU}. \]
    We divide this computation in several steps labelled by the exponent $p$ of the wedge product. We remind to Section~\ref{section:Koszul} for the decomposition of $\W^p\Sym^3\widetilde{\mU}$ into irreducible factors, which we will tacitly use.
    \begin{itemize}
        \item ($p=0$). In this case we need to compute the cohomology of $\Sigma_{(2n,n,n,0|n,n)}$. Let us put $v=(2n,n,n,0,n,n)$. According to the Borel--Weil--Bott algorithm (Section~\ref{section:BBW}), we first need to look at 
        \[ v+\delta=(2n+5,n+4,n+3,2,n+1,n). \]
        It is then clear that for $n=0$ we have only cohomology in degree $0$, given by the trivial partition, while for $n=1,2$ the bundle $\Sigma_{(2n,n,n,0|n,n)}$ is acyclic. Finally, for $n\geq3$ we always have $\ell(v+\delta)=2$, therefore in this case we have
        \[ \oH^2(\Sigma_{(2n,n,n,0|n,n)})=\Sigma_{(2n-2,n-2,n-2,n-3,n-3,0)}. \] 

        \item (p=1). In this case we need to compute the cohomology of $\Sigma_{(2n,n,n,0|n+3,n)}$, so that we put $v=(2n,n,n,0,n+3,n)$ and hence we need to look at
        \[ v+\delta=(2n+5,n+4,n+3,2,n+4,n). \]
        Since there are repeated entries, the bundle is acyclic for every $n\geq0$.

        \item (p=2). Here the bundle $\W^2\Sym^3\widetilde{\mU}$ has two irreducible factors, hence we divide the analysis in two sub-cases.
        \begin{enumerate}
            \item Let us first consider the vector $v=(2n,n,n,0,n+3,n+3)$.
            In this case
            \[ v+\delta=(2n+5,n+4,n+3,2,n+4,n+3) \]
            has repeated entries, therefore $\Sigma_{(2n,n,n,0|n,n)}\otimes\Sigma_{(3,3)}\widetilde{\mU}$ is acyclic.
            \item Let us then consider the vector $v=(2n,n,n,0,n+5,n+1)$. Then
            \[ v+\delta=(2n+5,n+4,n+3,2,n+6,n+1). \]
            For $n=0$ we have $\ell(v+\delta)=4$, so that there is only one non-trivial cohomology group in degree $4$, given by the trivial partition. For $n=1$ we have repeated entries, so that the bundle is acyclic. Finally, for $n\geq2$ we have $\ell(v+\delta)=4$, so that we have only one non trivial cohomology group in degree $4$, more precisely
            \[ \oH^4(\Sigma_{(2n,n,n,0|n,n)}\otimes\Sigma_{(5,1)}\widetilde{\mU})=\Sigma_{(2n-2,n,n-1,n-1,n-2,0)}. \]
        \end{enumerate}
        
        \item ($p=3$). In this case we have to consider the vector $v=(2n,n,n,0,n+6,n+3)$. Since
        \[ v+\delta=(2n+5,n+4,n+3,2,n+7,n+3), \]
        has repeated entries, the bundle $\Sigma_{(2n,n,n,0|n+6,n+3)}$ is acyclic for every $n\geq0$.

        \item ($p=4$). Let us finally consider the vector $v=(2n,n,n,0,n+6,n+6)$, and hence the vector
        \[ v+\delta=(2n+5,n+4,n+3,2,n+7,n+6). \]
        For $n=0$ we have $\ell(v+\delta)=8$ and again we have only one non-trivial cohomology group in degree $8$ given by the trivial partition. If $n=1,2$, then the bundle $\Sigma_{(2n,n,n,0|n+6,n+6)}$ is acyclic. For $n\geq3$ we have $\ell(v+\delta)=6$ so that 
        \[ \oH^6(\Sigma_{(2n,n,n,0|n+6,n+6)})=\Sigma_{(2n-2,n+1,n+1,n,n,0)}. \]
    \end{itemize}
    To get the claim is it enough to plug these results into the long exact sequences associated to the short exact sequences (\ref{eqn:explicit Koszul}).
\end{proof}

As a corollary we get the stability of $\Sym^m\mQ$.

\begin{corollary}\label{corollary:recap on sym}
    The vector bundles $\Sym^m\mQ$ are slope stable.
\end{corollary}
\begin{proof}
    First of all, by \cite[Remark~8.6]{O'Grady:Modular} the vector bundle $\mQ$ on $X$ is slope stable, hence by \cite[Theorem~3.2.11]{HL} the vector bundles $\Sym^m\mQ$ are polystable. Since by Proposition~\ref{prop: ext of sym} they are simple, the claim follows.
\end{proof}

\begin{remark}
    We have seen so far that $\Sym^m\mQ$ is a slope stable, modular and rigid vector bundle on a fourfold of type $\operatorname{K3}^{[2]}$. We point out that, unless $m=1$, these bundles do not satisfy the hypothesis of \cite[Theorem~1.4]{O'Grady:Modular}, for example because the ranks do not match, and are therefore new in literature. 
    
    Since they are rigid, in analogy with the case of stable sheaves on K3 surfaces it is natural to expect that the uniqueness statement of \cite[Theorem~1.4]{O'Grady:Modular} also applies to them.
\end{remark}

Finally, we prove that $\Sym^m\mQ$ are atomic.

\begin{proposition}\label{prop:Sym atomic}
    The bundles $\Sym^m\mQ$ are atomic with extended Mukai vector 
    \[ \tilde{v}(\Sym^m\mQ)=\left(r_m,\frac{m}{4}r_m c_1(\mQ),-\frac{3m-5}{4}r_m\right). \]
\end{proposition}
\begin{proof}
     First of all, let us remark that the equalities 
    \[ \Delta(\Sym^m\mQ)=\frac{1}{4}\frac{m^2+4m}{20}r_m^2\Delta(\mQ)\quad\mbox{ and }\quad \chi(\Sym^m\mQ,\Sym^m\mQ)=3\left(\frac{3m^2+12m-20}{20} \right)^2 r_m^2 \]
    from Section~\ref{section:Sym Q}, together with items $(1)$ and $(2)$ of Lemma~\ref{lemma:Bottini} forces the quantity $s(\Sym^m\mQ)$ in the normalisation (\ref{eqn:v tilde normalizzato}) to be $s(\Sym^m\mQ)=-\frac{3m-5}{4} r_m$. Therefore an extended Mukai vector for $\Sym^m\mQ$, if it exists, must be of the form    
    \[ \tilde{v}=\left(r_m,\frac{m}{4}r_m c_1(\mQ),-\frac{3m-5}{4}r_m\right)\in\widetilde{\oH}(X,\QQ). \]
    On the other hand, we claim that
    \begin{equation}\label{eqn:T(v tilde)=v}
    \operatorname{T}(\tilde{v}^{(2)})=v(\Sym^m\mQ), 
    \end{equation}
    from which it would follow that $\tilde{v}$ is indeed the extended Mukai vector of $\Sym^m\mQ$, and therefore, by Proposition~\ref{prop:se ann=ann allora atomico}, $\Sym^m\mQ$ would be atomic.

    Let us prove the claim by using item $(3)$ of Lemma~\ref{lemma:Bottini}:
    \begin{align*} 
    \operatorname{T}(\tilde{v}^{(2)})= & \left(r_m,\frac{m}{4}r_mc_1(\mQ), \right. \\
    & \frac{3m^2-3m+5}{120}r_mc_1(\mQ)^2+\frac{3m^2+12m-20}{60}r_m\ch_2(\mQ), \\
    & \left. -\frac{m(3m-5)}{288}r_mc_1(\mQ)^3,\frac{(3m-5)^2}{32}r_m \right) .
    \end{align*}
    Now, by item $(3)$ of Lemma~\ref{lemma:c_2(X)^2 and ch_4(Q)}, we have that 
    \[ \sqrt{\operatorname{td}_X}=\left( 1,0,\frac{1}{24}\operatorname{c}_2(X),0,\frac{25}{32}\mathsf{p}\right). \]
    The Mukai vector $v(\Sym^m\mQ)=\ch(\Sym^m\mQ)\sqrt{\operatorname{td}_X}$ can then be explicitly computed as
    \begin{align*}
        v(\Sym^m\mQ)= & \left( r_m, \frac{m}{4}r_mc_1(\mQ), \right. \\
         & \frac{1}{24} r_m \mathsf{c}_2(X)+\frac{m(m-1)}{40}r_mc_1(\mQ)^2+\frac{m(m+4)}{20}r_m\ch_2(\mQ), \\
         & \frac{m}{96} r_m c_1(\mQ).\mathsf{c}_2(X)+\frac{m^2}{4}r_m\ch_3(\mQ), \\
         & \left. \frac{25}{32}r_m+\frac{m(m-1)}{960}r_mc_1(\mQ)^2.\mathsf{c}_2(X)+\frac{m(m+4)}{480}r_m\ch_2(\mQ).\mathsf{c}_2(X)+\frac{m(7m-2)}{20}r_m\ch_4(\mQ)\right)
    \end{align*}
    Let us now recall that 
    \[ \mathsf{c}_2(X)=\Delta(\mQ)=c_1(\mQ)^2-8\ch_2(\mQ). \]
    Using the identities in Lemma~\ref{lemma:identities ch_3(Q) e ch_4(Q)}, it follows from this that
    \[ c_1(\mQ).\mathsf{c}_2(X)=\frac{5}{3}c_1(\mQ)^3,\quad \int_X c_1(\mQ)^2.\mathsf{c}_2(X)=180 \quad\mbox{and}\quad\int_X \ch_2(\mQ).\mathsf{c}_2(X)=-81. \]

    Substituting these equalities in the expression of $v(\Sym^m\mQ)$ above gives the  equality (\ref{eqn:T(v tilde)=v}), thus completing the proof.
\end{proof}

\begin{remark}
    Notice that 
    \[ \tilde{v}(\Sym^m\mQ)^2=\frac{3m^2+12m-20}{8} r_m^2. \]
    This number is negative for $m=1$, but it is positive for $m>1$. Since $\Sym^m\mQ$ is rigid for every $m\geq1$, this is in contrast with the case of K3 surfaces, where rigidity is equivalent to having a Mukai vector with negative square.
\end{remark}

\section{Two worked-out examples}\label{section:two worked out example}
Before diving into the proof of the main theorems, we wish to work out by hand two examples, in which some technical details are exploited: this will also serve as a leading guide for the rest of the paper.

\subsection{The vector bundle $\Sigma_{(2,1,0,0)}\mQ$}\label{section:Sigma 2 1}
Let us first consider the vector bundle $\Sigma_{(2,1,0,0)}\tilde{\mQ}$ on $\Gr(2,V_6)$. By Pieri's formula (see decomposition~(\ref{eqn:Pieri})) we have that
\[ \Sym^2\tilde{\mQ}\otimes\tilde{\mQ}=\Sym^3\tilde{\mQ}\oplus\Sigma_{(2,1,0,0)}\tilde{\mQ} \]
and hence, by restriction,
\[ \Sym^2\mQ\otimes\mQ=\Sym^3\mQ\oplus\Sigma_{(2,1,0,0)}\mQ. \]

In particular we can interpret $\Sigma_{(2,1,0,0)}\mQ$ as the kernel of the natural map of homogeneous bundles $\Sym^2\mQ\otimes\mQ\to\Sym^3\mQ$.

We start by computing the Chern character of $\Sigma_{(2,1,0,0)}\mQ$ using the decomposition above:
\[ \ch(\Sigma_{(2,1,0,0)}\mQ)=\ch(\Sym^2\mQ\otimes\mQ)-\ch(\Sym^3\mQ). \]
By the computations in Section~\ref{section:Sym Q} we know that 
\[ \ch(\Sym^2\mQ)=\left( 10, 5c_1(\mQ), \frac{1}{2}c_1(\mQ)^2+6\ch_2(\mQ), 10\ch_3(\mQ), 12\ch_4(\mQ) \right) \]
and
\[ \ch(\Sym^3\mQ)=\left( 20, 15c_1(\mQ), 3c_1(\mQ)^2+21\ch_2(\mQ), 45\ch_3(\mQ), 57\ch_4(\mQ) \right), \]
from which we deduce that (we use Lemma~\ref{lemma:identities ch_3(Q) e ch_4(Q)} to simplify the expressions)
\[ \ch(\Sym^2\mQ\otimes\mQ)=\left( 40, 30c_1(\mQ), 7c_1(\mQ)^2+34\ch_2(\mQ), 60\ch_3(\mQ), 34\ch_4(\mQ) \right). \]
We eventually get
\[ \ch(\Sigma_{(2,1,0,0)}\mQ)=\left( 20, 15c_1(\mQ), 4c_1(\mQ)^2+13\ch_2(\mQ), 15\ch_3(\mQ), -23\ch_4(\mQ) \right). \]
From this, it is easy to compute the discriminant and the Euler characteristic of $\Sigma_{(2,1,0,0)}\mQ$, which we collect in the following lemma.
\begin{lemma}
    The following equality hold:
    \[ \Delta(\Sigma_{(2,1,0,0)}\mQ)=65\mathsf{c}_2(X)\qquad\mbox{ and }\qquad \chi(\Sigma_{(2,1,0,0)}\mQ,\Sigma_{(2,1,0,0)}\mQ)=363. \]
\end{lemma}
\begin{proof}
    The first equality follows directly from the definition of the discriminant. 
    
    Let us prove the second equality. First of all, as in Section~\ref{section:Sym Q} we define the quantity
    \[ \Xi(\Sigma_{(2,1,0,0)}\mQ)=\ch_2(\Sigma_{(2,1,0,0)}\mQ)-2c_1(\Sigma_{(2,1,0,0)}\mQ).\ch_3(\Sigma_{(2,1,0,0)}\mQ)+2\rk(\Sigma_{(2,1,0,0)}\mQ)\ch_4(\Sigma_{(2,1,0,0)}\mQ) \]
    so we have
    \[ \ch(\mE nd(\Sigma_{(2,1,0,0)}\mQ))=\left(400, 0, -\Delta(\Sigma_{(2,1,0,0)}\mQ), 0, \Xi(\Sigma_{(2,1,0,0)}\mQ) \right). \]
    By an explicit computation (and the help of Lemma~\ref{lemma:identities ch_3(Q) e ch_4(Q)} to simplify the expressions) we get
    \[ \Xi(\Sigma_{(2,1,0,0)}\mQ)=4864\ch_4(\mQ). \]
    Now, using again the expression
    \[ \operatorname{td}_X=\left(1,0,\frac{1}{12}\mathsf{c}_2(X),0,3\mathsf{p}\right) \]
    and the Hirzebruch--Riemann--Roch formula, we eventually get
    \[ \chi(\Sigma_{(2,1,0,0)}\mQ,\Sigma_{(2,1,0,0)}\mQ) =\int_X\left(1200\mathsf{p}-\frac{65}{12}\mathsf{c}_2(X)^2+4864\ch_4(\mQ)\right)=363, \]
    where we used again items (1) and (4) of Lemma~\ref{lemma:c_2(X)^2 and ch_4(Q)}.
\end{proof}


In the rest of this section we want to compute and understand the Ext-algebra of $\Sigma_{(2,1,0,0)}\mQ$, namely, we want to prove the following.
\begin{proposition}\label{prop:ext210}
    Consider the vector bundle $\Sigma_{(2,1,0,0)}\mQ$ on $X$. Then
\[ \dim\Ext^p(\Sigma_{(2,1,0,0)}\mQ,\Sigma_{(2,1,0,0)}\mQ) \cong
    \begin{cases}
         1 \ p=0,4 \\
         20 \ p=1,3 \\
        401 \ \ p=2 \\
        0 \ \ \textrm{otherwise}.
    \end{cases}\]

     More precisely, if $\Ext^\ast_0(\Sigma_{(2,1,0,0)}\mQ,\Sigma_{(2,1,0,0)}\mQ)$ denotes the traceless part of the Ext-algebra, then
    \[
 \Ext^\ast_0(\Sigma_{(2,1,0,0)}\mQ,\Sigma_{(2,1,0,0)}\mQ) \cong \oH^\ast(\mE nd_0 \mQ) \oplus (\Sigma_{(2,2,1,1,0,0)} V_6^{\vee} \oplus \CC  \oplus \Sigma_{(2,1,1,1,1,0)}V_6^{\vee} \oplus \Sigma_{(2,2,2,0,0,0)} V_6^{\vee})[2]
    \]
    In particular, we have \[ \Ext^2_0(\Sigma_{(2,1,0,0)}\mQ,\Sigma_{(2,1,0,0)}\mQ)\cong\End(\W^2 \W^3 V_6^{\vee}), \] 
    which contains $\W^2 \Ext^1(\mQ_{(2,1,0,0)}, \mQ_{(2,1,0,0)})$
    as an $\operatorname{SL}(6)$-irreducible component.   
\end{proposition}

\begin{remark}
    We remark that a priori the cohomology groups analyzed in the Proposition above do not admit a natural $\operatorname{SL}(6)$-action. This is a peculiarity of this example and other in the rest of the paper, where the only cohomological contributions come from $\Gr(2,6)$. This is a non-trivial computation, which in this case is detailed in the proof below.
\end{remark}
\begin{proof}

First of all, we need to compute the decomposition in irreducibles of $\Sigma_{(2,1,0,0)} \tilde{\mQ}\otimes \Sigma_{(2,1,0,0)}\tilde{\mQ}^{\vee}$. We use the Littlewood--Richardson formula, and obtain that the latter is equal to
\begin{align}\label{eqn:dec21}
    \mE nd(\Sigma_{(2,1,0,0)}\mQ)= & \Sigma_{(4,3,1,0\mid 2,2)} \oplus \Sigma_{(4,2,2,0\mid 2,2)} \oplus \Sigma_{(4,2,1,1\mid 2,2)} \oplus \Sigma_{(3,3,2,0\mid 2,2)} \oplus \\
    & \Sigma_{(3,3,1,1\mid 2,2)} \oplus \Sigma_{(3,2,2,1\mid 2,2)}^{\oplus2} \oplus \Sigma_{(2,2,2,2\mid 2,2)}. \nonumber
\end{align}

We analyze each factor of the above decomposition separately. In fact, for every partition $\mu$ appearing in the decomposition (\ref{eqn:dec21}), in order to compute the cohomology of $\Sigma_\mu \mQ$, we need to tensor $\Sigma_\mu \tilde{\mQ}$ with the Koszul complex (\ref{eqn:Koszul}), and understanding the associated spectral sequence.

Therefore, the first step is a lengthy but straightforward application of the Borel--Weil--Bott rule. We refer to Section~\ref{section:BBW} for a quick recap of the Borel--Weil--Bott formula, and to the proof of Lemma~\ref{lemma: (2n n n 0 | n n)} for a keen analysis of a similar computation. In the following we leave most of the easy check to the reader. 

To give at least one example, the first $\mu$ to consider is $\mu=(4,3,1,0,2,2)$. The associated vector bundle is acyclic, since $\mu+\delta=(9,7,4,2,3,2)$ has repeated entries. The same happens for $\mu+ (0^4,3,0), \mu+ (0^4,6,3),\mu+ (0^4,6,6)$, corresponding to $\Sigma_\mu \tilde{\mQ} \otimes \W^i \Sym^3 \tilde \mU$ in the Koszul complex, for $i=0,1,3,4$.
The situation is different in the case $\Sigma_\mu \tilde{\mQ} \otimes \W^2 \Sym^3 \tilde \mU$. In fact, both the vector bundles associated to $\mu_1:=\mu + (0^4, 5,1)$ and $\mu_2:=\mu +(0^4, 3,3)$ have cohomology. In the first case, the partition has $\ell(\mu_1 +\delta)=4$, and we have 
$\tilde{\mu_1}=(2,2,2, 0^3)$. In the second case, also $\ell(\mu_1 +\delta)=4$, and we have 
$\tilde{\mu_1}=(2,1,1,1,1,0)$, corresponding to $\mathfrak{sl}(6)$. In particular,
\[
\oH^4(\Gr(2,6), \Sigma_{(4,3,1,0,7,3)}\tilde \mQ)= \Sigma_{(2,2,2,0,0,0)} V_6^{\vee}; \ \oH^4(\Gr(2,6), \Sigma_{(4,3,1,0,5,5)}\tilde \mQ)= \Sigma_{(2,1,1,1,1,0)} V_6^{\vee},
\]
where we also notice the decomposition
\[
\Sym^2 \W^3 V_6^{\vee}=  \Sigma_{(2,2,2,0,0,0)} V_6^{\vee} \oplus \Sigma_{(2,1,1,1,1,0)} V_6^{\vee}
\]

We proceed to analyze the nonzero contribution to the cohomology of each of the factors appearing in the decomposition~(\ref{eqn:dec21}). In Table~\ref{dimcoho (2,1,0,0)} below, $\mu$ denotes a partition appearing in the decomposition, $K^i$ the graded piece of the Koszul complex where the cohomology appear, $\oH^j$ in which cohomology group we find some contribution, given by the representation $\Sigma_{\tilde\mu}$, of the given dimension.
\FloatBarrier
\begin{table}[h!bt]
\caption{Ext-table for $\Sigma_{(2,1,0,0)}\mQ$}
\label{dimcoho (2,1,0,0)}
\begin{tabular}{ccccc} \toprule
$\mu$ & $K^i$& $\oH^j$& $\Sigma_{\mu}$ & dim  \\
					\midrule
				$(4,3,1,0\mid 2,2)$ & 2 & 4 & $(2,1,1,1,1)$ & 35\\
				$(4,3,1,0\mid 2,2)$ & 2 & 4 & (2,2,2)& 175\\
    $(4,2,2,0\mid 2,2)$ & 2 & 4 & (2,2,1,1)& 189\\
    $(3,3,1,1\mid 2,2)$ & 1 & 2 & (1,1,1)& 20\\
        $(3,3,1,1\mid 2,2)$ & 2 & 4 & (0)& 1\\
    $(3,3,1,1\mid 2,2)$ & 3 & 6 & (1,1,1)& 20\\
    $(0,0,0,0\mid 0,0)$ & 0 & 0 & (0)& 1\\
            $(0,0,0,0\mid 0,0)$ & 2 & 4 & (0)& 1\\
        $(0,0,0,0\mid 0,0)$ & 4 & 8 & (0)& 1\\

					\bottomrule
				\end{tabular}
\end{table}

Notice in particular how, in analogy with the symmetric case,
\[
\W^2 \W^3 V_6^{\vee} = \Sigma_{(2,2,1,1,0,0)} V_6^{\vee} \oplus \CC.
\]
Using these data, one can compute the self-ext groups of $\Sigma_{2,1}\mQ$. In particular, we have
\[
\Ext^i(\Sigma_{(2,1,0,0)}\mQ, \Sigma_{(2,1,0,0)}\mQ)= \W^3 V_6^{\vee}, \ i=1,3
\] 
and
\[
\Ext^2_0(\Sigma_{(2,1,0,0)}\mQ, \Sigma_{(2,1,0,0)}\mQ)= \oH^2(X, \mE nd_0(\Sigma_{(2,1,0,0)}\mQ))=\Sym^2 \W^3 V_6^{\vee}  \oplus \W^2 \W^3 V_6^{\vee}= \textrm{End}(\W^3 V_6^{\vee}).
\]
\end{proof}

\subsection{The vector bundle $\Sigma_{(3,2,1,0)}\mQ$}\label{section:Sigma 3 2 1}

 First of all, let us observe that the vector bundle $\Sigma_{(3,2,1,0)}\widetilde{\mQ}$ on the Grassmannian can be interpreted as the kernel of the natural morphism of vector bundles
\[ \widetilde{\mQ}\otimes\W^2\widetilde{\mQ}\otimes\W^3\widetilde{\mQ}\longrightarrow \mE nd(\widetilde{\mQ})(1)\oplus\mE nd(\widetilde{\mQ}^\vee)(2). \]

We start with the computation of the Chern character, the discriminant and the Euler characteristic of the restricted vector bundle $\Sigma_{(3,2,1,0)}\mQ$ on $X$.
\begin{proposition}
    We have
    \[ \ch(\Sigma_{(3,2,1,0)}\mQ)=(64, 96c_1(\mQ), 64c_1(\mQ)^2+64\ch_2(\mQ), -384\ch_3(\mQ), 208\ch_4(\mQ)). \]
    Moreover,
    \[ \Delta(\Sigma_{(3,2,1,0)}\mQ)=1024\mathsf{c}_2(X)\quad\mbox{and}\quad \chi(\Sigma_{(3,2,1,0)}\mQ,\Sigma_{(3,2,1,0)}\mQ)=35328. \]
\end{proposition}
\begin{proof}
    We refer to Section~\ref{section:strategy} for the strategy of computation of the Chern character. The reader is welcome to perform the computation by hand, otherwise they can use the formulae in Theorem~\ref{thm:Chern character}. The discriminant is easy to check. The Euler characteristic can be either checked by hand (using the relations in Lemma~\ref{lemma:identities ch_3(Q) e ch_4(Q)}) or extracted from (the proof of) Proposition~\ref{prop:eulero}.
\end{proof}

Let us now compute the $\Ext$-algebra of $\Sigma_{(3,2,1,0)}\mQ$.

\begin{proposition}
    We have 
    \[ \dim\Ext^p(\Sigma_{(2,1,0,0)}\mQ,\Sigma_{(2,1,0,0)}\mQ) \cong
    \begin{cases}
         1 \ p=0,4 \\
         40 \ p=1,3 \\
        35406 \ \ p=2 \\
        0 \ \ \textrm{otherwise}.
    \end{cases}\]
    More precisely, $\Ext^1(\Sigma_{(3,2,1,0)}\mQ,\Sigma_{(3,2,1,0)}\mQ)=\left(\W^3 V_6^\vee\right)^{\oplus2}$.
\end{proposition}

\begin{proof}
First of all, we have to decompose the endomorphism bundle of $\Sigma_{3,2,1}\mQ$ in irreducible components. Such a decomposition is
\begin{align}\label{eqn:decomp 3210}
  \mE nd(\Sigma_{(3,2,1,0)}\mQ)= & \Sigma_{(6,4,2,0\mid 3,3)} \oplus \Sigma_{(6,4,1,1\mid 3,3)} \oplus \Sigma_{(6,3,3,0\mid 3,3)} \oplus \Sigma_{6,3,2,1\mid 3,3)}^{\oplus 2} \oplus \\
  &  \Sigma_{(6,2,2,2\mid 3,3)} \oplus \Sigma_{(5,5,2,0\mid 3,3))} \oplus \Sigma_{(5,5,1,1\mid 3,3)} \oplus \Sigma_{(5,4,3,0\mid 3,3)}^{\oplus 2} \oplus \nonumber \\
  &  \Sigma_{(5,4,2,1\mid 3,3)}^{\oplus 4} \oplus \Sigma_{(5,3,3,1\mid 3,3)}^{\oplus 3} \oplus \Sigma_{(5,3,2,2\mid 3,3)} \oplus \Sigma_{(4,4,4,0\mid 3,3)} \oplus \nonumber \\
  & \Sigma_{(4,4,3,1\mid 3,3)}^{\oplus 3} \oplus \Sigma_{(4,4,2,2\mid 3,3)}^{\oplus 2} \oplus \Sigma_{(4,3,3,2\mid 3,3)}^{\oplus 3} \oplus \Sigma_{3,3,3,3\mid 3,3)} . \nonumber
\end{align}

This decomposition is considerably more complicated than the previous example. Many of the irreducible factors above have cohomology when tensored with specific pieces of the Koszul complex.
We proceed as in the previous case, and we record their values (together with the multplicity) in Table~\ref{dimcoho (3,2,1,0)} below.
\end{proof}
\FloatBarrier
\begin{table}[!htbp]
\caption{Ext-table for $\Sigma_{(3,2,1,0)}\mQ$}
\label{dimcoho (3,2,1,0)}
\begin{tabular}{cccccc} \toprule
$\mu$ & $K^i$& $\oH^j$& $\Sigma_{\mu}$ & dim & mult \\
					\midrule
				$(6,4,2,0\mid 3,3)$ & 2 & 4 & (4,3,3,1,1)& 3969 &1\\
				$(6,4,2,0\mid 3,3)$ & 2 & 4 & (4,2,2,2,2)& 405 &1\\
				$(6,4,2,0\mid 3,3)$ & 4 & 6 & (4,4,4,4,2)& 1134 &1\\
				$(6,4,2,0\mid 3,3)$ & 3 & 5 & (4,4,3,2,2)& 3240 &1\\
				$(6,4,2,0\mid 3,3)$ & 1 & 3 & (4,2,2,1)& 3240 &1\\
				$(6,4,2,0\mid 3,3)$ & 1 & 3 & (4,2)& 1134 &1\\
				$(6,4,1,1\mid 3,3)$ & 4 & 6 & (3,3,3,3)& 490 &1\\
				$(6,4,1,1\mid 3,3)$ & 3 & 5 & (3,3,2,1)& 1960 &1\\
				$(6,4,1,1\mid 3,3)$ & 2 & 4 & (3,1,1,1)& 280 &1\\
				$(6,3,3,0\mid 3,3)$ & 4 & 6 & (4,4,4,3,3)& 840 &1\\
				$(6,3,3,0\mid 3,3)$ & 2 & 4 & (4,3,2,2,1)& 3675 &1\\
				$(6,3,3,0\mid 3,3)$ & 0 & 2 & (4,1,1)& 840 &1\\
				$(6,3,2,1\mid 3,3)$ & 4 & 6 & (3,3,3,2,1)& 896 &2\\
				$(6,3,2,1\mid 3,3)$ & 3 & 5 & (3,3,1,1,1)& 560 &2\\
				$(6,3,2,1\mid 3,3)$ & 2 & 4 & (3,2,1)& 896 &2\\
				$(6,2,2,2\mid 3,3)$ & 4 & 6 & (2,2,2)& 175 &1\\
				$(6,2,2,2\mid 3,3)$ & 1 & 3 & (3)& 56 &1\\
				$(5,5,2,0\mid 3,3)$ & 2 & 4 & (3,3,2,2,2)& 280 &1\\
				$(5,5,2,0\mid 3,3)$ & 1 & 3 & (3,3,2,1)& 1960 &1\\
				$(5,5,2,0\mid 3,3)$ & 0 & 2 & (3,3)& 490 &1\\
				$(5,5,1,1\mid 3,3)$ & 2 & 4 & (2,2,1,1)& 189 &1\\
				$(5,4,3,0\mid 3,3)$ & 2 & 4 & (3,3,3,2,1)& 896 &2\\
				$(5,4,3,0\mid 3,3)$ & 1 & 3 & (3,2,2,2)& 560 &2\\
				$(5,4,3,0\mid 3,3)$ & 0 & 2 & (3,2,1)& 896 &2\\
				$(5,4,2,1\mid 3,3)$ & 2 & 4 & (2,2,2)& 175 &4\\
				$(5,4,2,1\mid 3,3)$ & 2 & 4 & (2,1,1,1,1)& 35 &4\\
				$(5,3,3,1\mid 3,3)$ & 2 & 4 & (2,2,1,1)& 189 &3\\
				$(4,4,4,0\mid 3,3)$ & 3 & 5 & (3,3,3,3,3)& 56 &1\\
				$(4,4,4,0\mid 3,3)$ & 0 & 2 & (2,2,2)& 175 &1\\
				$(4,4,2,2\mid 3,3)$ & 3 & 6 & (1,1,1)& 20 &2\\
				$(4,4,2,2\mid 3,3)$ & 2 & 4 & (0)& 1 &2\\
				$(4,4,2,2\mid 3,3)$ & 1 & 2 & (1,1,1)& 20 &2\\
				$(3,3,3,3\mid 3,3)$ & 0 & 0 & (0)& 1&1\\
                $(3,3,3,3\mid 3,3)$ & 2 & 4 & (0)& 1&1\\
                $(3,3,3,3\mid 3,3)$ & 4 & 8 & (0)& 1&1 \\

					\bottomrule
				\end{tabular}
\end{table}

\section{Chern character of $\Sigma_\lambda\mQ$}\label{section:Chern}

In this section we will provide the preliminary computations needed for the proof of \autoref{theorem A}. In particular we will compute the Chern character of the vector bundles $\Sigma_\lambda\mQ$ on the irreducible symplectic fourfold $X$. 
Before stating the main result, let us introduce the following notation.

For a partition $\lambda=(m,t,s,0)$, we define the following functions:
\begin{align}
    r_\lambda & = \frac{(m+3)(t+2)(s+1)(m-t+1)(m-s+2)(t-s+1)}{12} \label{eqn:rank} \\
    \ell(\lambda) & = \frac{m+t+s}{4} \label{eqn:ell} \\
    \delta(\lambda) & = \frac{1}{5}\left(4\ell(\lambda)^2+\tau(\lambda)\right) \label{eqn:delta} \\ 
    \tau(\lambda) & = \frac{-2(mt+ms+ts)+3m+t-s}{3} \label{eqn:gamma} \\
    \xi(\lambda) & = \frac{\alpha(t,s)m^2+\beta(t,s)m+\gamma(t,s)}{60} \label{eqn:xi}
\end{align}
where
\begin{align*}
    \alpha(t,s) & = 9t^2 +21ts +9s^2 -33t -30s+21 \\
    \beta(t,s) & = 21t^2s+21ts^2-15t^2-18ts+6s^2+t-8s-6 \\
    \gamma(t,s) & = 9t^2s^2-9t^2s+9ts^2-3t^2+13ts-3s^2-14t+14s.
\end{align*}


\begin{thm}\label{thm:Chern character}
    Let $\lambda=(m,t,s,0)$ be a partition as above. Then
    \begin{align}
        \rk(\Sigma_\lambda\mQ)= & r_\lambda \label{rk}\\
        c_1(\Sigma_\lambda\mQ)= & \ell(\lambda)r_\lambda c_1(\mQ) \label{c1}\\
        \ch_2(\Sigma_\lambda\mQ)= & \frac{1}{2}\left(\ell(\lambda)^2-\frac{1}{4}\delta(\lambda)\right)r_\lambda c_1(\mQ)^2+\delta(\lambda)r_\lambda\ch_2(\mQ) \label{ch2}\\
        \ch_3(\Sigma_\lambda\mQ)= & \tau(\lambda)\ell(\lambda)r_\lambda\ch_3(\mQ) \label{ch3} \\
        \ch_4(\Sigma_\lambda\mQ)= & \xi(\lambda)r_\lambda\ch_4(\mQ). \label{ch4}
    \end{align}
\end{thm}
The rank formula is the very well-known Weyl formula, and the formula for the first Chern class is a particular case of \cite[Theorem~2]{Rub13}; the formulas for the higher Chern classes seems to be new in literature, even in this very special case.
\begin{remark}
    Notice that one can also write
    \[ \ch_2(\Sigma_\lambda\mQ)=\frac{1}{10}\left(4\ell(\lambda)^2-\frac{1}{4}\tau(\lambda)\right)r_\lambda c_1(\mQ)^2+\frac{1}{5}\left(4\ell(\lambda)^2+\tau(\lambda)\right)r_\lambda\ch_2(\mQ). \]
\end{remark}
The reason to state the result only for partitions with last entry equal to zero is motivated by the fact that we can always reduce to this case after a twist by a suitable line bundle. We will explain this in Section~\ref{section:reduction}. Moreover we will explain in Section~\ref{section:strategy} the strategy to prove the theorem and we will perform the proof in Section~\ref{section:proof of Chern}.

\subsection{Strategy of proof}\label{section:strategy}

The main idea is to use several strong inductions on the entries of $\lambda$. The starting point of the argument is the case $\lambda=(m,0,0,0)$: in this case $\Sigma_\lambda\mQ=\Sym^{m}\mQ$ and the computations have been performed in Section~\ref{section:Sym Q} thanks to the work of \cite{Dragutin}. 

Next we will use Pieri's rule to perform the inductive steps. We will divide the proof in three cases.
\medskip

{\bf $\bullet$ \underline{Case $s=0$}.}
Here we will deal with partitions of the form $\lambda=(m,t,0,0)$. Pieri's rule yields the decomposition
\begin{equation}\label{eqn:Pieri 1}
\Sym^m\mQ\otimes\Sym^t\mQ=\bigoplus_{i=0}^t\Sigma_{(m+t-i,i,0,0)}\mQ,
\end{equation}
which allows us to perform an induction on the index $t$, using the case $t=0$ as base of the induction.
\medskip

{\bf $\bullet$ \underline{Case $s=1$}.}
Here we will deal with the partition $\lambda=(m,t,1,0)$. The corresponding plethysm appear in the decomposition
\begin{equation}\label{eqn:Pieri m,t,1}
\Sigma_{(m,t,0,0)}\mQ\otimes\mQ=\Sigma_{(m+1,t,0,0)}\mQ\oplus\Sigma_{(m,t+1,0,0)}\mQ\oplus\Sigma_{(m,t,1,0)}\mQ
\end{equation}
(notice that if $m=t$, then the middle term does not appear). Using the case before, this will allow us to prove the theorem for all such partitions.
\medskip

{\bf $\bullet$ \underline{Case $s>1$}.}
Finally, here we will deal with the general case when $\lambda=(m,t,s,0)$. 
By Pieri's rule again, we have the decomposition
\begin{equation}\label{eqn:Pieri 2}
\Sigma_{(m,t,s,0)}\mQ\otimes\mQ=\Sigma_{(m+1,t,s,0)}\mQ\oplus\Sigma_{(m,t+1,s,0)}\mQ\oplus\Sigma_{(m,t,s+1,0)}\mQ\oplus\Sigma_{(m,t,s,1)}\mQ
\end{equation}
(as before the extremal cases in this decomposition will give less factors). By the reduction trick (see equation (\ref{eqn:reduction})) applied to the last factor we get
\[ \ch(\Sigma_{(m,t,s,1)}\mQ)=\ch(\Sigma_{(m-1,t-1,s-1,0)}\mQ)\ch(\mO_X(1)) \]
which can be computed explicitly thanks to the inductive hypothesis.

\subsection{Proof of Theorem~\ref{thm:Chern character}}\label{section:proof of Chern}

The fact that $\rk(\Sigma_\lambda\mQ)=r_\lambda$ is classical. In fact, clearly $\rk(\Sigma_\lambda\mQ)=\rk(\Sigma_\lambda\tilde{\mQ})$ and the latter is the dimension of the irreducible representation of $\SL(4)$ of maximal weight $\lambda$, see for example \cite{Weyman}. 

Also the formula for the first Chern class is known in literature and can be found as a special case of \cite[Theorem~2]{Rub13}.

Therefore we are left with proving the remaining Chern characters, which we will do in the rest of this section by following the strategy outlined in Section~\ref{section:strategy}, which in particular consists of three cases.

{\bf $\bullet$ \underline{Case $s=0$}.} Let us suppose that $\lambda=(m,t,0,0)$, with $t\geq1$, and that the claim holds for all partitions $\mu=(\mu_1,\mu_2,0,0)$ with $\mu_1\leq m+t$ and $\mu_2\leq t-1$. Using the decomposition (\ref{eqn:Pieri 1}) we get the following equality
\begin{equation}\label{eqn:step 1 s=0} \ch(\Sigma_\lambda\mQ)=\ch(\Sym^m\mQ\otimes\Sym^t\mQ)-\sum_{i=0}^{t-1}\Sigma_{(m+t-i,i,0,0)}\mQ, 
\end{equation}
where, by inductive hypothesis, all the members on the right-hand side are determined.

Let us explicitly write the term $\ch(\Sym^m\mQ\otimes\Sym^t\mQ)$. To keep the notation more compressed, we used the following notation:
\[ r_m=\binom{m+3}{3}\qquad\mbox{ and }\qquad r_t=\binom{t+3}{3}. \]
By multiplicativity of the Chern character, the inductive hypothesis and the computations in Section~\ref{section:Sym Q} we get
\begin{align*}
    \rk(\Sym^m\mQ\otimes\Sym^t\mQ) & = r_mr_t \\
    c_1(\Sym^m\mQ\otimes\Sym^t\mQ) & = \frac{m+t}{4}r_mr_t c_1(\mQ) \\
    \ch_2(\Sym^m\mQ\otimes\Sym^t\mQ) & = \frac{1}{8}\left(\frac{t(t-1)}{5}+\frac{mt}{2}+\frac{m(m-1)}{5}\right)r_mr_t c_1(\mQ)^2 + \\
     & + \frac{1}{20}\left(t(t+4)+m(m+4)\right)r_mr_t \ch_2(\mQ) \\
    \ch_3(\Sym^m\mQ\otimes\Sym^t\mQ) & = \frac{2t^2+2m^2-mt(m+t-4)}{8}r_m r_t\ch_3(\mQ) \\
    \ch_4(\Sym^m\mQ\otimes\Sym^t\mQ) & = \frac{5m(7m-2)+5t(7t-2)+3mt(3mt-13m-13t+23)}{100}r_m r_t \ch_4(\mQ)
\end{align*}

(We used the identities in Lemma~\ref{lemma:identities ch_3(Q) e ch_4(Q)} to simplify the expressions above.)

\subsubsection*{The second Chern character}


The claim is analogous to the following algebraic result.

\begin{lemma}\label{lemma:ch_2 s=0}
    With notations as above, the following relations hold.
    \begin{enumerate}
        \item $\sum_{i=0}^t\frac{1}{2}\left(\ell(m+t-i,i,0)^2-\frac{1}{4}\delta(m+t-i,i,0)\right)r_{m+t-i,i}=\frac{1}{8}\left(\frac{t(t-1)}{5}+\frac{mt}{2}+\frac{m(m-1)}{5}\right)r_mr_0$.
        \item $\sum_{i=0}^t\delta(m+t-i,i,0)r_{m+t-i,i}=\frac{1}{20}\left(t(t+4)+m(m+4)\right)r_mr_t$.
    \end{enumerate}
\end{lemma}
\begin{proof}
    This is a direct computation.
\end{proof}
Now, re-arranging the terms and using an inductive argument, we see that item $(1)$ of Lemma~\ref{lemma:ch_2 s=0} gives the coefficient in front of $c_1(Q)^2$, while item $(2)$ of Lemma~\ref{lemma:ch_2 s=0} gives the coefficient in front of $\ch_2(Q)$, thus proving the $\ch_2$-part of the theorem in this case.


\subsubsection*{The third Chern character}


  The claim is analogous to the following algebraic result.

\begin{lemma}\label{lemma:ch_3 s=0}
    The following equality holds:
    \[ \sum_{i=0}^t\tau(m+t-i,i,0)r_{m+t-i,i}=\frac{1}{8}(2t^2+2m^2-mt(m+t-4))r_m r_t \]
\end{lemma}
\begin{proof}
    This is a direct computation.
\end{proof}
Again, re-arranging the terms and using the induction argument, this proves the $\ch_3$-part of the theorem in this case.

\subsubsection*{The fourth Chern character}

The claim is analogous to the following algebraic result.

\begin{lemma}\label{lemma:ch_4 s=0}
    The following equality holds:
    \[ \sum_{i=0}^t\xi(m+t-i,i,0)r_{m+t-i,i}=\frac{5m(7m-2)+5t(7t-2)+3mt(3mt-13m-13t+23)}{100}r_m r_t.  \]
\end{lemma}
\begin{proof}
    This is a direct computation.
\end{proof}
As before, this equality confirms the $\ch_4$-part of the theorem in this case.
\medskip

{\bf $\bullet$ \underline{Case $s=1$}.}
Let us now prove the theorem when $\lambda=(m,t,1,0)$. By the previous case, we know that the claim holds for any partition of the form $\mu=(\mu_1,\mu_2,0,0)$. By the decomposition (\ref{eqn:Pieri m,t,1}) we have the equality
\begin{equation}\label{eqn:ch m,t,1} 
\ch(\Sigma_\lambda\mQ)=\ch(\Sigma_{(m,t,0,0)}\mQ\otimes\mQ)-\ch(\Sigma_{(m+1,t,0,0)}\mQ)-\ch(\Sigma_{(m,t+1,0,0)}\mQ), 
\end{equation}
where everything on the right hand side is determined. Let us explicitly right down the Chern character of $\Sigma_{(m,t,0,0)}\mQ\otimes\mQ$. Using the notation $r_{m,t}$ for the rank of $\Sigma_{(m,t,0,0)}$ and the identities in Lemma~\ref{lemma:identities ch_3(Q) e ch_4(Q)}, we get
\begin{align*}
    \rk(\Sigma_{(m,t,0,0)}\mQ\otimes\mQ) & = 4r_{m,t} \\
    c_1(\Sigma_{(m,t,0,0)}\mQ\otimes\mQ) & = (1+4\ell(m,t))r_{m,t}c_1(\mQ)=(1+m+t)r_{m,t}c_1(\mQ) \\
    \ch_2(\Sigma_{(m,t,0,0)}\mQ\otimes\mQ) & = \left(2\ell(m,t)^2+\ell(m,t)-\frac{1}{2}\delta(m,t)\right)r_{m,t}c_1(\mQ)^2 + \\
    & +(4\delta(m,t)+1)r_{m,t}\ch_2(\mQ) \\
    \ch_3(\Sigma_{(m,t,0,0)}\mQ\otimes\mQ) & = (1+2\ell(m,t)(1-4\ell(m,t))+\tau(m,t)(1+4\ell(m,t))r_{m,t}\ch_3(\mQ) \\
    \ch_4(\Sigma_{(m,t,0,0)}\mQ\otimes\mQ) & = \left(1+6\ell(m,t)(\frac{4}{5}\ell(m,t)-1)+3\tau(m,t)(\frac{9}{10}-2\ell(m,t))+4\xi(m,t)\right)r_{m,t}\ch_4(\mQ).
\end{align*}

From now on we assume that $m>t$. The case when $m=t$ is left to the reader: the last term in the equality (\ref{eqn:ch m,t,1}) disappears, so making the computations shorter (and essentially the same). 

\subsubsection*{The second Chern character}

The claim is analogous to the following algebraic result.

\begin{lemma}\label{lemma:ch_2 s=1}
    The following equalities hold.
    \begin{enumerate}
        \item $\frac{1}{2}\left(\ell(m+1,t,0)^2-\frac{1}{4}\delta(m+1,t,0)\right)r_{m+1,t}+\frac{1}{2}\left(\ell(m,t+1,0)^2-\frac{1}{4}\delta(m,t+1,0)\right)r_{m,t+1}+ \\ 
        +\frac{1}{2}\left(\ell(m,t,1)^2-\frac{1}{4}\delta(m,t,1)\right)r_{m,t,1}=\left(2\ell(m,t)^2+\ell(m,t)-\frac{1}{2}\delta(m,t)\right)r_{m,t}$.
        \item $\delta(m+1,t,0)r_{m+1,t}+\delta(m,t+1,0)r_{m,t+1,0}+\delta(m,t,1)r_{m,t,1}=(4\delta(m,t)+1)r_{m,t}$.
    \end{enumerate}
\end{lemma}
\begin{proof}
    This is a direct computation.
\end{proof}
As before, these equalities and the inductive argument imply the $\ch_2$-part of the theorem in this case.

\subsubsection*{The third Chern character}

The claim is analogous to the following algebraic result.

\begin{lemma}\label{lemma:ch_3 s=1}
    The following equality holds:
\begin{align*} \tau(m+1,t,0)r_{m+1,t}+&\tau(m,t+1,0)r_{m,t+1,0}+\tau(m,t,1)r_{m,t,1}=\\=(1+2\ell(m,t)(1-& 4\ell(m,t))+\tau(m,t)(1+4\ell(m,t))r_{m,t}.
\end{align*}
\end{lemma}
\begin{proof}
    This is a direct computation.
\end{proof}
As before, this proves the $\ch_3$-part of the theorem in this case.

\subsubsection*{The fourth Chern character}

The claim is analogous to the following algebraic result.

\begin{lemma}\label{lemma:ch_4 s=1}
    The following equality holds:
    \[
    \begin{array}{l}
    \xi(m+1,t,0)r_{m+1,t}+\xi(m,t+1,0)r_{m,t+1,0}+\xi(m,t,1)r_{m,t,1}= \\ 
    =\left(1+6\ell(m,t)(\frac{4}{5}\ell(m,t)-1)+3\tau(m,t)(\frac{9}{10}-2\ell(m,t))+4\xi(m,t)\right)r_{m,t}. 
    \end{array}\]
\end{lemma}
\begin{proof}
    This is a direct computation.
\end{proof}
As before, this proves the $\ch_4$-part of the theorem in this case.
\medskip

{\bf $\bullet$ \underline{Case $s>1$}.}
Let us prove the theorem for $\lambda=(m,t,s,0)$, by induction on $s$ (using the case $s=1$ above as base of the induction). Starting with the decomposition (\ref{eqn:Pieri 2}), we can write
\[ \ch(\Sigma_{m,t,s+1,0)}\mQ)=\ch(\Sigma_{(m,t,s,0)}\mQ\otimes\mQ)-\ch(\Sigma_{(m+1,t,s,0)}\mQ)-\ch(\Sigma_{(m,t+1,s,0)}\mQ)-\ch(\Sigma_{(m,t,s,1)}\mQ). \]

Now, as before we have
\begin{align*}
    \rk(\Sigma_{(m,t,s,0)}\mQ\otimes\mQ) & = 4r_{m,t,s} \\
    c_1(\Sigma_{(m,t,s,0)}\mQ\otimes\mQ) & = (1+4\ell(m,t,s))r_{m,t,s}c_1(\mQ)=(1+m+t+s)r_{m,t,s}c_1(\mQ) \\
    \ch_2(\Sigma_{(m,t,s,0)}\mQ\otimes\mQ) & = \left(2\ell(m,t,s)^2+\ell(m,t,s)-\frac{1}{2}\delta(m,t,s)\right)r_{m,t,s}c_1(\mQ)^2 + \\
    & +(4\delta(m,t,s)+1)r_{m,t,s}\ch_2(\mQ) \\
    \ch_3(\Sigma_{(m,t,s,0)}\mQ\otimes\mQ) & = (1+2\ell(m,t,s)(1-4\ell(m,t,s))+\tau(m,t,s)(1+4\ell(m,t,s)))r_{m,t,s}\ch_3(\mQ) \\
    \ch_4(\Sigma_{(m,t,s,0)}\mQ\otimes\mQ) & = \left(1+6\ell(m,t,s)(\frac{4}{5}\ell(m,t,s)-1)+3\tau(m,t,s)(\frac{9}{10}-2\ell(m,t,s))+4\xi(m,t,s)\right)r_{m,t,s}\ch_4(\mQ).
\end{align*}

On the other hand, by the reduction trick of Section~\ref{section:reduction}, we have
\[ \ch(\Sigma_{(m,t,s,1)}\mQ)=\ch(\Sigma_{(m-1,t-1,s-1,0)})\ch(\mO_X(1)). \]
Let us put $r_1:=r_{(m,t,s,1)}$, $m'=m-1$, $t'=t-1$ and $s'=s-1$.
Rearranging the terms and using the identities of Lemma~\ref{lemma:identities ch_3(Q) e ch_4(Q)}, we get then
\begin{align*}
    c_1(\Sigma_{(m,t,s,1)}\mQ) & = (\ell(m',t',s')+1)r_1c_1(\mQ) \\
    \ch_2(\Sigma_{(m,t,s,1)}\mQ) & = \frac{1}{2}\left(\ell(m',t',s')^2+2\ell(m',t',s')-\frac{1}{4}\delta(m',t',s')+1\right)r_{1}c_1(\mQ)^2 +\delta(m',t',s')r_{1}\ch_2(\mQ) \\
    \ch_3(\Sigma_{(m,t,s,1)}\mQ) & = -\left( 12\ell(m',t',s')^2+12\ell(m',t',s')-\ell(m',t',s')\tau(m',t',s')-5\delta(m',t',s')+4\right)r_1\ch_3(\mQ) \\
    \ch_4(\Sigma_{(m,t,s,1)}\mQ) & = \left(6+6\ell(m',t',s')(4-\tau(m',t',s'))+36\ell(m',t',s')^2-15\delta(m',t',s')+\xi(m',t',s')\right)r_1\ch_4(\mQ).
\end{align*}

We proceed as in the previous cases by analysing each Chern character separately.

\subsubsection*{The second Chern character}


The claim is analogous to the following algebraic result.

\begin{lemma}
    Using the same notations as above, the following identities hold.
    \begin{enumerate}
        \item $\frac{1}{2}\left(\ell(m+1,t,s)^2-\frac{1}{4}\delta(m+1,t,s)\right)r_{m+1,t,s}+\frac{1}{2}\left(\ell(m,t+1,s)^2-\frac{1}{4}\delta(m,t+1,s)\right)r_{m,t+1,s}+ \\ 
        +\frac{1}{2}\left(\ell(m,t,s+1)^2-\frac{1}{4}\delta(m,t,s+1)\right)r_{m,t,s+1}+ \\ +\frac{1}{2}\left(\ell(m',t',s')^2+2\ell(m',t',s')-\frac{1}{4}\delta(m',t',s')+1\right)r_{1} = \\ =\left(2\ell(m,t,s)^2+\ell(m,t,s)-\frac{1}{2}\delta(m,t,s)\right)r_{m,t,s}$.
        \item $\delta(m+1,t,s)r_{m+1,t,s}+\delta(m,t+1,s)r_{m,t+1,s}+\delta(m,t,s+1)r_{m,t,s+1}+\delta(m',t',s')r_1= \\
        =(4\delta(m,t,s)+1)r_{m,t,s}$.
    \end{enumerate}
\end{lemma}
\begin{proof}
    This is a direct computation.
\end{proof}

As before, this proves the $\ch_2$-part of the theorem in this case.


\subsubsection*{The third Chern character}

The claim is analogous to the following algebraic result.

\begin{lemma}
    The following equality holds:
    \[ \begin{array}{l}
     \tau(m+1,t,s)r_{m+1,t,s}+\tau(m,t+1,s)r_{m,t+1,s}+\tau(m,t,s+1)r_{m,t,s+1} - \\
     -\left( 12\ell(m',t',s')^2+12\ell(m',t',s')-\ell(m',t',s')\tau(m',t',s')-5\delta(m',t',s')+4\right)r_1= \\
     = (1+2\ell(m,t,s)(1-4\ell(m,t,s))+\tau(m,t,s)(1+4\ell(m,t,s)))r_{m,t,s}.
    \end{array} \]
\end{lemma}
\begin{proof}
    This is a direct computation.
\end{proof}  
As usual, this proves the $\ch_3$-part of the theorem in this case.

\subsubsection*{The fourth Chern character}

   The claim is analogous to the following algebraic result.

\begin{lemma}
    The following equality holds:
    \[ \begin{array}{l}
     \xi(m+1,t,s)r_{m+1,t,s}+\xi(m,t+1,s)r_{m,t+1,s}+\xi(m,t,s+1)r_{m,t,s+1} + \\
     \left(6+6\ell(m',t',s')(4-\tau(m',t',s'))+36\ell(m',t',s')^2-15\delta(m',t',s')+\xi(m',t',s')\right)r_1= \\
     = \left(1+6\ell(m,t,s)(\frac{4}{5}\ell(m,t,s)-1)+3\tau(m,t,s)(\frac{9}{10}-2\ell(m,t,s))+4\xi(m,t,s)\right)r_{m,t,s}.
    \end{array} \]
\end{lemma}
\begin{proof}
    This is a direct computation.
\end{proof}  

Finally, this proves the $\ch_4$-part of the theorem in this case, thus concluding the entire proof.

\section{Proof of Theorem~\ref{theorem A}}\label{section:proof of thm A}
Here $\lambda=(\lambda_1,\lambda_2,\lambda_3,\lambda_4)$ is a partition and $\Sigma_\lambda\mQ$ is the associated vector bundle on $X$. From now on we will always assume that $X$ is very general, i.e.\ $\operatorname{Pic}(X)$ has rank $1$ and is generated by the restriction of the Pl\"ucker class.

\begin{proposition}
    The vector bundle $\Sigma_\lambda\mQ$ is slope polystable.
\end{proposition}
\begin{proof}
The proof is divided in a few steps, but before proceeding, let us simplify a bit our analysis. Since tensoring with a multiple of the polarisation does not change the stability, using the reduction trick explained in Section~\ref{section:reduction} it will be enough to prove the proposition for partitions of the form $\lambda=(m,t,s,0)$. 

First of all we have already noticed in Corollary~\ref{corollary:recap on sym} that the symmetric bundles $\Sym^m\mQ$ are slope stable.

By \cite[Theorem~3.2.11]{HL} we then have that $\Sym^m\mQ\otimes\Sym^t\mQ$ is slope polystable. In particular every factor in the decomposition (\ref{eqn:Pieri 1}) is also slope polystable. 

The rest of proof proceeds in the same way, where the decompositions (\ref{eqn:Pieri m,t,1}) and (\ref{eqn:Pieri 2}) are used. We leave the details to the reader.
\end{proof}

What prevents us to claim the slope stability of $\Sigma_\lambda\mQ$ is the fact that we do not have an a priori control of $\Hom(\Sigma_\lambda\mQ,\Sigma_\lambda\mQ)$. Conjecture~\ref{conjecture:simple+ext 1} would imply that the $\Sigma_\lambda\mQ$'s are indeed stable. We exhibit in Section~\ref{section:Evidences} many cases for which the simplicty of $\Sigma_\lambda\mQ$ can be deduced.

Next we prove the modularity.
\begin{proposition}\label{prop:modular}
    The vector bundle $\Sigma_\lambda\mQ$ is modular. More precisely,
     \[ \Delta(\Sigma_\lambda\mQ)=\frac{\delta(\lambda)}{4}\operatorname{rk}(\Sigma_\lambda\mQ)^2\mathsf{c}_2(X), \]
    where 
    \[ \delta(\lambda)=\frac{3(\sum_{i=1}^4\lambda_i^2)-2(\sum_{1\leq i<j\leq4}\lambda_i\lambda_j)+12\lambda_1+4\lambda_2-4\lambda_3-12\lambda_4}{60}. \]
\end{proposition}
\begin{proof}
    First of all, as we recalled in Section~\ref{section:reduction}, there is an equality
    \[ \Sigma_\lambda\mQ=\Sigma_{\lambda'}\mQ\otimes\mO_X(\lambda_4), \]
where $\lambda'=(\lambda_1-\lambda_4,\lambda_2-\lambda_4,\lambda_3-\lambda_4,0)$, and clearly 
\[ \Delta(\Sigma_\lambda\mQ)=\Delta(\Sigma_{\lambda'}\mQ). \]
Since $\lambda'_4=0$, we can apply Theorem~\ref{thm:Chern character} to explicitly compute the discriminant:
\[ \Delta(\Sigma_{\lambda'}\mQ)=\frac{\delta(\lambda')}{4}\operatorname{rk}(\Sigma_{\lambda'}\mQ)^2\Delta(\mQ). \]
Finally it is enough to define
\[ \delta(\lambda):=\delta(\lambda'), \]
notice that $\operatorname{rk}(\Sigma_{\lambda'}\mQ)=\operatorname{rk}(\Sigma_{\lambda}\mQ)$ and recall that $\Delta(\mQ)=\mathsf{c}_2(X)$ (see for example \cite[Section~2.1]{O'Grady:Modular}).
\end{proof}

\section{Proof of Theorem~\ref{theorem B}}\label{section:proof of thm B}

In this section we want to study which of the modular vector bundles $\Sigma_\lambda\mQ$ are atomic. 

We start by computing the Euler characteristic of the bundles $\Sigma_\lambda\mQ$ via Hirzebruch--Riemann--Roch.

\subsection{The Euler characteristic of $\Sigma_\lambda\mQ$}\label{section:euler char}

We have already explicitly computed the discriminant of $\Sigma_\lambda\mQ$ in Proposition~\ref{prop:modular}, so that we compute now the Euler characteristic of $\Sigma_\lambda\mQ$ in order to compare it with their values in case they were atomic. 

\begin{proposition}\label{prop:eulero}
    There exists a degree $4$ polynomial $\mathsf{P}\in\QQ[x_1,x_2,x_3,x_4]$ such that the Euler characteristic of the vector bundle $\Sigma_\lambda\mQ$ is
    \[ \chi(\Sigma_\lambda\mQ,\Sigma_\lambda\mQ)= 3 \mathsf{P}(\lambda) r_\lambda^2. \]
\end{proposition}
The polynomial $\mathsf{P}(\lambda)$ is explicitly determined at the end of the proof as a combination of the polynomials $\ell$, $\tau$ and $\xi$ introduced at the beginning of Section~\ref{section:Chern}.
\begin{proof}
    As in the proof of Proposition~\ref{prop:modular}, we use the reduction trick of Section~\ref{section:reduction} to get
    \[ \chi(\Sigma_\lambda\mQ,\Sigma_\lambda\mQ)=\chi(\Sigma_{\lambda'}\mQ,\Sigma_{\lambda'}\mQ), \]
    where $\lambda'=(\lambda_1-\lambda_4,\lambda_2-\lambda_4,\lambda_3-\lambda_4,0)$.

    As usual, if we put 
    \[ \Xi(\Sigma_{\lambda'}\mQ)=\ch_2(\Sigma_{\lambda'}\mQ)^2-2c_1(\Sigma_{\lambda'}\mQ)\ch_3(\Sigma_{\lambda'}\mQ)+2\rk(\Sigma_{\lambda'}\mQ)\ch_4(\Sigma_{\lambda'}\mQ) \]
    then 
    \[ \ch(\Sigma_{\lambda'}\mQ)=\left(r_{\lambda'}^2,0,-\Delta(\Sigma_{\lambda'}\mQ),0,\Xi(\Sigma_{\lambda'}\mQ)\right). \]
    Since 
    \[ \operatorname{td}_X=\left(1,0,\frac{1}{12}c_2(X),0,3\mathsf{p}\right) \]
    (where $\mathsf{p}$ is the class of a point)
    by Hirzebruch--Riemann--Roch formula we get
    \begin{align*} \chi(\Sigma_{\lambda'}\mQ,\Sigma_{\lambda'}\mQ) & =\int_X\ch(\mE nd(\Sigma_{\lambda'}\mQ)\operatorname{td}_X \\
    & =3r_{\lambda'}^2-\frac{1}{12}\int_X \Delta(\Sigma_{\lambda'}\mQ).\mathsf{c}_2(X)+\int_X\Xi(\Sigma_{\lambda'}\mQ). 
\end{align*}
By Proposition~\ref{prop:modular} we have that 
\[ \Delta(\Sigma_{\lambda'}\mQ)=\frac{\delta(\lambda')}{4}r_{\lambda'}^2\mathsf{c}_2(X). \]
Using the equality $\int_X\mathsf{c}_2(X)^2=828$ (see \cite[Equality~(5.3.4)]{O'Grady:Modular}), we get
\[ -\frac{1}{12}\int_X\Delta(\Sigma_{\lambda'}\mQ).\mathsf{c}_2(X)=-\frac{69}{4}\delta(\lambda')r_{\lambda'}^2=-\frac{69}{20}(4\ell(\lambda')^2+\tau(\lambda'))r_{\lambda'}^2, \]
where the last equality is the definition $\delta(\lambda')=\frac{1}{5}(4\ell(\lambda')^2+\tau(\lambda'))$.
Now, by Theorem~\ref{thm:Chern character} and Lemma~\ref{lemma:identities ch_3(Q) e ch_4(Q)} we get
\[ \Xi(\Sigma_{\lambda'}\mQ)=\left( 36\ell(\lambda')^4-30\ell(\lambda')^2\delta(\lambda')+\frac{69}{4}\delta(\lambda')^2+12\ell(\lambda')^2\tau(\lambda')+2\xi(\lambda') \right) r_{\lambda'}^2\ch_4(\mQ). \]
Since by Lemma~\ref{lemma:c_2(X)^2 and ch_4(Q)} we know that $\int_X\ch_4(\mQ)=\frac{3}{4}$, we finally get
\begin{align*} \int_X\Xi(\Sigma_{\lambda'}\mQ) & =\left( 27\ell(\lambda')^4-\frac{45}{2}\ell(\lambda')^2\delta(\lambda')+\frac{207}{16}\delta(\lambda')^2+9\ell(\lambda')^2\tau(\lambda')+\frac{3}{2}\xi(\lambda') \right)r_{\lambda'}^2 \\
 & =\frac{3}{100}\left( 576\ell(\lambda')^4+288\ell(\lambda')^2\tau(\lambda')+\frac{69}{4}\tau(\lambda')^2+50\xi(\lambda')\right) r_{\lambda'}^2,
\end{align*}
where again the last equality follows from $\delta(\lambda')=\frac{1}{5}(4\ell(\lambda')^2+\tau(\lambda'))$.
Putting all together, we have
\[ \chi(\Sigma_{\lambda'}\mQ,\Sigma_{\lambda'}\mQ)=3\left\{ 1-\frac{23(4\ell(\lambda')^2+\tau(\lambda'))}{20}+\frac{ 576\ell(\lambda')^4+288\ell(\lambda')^2\tau(\lambda')+\frac{69}{4}\tau(\lambda')^2+50\xi(\lambda')}{100}\right\} r_{\lambda'}^2. \]
Therefore the claim holds with 
\begin{equation}\label{eqn:chi(Sigma_lambdaQ)} 
\mathsf{P}(\lambda):=1-\frac{23(4\ell(\lambda')^2+\tau(\lambda'))}{20}+\frac{ 576\ell(\lambda')^4+288\ell(\lambda')^2\tau(\lambda')+\frac{69}{4}\tau(\lambda')^2+50\xi(\lambda')}{100}. 
\end{equation}
\end{proof}

\begin{remark}
    If $\lambda=(m,0,0,0)$ or $\lambda=(2,1,0,0)$, we have seen in Section~\ref{section:Sym Q} and Section~\ref{section:Sigma 2 1} that $\mathsf{P}(\lambda)$ is a perfect square. On the other hand, already in the case $\lambda=(3,2,1,0)$ considered in Section~\ref{section:Sigma 3 2 1} we have $\mathsf{P}(3,2,1,0)=\frac{23}{8}$, which is not a perfect square.
\end{remark}

\subsection{Proof of Theorem~\ref{theorem B}}
Let us recall the statement of Theorem~\ref{theorem B}.

\begin{thm}[Theorem~\ref{theorem B}]
    Let $\lambda=(\lambda_1,\lambda_2,\lambda_3,\lambda_4)$ be a partition. Then $\Sigma_\lambda\mQ$ is atomic if and only if $\lambda_1=\lambda_2=\lambda_3$ or $\lambda_2=\lambda_3=\lambda_4$.
\end{thm}

We will prove each implication in a different proposition. 

\begin{proposition}\label{prop:not atomic}
    Let $\lambda=(\lambda_1,\lambda_2,\lambda_3,\lambda_4)$ be a partition. If at most two entries are equal, then $\Sigma_\lambda\mQ$ is not atomic.
\end{proposition}
\begin{proof}
    First of all, by Proposition~\ref{prop:eulero} and Lemma~\ref{lemma:Bottini} we get that if $\mathsf{P}(\lambda)$ is not the square of a rational number, then $\Sigma_\lambda\mQ$ cannot be atomic.

    Let us then consider those partitions $\lambda$ for which $\mathsf{P}(\lambda)$ is of the form $q_\lambda^2$ for some $q_\lambda\in\QQ$. If $\Sigma_\lambda\mQ$ were atomic, then let us denote by $\tilde{v}(\Sigma_\lambda\mQ)$ its extended Mukai vector. By Lemma~\ref{lemma:Bottini}.(1) and Theorem~\ref{theorem A} then we must have
    \[ \tilde{q}(\tilde{v}(\Sigma_\lambda\mQ))=\frac{5}{2}q_\lambda r_\lambda^2. \]
    
   On the other hand, by Lemma~\ref{lemma:Bottini}.(2) and Proposition~\ref{prop:modular} we must also have
    \[ \tilde{q}(\tilde{v}(\Sigma_\lambda\mQ))=\frac{5}{2}(3\delta(\lambda)-1) r_\lambda^2. \]
    Putting these two equality together we get that
    \[ q_\lambda=3\delta(\lambda)-1. \]
    One can directly see that the equation
    \[ \left(3\delta(\lambda)-1\right)^2-\mathsf{P}(\lambda)=0 \]
    has no integral solutions as soon as at most two entries are equal.
\end{proof}


Let us then prove the reverse implication.

\begin{proposition}
    If $\lambda_1=\lambda_2=\lambda_3$ or $\lambda_2=\lambda_3=\lambda_4$, then $\Sigma_\lambda\mQ$ is atomic.
\end{proposition}
\begin{proof}
    Let us first reduce the proof to the case of symmetric bundles. Of course if $\lambda_1=\lambda_2=\lambda_3=\lambda_4$, then $\Sigma_\lambda\mQ=\mO_X$ is atomic, therefore we can suppose that at least two entries are different.

    Under this assumption we have only two cases, namely $\lambda=(\lambda_1,\lambda_2,\lambda_2,\lambda_2)$ and $\lambda=(\lambda_1,\lambda_1,\lambda_1,\lambda_4)$.

    In the first case, by the reduction trick of Section~\ref{eqn:reduction}, we have
    \[ \Sigma_{(\lambda_1,\lambda_2,\lambda_2,\lambda_2)}\mQ=\Sym^{\lambda_1-\lambda_2}\mQ\otimes\mO_X(\lambda_2). \]
    In the second case, by Example~\ref{example:duale} we have
    \[ \Sigma_{(\lambda_1,\lambda_1,\lambda_1,\lambda_4)}\mQ=\Sym^{\lambda_1-\lambda_4}\mQ^\vee\otimes\mO_X(\lambda_1). \]
    Since tensoring with a line bundle and passing to the dual are operations that preserve the atomicity (\cite[Proposition~3.10]{Beckmann}), without loss of generality we can reduce to the case $\lambda=(m,0,0,0)$, so that $\Sigma_\lambda\mQ=\Sym^m\mQ$ and the claim follows at once from Proposition~\ref{prop:Sym atomic}.   
\end{proof}

\section{Behaviour of the Ext-groups}\label{section:Ext groups}
Computing the Ext-groups of the vector bundles $\Sigma_\lambda\mQ$ can be very difficult. In fact the spectral sequence (\ref{eqn:spectral sequence}) can have non-trivial differentials, and a systematic and a priori analysis seems out of reach as for now. 

Nevertheless in many cases the spectral sequence has enough zero differentials to allow us to determine at least the groups $\Ext^0$ and $\Ext^1$. 

For example we have the following result.

 \begin{proposition}\label{prop:sym simple and rigid}
        If $\lambda_1=\lambda_2=\lambda_3$ or $\lambda_2=\lambda_3=\lambda_4$, then $\Sigma_\lambda\mQ$ is simple and rigid.
    \end{proposition}
    \begin{proof}
        First of all, using the reduction trick in Section~\ref{section:reduction} we can assume that $\lambda_4=0$. We have then only two cases to consider: $\lambda=(m,m,m,0)$ and $\lambda=(m,0,0,0)$. These two cases are equivalent: in fact by Example~\ref{example:duale} they are dual to each other, up to a twist by a line bundle.

        On the other hand, $\Sigma_{(m,0,0,0)}\mQ=\Sym^m\mQ$ and the result follows from Theorem~\ref{thm:Sym Q}.
    \end{proof}
    Notice in particular that in this cases $\Sigma_\lambda\mQ$ is slope stable.

Based on the examination of many cases, we are led to state the following conjectural behaviour.

\begin{conjA}
    Let $X\subset\Gr(2,V_6)$ be the Fano variety of lines of a smooth cubic fourfold, and let $\lambda=(\lambda_1,\lambda_2,\lambda_3,\lambda_4)$ be a partition. Then 
    \begin{enumerate}
        \item $\Sigma_\lambda\mQ$ is simple;
        \item if $\lambda$ has at most two different entries, then 
        \[ \Ext^1(\Sigma_\lambda\mQ,\Sigma_\lambda\mQ)=\left\{
    \begin{array}{ll}
      \W^3V_6^\vee   & \mbox{if}\quad\lambda_1=\lambda_2>\lambda_3\quad\mbox{or}\quad\lambda_1>\lambda_2=\lambda_3>\lambda_4\quad\mbox{or}\quad\lambda_2>\lambda_3=\lambda_4 \\
      (\W^3V_6^\vee)^{\oplus 2} & \mbox{otherwise.}
    \end{array}
    \right.\]
    \end{enumerate}    
\end{conjA}

If the conjecture holds true, then the vector bundles $\Sigma_\lambda\mQ$ are slope stable. Moreover, using Serre duality and the explicit computation of the Euler characteristic in Proposition~\ref{prop:eulero} we would have
    \[ \operatorname{dim}\Ext^2(\Sigma_\lambda\mQ,\Sigma_\lambda\mQ)=\left\{
    \begin{array}{ll}
      3\mathsf{P}(\lambda)r_\lambda^2-2              & \mbox{if}\quad \lambda_1=\lambda_2=\lambda_3\;\mbox{or}\;\lambda_2=\lambda_3=\lambda_4 \\
      3\mathsf{P}(\lambda)r_\lambda^2+38   & \mbox{if}\quad\lambda_1=\lambda_2>\lambda_3\;\mbox{or}\;\lambda_1>\lambda_2=\lambda_3>\lambda_4\;\mbox{or}\;\lambda_2>\lambda_3=\lambda_4 \\
      3\mathsf{P}(\lambda)r_\lambda^2+78 & \mbox{otherwise}
    \end{array}
    \right.\]
    where $\mathsf{P}(\lambda)$ is the polynomial in Proposition~\ref{prop:eulero}.
    Notice that in this case the dimension of $\Ext^2(\Sigma_\lambda\mQ,\Sigma_\lambda\mQ)$  
    diverges to infinity as the first component $\lambda_1$ of $\lambda$ goes to infinity.


    As a first evidence we notice that the $\SL(6)$-representation 
    $\W^3 V_6^\vee$
    is always contained in $\Ext^1(\Sigma_\lambda\mQ,\Sigma_\lambda\mQ)$, with the prescribed multiplicity.

    \begin{proposition}\label{prop:W3 in Ext1}
        For every partition $\lambda$ we have 
       \[ \Ext^1(\Sigma_\lambda\mQ,\Sigma_\lambda\mQ)\supset\left\{
    \begin{array}{ll}
      \W^3V_6^\vee   & \mbox{if}\quad\lambda_1=\lambda_2>\lambda_3\quad\mbox{or}\quad\lambda_1>\lambda_2=\lambda_3>\lambda_4\quad\mbox{or}\quad\lambda_2>\lambda_3=\lambda_4 \\
      (\W^3V_6^\vee)^{\oplus 2} & \lambda_1>\lambda_2>\lambda_3>\lambda_4\, .
    \end{array}
    \right.\]
    \end{proposition}
    \begin{proof}
        First of all, using again the reduction trick in Section~\ref{section:reduction} we can reduce to study the case $\lambda=(m,t,s,0)$. 

        We claim that in the Littlewood--Richardson decomposition of $\mE nd(\Sigma_\lambda\mQ)$ we can always find the factor $\Sigma_{(m+1,m+1,m-1,m-1|m,m)}$ with multiplicity $1$ if $m=t$ or $t=s$, and with multiplicity $2$ if $m>t>s>0$.

        Let us show how the claim finishes the proof. First of all, by equality (\ref{eqn: + n}) we have 
        \[ \Sigma_{(m+1,m+1,m-1,m-1|m,m)}=\Sigma_{(2,2,0,0|1,1)}. \]
        Now, by the Borel--Weil--Bott Theorem (Section~\ref{section:BBW}) it is easy to see that 
        \[ \oH^1(X,\Sigma_{(2,2,0,0|1,1)})=\Sigma_{(1,1,1,0,0,0)}=\W^3 V_6^\vee, \]
        from which the proposition follows.

        Let us then prove the claim. First of all we notice that by Example~\ref{example:duale} it is enough to prove that the Schur module
        $\Sigma_{(m+1,m+1,m-1,m-1)}$ appears in the Littlewood--Richardson decomposition of 
        \begin{equation}\label{eqn:end meno m} 
    \Sigma_{(m,t,s,0)}\mQ\otimes\Sigma_{(m,m-s,m-t,0)}\mQ 
        \end{equation}
        with the prescribed multiplicity. Moreover, being the partition $(m+1,m+1,m-1,m-1)$ fixed, we only need to check that it's multiplicity is $1$ or $2$. We proceed by a case-by-case inspection.
        \begin{itemize}
            \item {\bf Case $\lambda=(m,t,0,0)$.} Starting from $(m,m,m-t,0)$, we need to arrive to the shape $(m+1,m+1,m-1,m-1)$ in two steps: in the first step we need to add $m$ boxes labelled with $1$; in the second step we need to add $t$ boxes labelled with $2$. At each step we follow Pieri's rule and eventually the Littlewood--Richardson rule (see decomposition~(\ref{eqn:LR})) must be satisfied. This can be done in a unique way, i.e.
            \[
            \begin{tabular}{l}
            \begin{Young}
               & $\cdots$ &  &  & $\cdots$ &  &  & 1 \cr
               & $\cdots$ &  &  & $\cdots$ &  &  & 2 \cr
               & $\cdots$ &  & 1 & $\cdots$ & 1  \cr
           1 & $\cdots$ & 1 & 2 & $\cdots$ & 2  \cr 
            \end{Young}
            \end{tabular}
            \]

            \item {\bf Case $\lambda=(m,t,t,0)$.} Without loss of generality we can suppose that $t>m-t$. Then we need to arrive to the desired shape starting from $(m,t,t,0)$ and adding at each step: $m$ boxes labelled with $1$; $m-t$ boxes labelled with $2$; $m-t$ boxes labelled with $3$. As before, there is a unique way to do so, i.e.
            \[
            \begin{tabular}{l}
            \begin{Young}
               & $\cdots$ &  &  &  & $\cdots$ &  &  & 1 \cr
               & $\cdots$ &  &  & 1 & $\cdots$ & 1 & 1 & 2 \cr
               & $\cdots$ &  &  & 2 & $\cdots$ & 2  \cr
           1 & $\cdots$ & 1 & 3 & 3 & $\cdots$ & 3  \cr 
            \end{Young}
            \end{tabular}
            \]

            \item {\bf Case $\lambda=(m,t,s,0)$ with $m>t>s>0$.} As before, without loss of generality, we can suppose that $t>m-t$. Then we need to reach the shape $(m+1,m+1,m-1,m-1)$ starting from $(m,t,s,0)$ and adding at each step: $m$ boxes labelled with $1$; $m-s$ boxes labelled with $2$; $m-t$ boxes labelled with $3$. Let us start with some easy remarks. First of all, by Remark~\ref{remark:tutti 1} on the first row we must have one box labelled with $1$. By the Pieri and Littlewood--Richardson rules (and Remark~\ref{rmk:strictly increasing on colomns}) the second row must be filled with $m-t$ boxes labelled with $1$ and one last box labelled with $2$. On the third row, the Pieri and Littlewood--Richardson rules force all the boxes after the $t^{\operatorname{th}}$ to be labelled with $2$ (one needs to consider separately the case in which the last box is labelled with $3$: this choice forces the remaining of the labelling, which can be seen violating the Littlewood--Richardson rule). Similarly, the boxes before the $m-t-1$ boxes before the $t^{\operatorname{th}}$ box on the third row must be labelled with $1$. The only indeterminacy is then on the $t^{\operatorname{th}}$ box on the third row. Labelling this box with $1$ or $2$ fixes all the remaining labelling, and both outcomes are easily seen to be admissible, i.e.
             \[
            \begin{tabular}{l}
            \begin{Young}
              & $\cdots$ &  &  &  & $\cdots$ &  &  &  & $\cdots$ &  &  & 1 \cr
              & $\cdots$ &  &  &  & $\cdots$ &  &  & 1 & $\cdots$ & 1 & 1 & 2 \cr 
             & $\cdots$ &  &  & 1 & $\cdots$ & 1 & 1 & 2 & $\cdots$ & 2 \cr 
            1 & $\cdots$ & 1 & 2 & 2 & $\cdots$ & 2 & 3 & 3 & $\cdots$ & 3 \cr  
            \end{Young}
            \end{tabular}
            \mbox{ and }
            \begin{tabular}{l}
            \begin{Young}
              & $\cdots$ &  &  &  & $\cdots$ &  &  &  & $\cdots$ &  &  & 1 \cr
              & $\cdots$ &  &  &  & $\cdots$ &  &  & 1 & $\cdots$ & 1 & 1 & 2 \cr 
             & $\cdots$ &  &  & 1 & $\cdots$ & 1 & 2 & 2 & $\cdots$ & 2 \cr 
            1 & $\cdots$ & 1 & 1 & 2 & $\cdots$ & 2 & 3 & 3 & $\cdots$ & 3 \cr  
            \end{Young}
            \end{tabular}
            \]
            (To help reading the diagrams: in the last row of the left diagram we have the first $s-1$ boxes labelled with $1$, the next $t-s$ boxes labelled with $2$, and the last $m-t$ boxes labelled with $3$; in the last row of the right diagram we have the first $s$ boxes labelled with $1$, the next $t-s-1$ labelled with $2$, and the last $m-t$ boxes labelled with $3$.)
        \end{itemize}

        Since the case $\lambda=(m,m,s,0)$ is equivalent to the case $(m,t,0,0)$ by equation (\ref{eqn:end meno m}), this completes the proof.
    \end{proof}

\section{Proof of Theorem~\ref{theorem C}}\label{section:Evidences}

If we write $\lambda=(m,t,s,0)$, in this section we analyse Conjecture~\ref{conjecture:simple+ext 1} for small values of $t$ and $s$. 
The following statement resume the results of the next sections.
\begin{thm}\label{thm:garden of happiness}
    Keep the notation as above.
    \begin{enumerate}
        \item Conjecture~\ref{conjecture:simple+ext 1} holds for partitions of the form:
        \begin{itemize}
            \item $\lambda=(m,1,0,0)$, for any $m\geq1$;
            \item $\lambda=(m,1,1,0)$. for any $m\geq1$;
            \item $\lambda=(m,2,0,0)$, for any $m\geq2$;
            \item $\lambda=(3,2,1,0),(4,2,1,0)$;
            \item $\lambda=(m,2,2,0)$, for any $m\geq2$;
            \item $\lambda=(3,3,0,0),(4,3,0,0),(5,3,0,0)$;
            \item $\lambda=(4,3,1,0)$.
        \end{itemize}
        \item Item $(1)$ of Conjecture~\ref{conjecture:simple+ext 1} holds for partitions of the form $\lambda=(m,t,s,0)$ with $t,s\leq3$.   
    \end{enumerate}
    Moreover, if $\lambda=(m,t,s,0)$ is one of the partitions above, then the same conclusions hold for partitions of the form
    \[ (m+k,t+k,s+k,k)\quad\mbox{ and }\quad (m+k,m-s+k,m-t+k,k). \]
\end{thm}

\begin{remark}
    Let us explain the last part of the statement above. 

    By Example~\ref{example:duale} we know that $\Sigma_{m,t,s,0)}\mQ=\Sigma_{(m,m-s,m-t,0)}\mQ^\vee\otimes\mO_X(-m)$. Therefore if Conjecture~\ref{conjecture:simple+ext 1} holds for $(m,t,s,0)$, then it also holds for $(m,m-s,m-t,0)$, and vice versa.

    In particular, for example, Conjecture~\ref{conjecture:simple+ext 1} also holds for 
    \[ \lambda=(2,2,1,0), (3,3,1,0), (3,3,2,0), (4,3,2,0), (4,3,3,0), (5,3,3,0). \]

    The partitions in item (1) then should be interpreted as the "minimum set" of partitions for which the conjecture holds (and similarly for item (2)).
\end{remark}

Theorem~\ref{thm:garden of happiness} will be proved by a case-by-case analysis, each case being performed in a different section. In each case the strategy will be the same and it will follow the same lines outlined in Section~\ref{section:two worked out example}: we first find a set of partitions that "generate" the Littlewood--Richardson decomposition in a precise sense (that will be clear later); then we use Borel--Weil--Bott Theorem and the Koszul resolution to compute the cohomology of each factor. 

Before diving into the case by case analysis, we collect in the next section some remarks and notations that will allow us to simplify the computations.

\subsection{Preliminary remarks}

In the next sections we want to highlight the recursive nature of the Littlewood--Richardson decomposition of $\mE nd(\Sigma_\lambda\mQ)$ for $\lambda=(m,t,s,0)$, as $m$ increases of value. 

First of all, let us introduce the notation
\[ \delta_1=(1,0,0,0). \]
Notice that $\delta_1$ is the first fundamental weight for $\SL(4)$. Recall also that if $\mu=(\mu_1,\mu_2,\mu_3,\mu_4\mid \mu_5,\mu_6)$, then 
\[ \mu+1=(\mu_1+1,\mu_2+1,\mu_3+1,\mu_4+1\mid \mu_5+1,\mu_6+1) \]
and $\Sigma_{\mu+1}=\Sigma_\mu$ (see Section~\ref{section:homogeneous}).

\begin{proposition}\label{prop:+ delta 1 non cambia}
    If $\Sigma_{\mu}$ appears as a  factor in the Littlewood--Richardson decomposition of $\mE nd(\Sigma_\lambda\mQ)$, then $\Sigma_{\mu+1}$ appears as a factor in the decomposition of $\mE nd(\Sigma_{\lambda+\delta_1}\mQ)$.

    In particular, we have an abstract inclusion
    \[ \mE nd(\Sigma_\lambda\mQ)\subset\mE nd(\Sigma_{\lambda+\delta_1}\mQ). \]
\end{proposition}
\begin{proof}
    First of all, from Example~\ref{example:duale} it follows that 
    \[ \mE nd(\Sigma_\lambda\mQ)=\Sigma_{(m,t,s,0)}\mQ\otimes\Sigma_{(m,m-s,m-t,0)}\mQ\otimes\mO_X(-m). \]
    Therefore each factor $\Sigma_\mu$ in the Littlewood--Richardson decomposition of $\mE nd(\Sigma_\lambda\mQ)$ must be of the form
    $\mu=(\mu_1,\mu_2,\mu_3,\mu_4\mid m,m)$ such that $\Sigma_{(\mu_1,\mu_2,\mu_3,\mu_4)}\mQ$ is a factor in the Littlewood--Richardson decomposition of 
    \[ \Sigma_{(m,t,s,0)}\mQ\otimes\Sigma_{(m,m-s,m-t,0)}\mQ. \] 
    In particular there must be an admissible labelling $\mathcal{L}_\mu$ on the Young diagram of $(\mu_1,\mu_2,\mu_3,\mu_4)$, where necessarily we have added $\mu_1-m$ boxes labelled with $1$ on the first row (cf.\ Remark~\ref{remark:tutti 1}). Notice that since we must add a total of $m+(m-s)+(m-t)$ labelled boxes, we must have $\mu_1+\mu_2+\mu_3+\mu_4=4m$.

    It will be enough to prove that $\Sigma_{(\mu_1+1,\mu_2+1,\mu_3+1,\mu_4+1)}$ is a factor in the Littlewood--Richardson decomposition of 
    \[ \Sigma_{(m+1,t,s,0)}\mQ\otimes\Sigma_{(m+1,m-s+1,m-t+1,0)}\mQ. \]
    Let us recall that starting from $(m+1,t,s,0)$ we need to arrive to $(\mu_1+1,\mu_2+1,\mu_3+1,\mu_4+1)$ in three steps: at the first step we add $m+1$ boxes labelled with $1$; at the second step we add $m-s+1$ boxes labelled with $2$; at the hird step we add $m-t+1$ boxes labelled with $3$. Eventually, we need to add $3m+(t-s)+3$ labelled boxes, so exactly three more boxes than the case $(m,t,s,0)$ we started with. More precisely, we need to add exactly one more box labelled with $1$, one more box labelled with $2$ and one more box labelled with $3$.
    
    To this end, let us take the Young diagram associated to $(\mu_1,\mu_2,\mu_3,\mu_4)$, and let us add to it a column to the left, i.e.
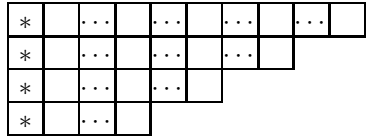
    \[  
    \begin{tabular}{l}
            \begin{Young}
            $\ast$ &  & $\cdots$  & & $\cdots$ & & $\cdots$ & & $\cdots$  & \cr
            $\ast$ &  & $\cdots$ & & $\cdots$ & & $\cdots$ &  \cr
            $\ast$ &  & $\cdots$ & & $\cdots$ &  \cr
            $\ast$ &  & $\cdots$ &  \cr
            \end{Young}
            \end{tabular}
    \]
    Now we need to label this new Young diagram. First of all, on the first row we must add $\mu_1+1-(m+1)=\mu_1-m$ boxes labelled with $1$ (cf.\ Remark~\ref{remark:tutti 1} again): this is the same row of the labelling $\mathcal{L}_\mu$ we started with. On the second row we must add $\mu_2+1-t$ labelled boxes: labelling the first added box with $1$ and keeping the remaining $\mu_2-t$ with the same labelling $\mathcal{L}_\mu$ violates no rules, hence it is admissible. Similarly for the third row, we must add $\mu_3+1-s$ labelled boxes: we label the first added box with $2$ and keep for the remaining $\mu_3-s$ boxes the labelling $\mathcal{L}_\mu$. Again this does not violate any rule. Finally, on the fourth row we must add $\mu_4+1$ labelled boxes and again we label the first added box with $3$ and keep the remaining $\mu_4$ boxes labelled as in $\mathcal{L}_\mu$. 

    At the end of this process we have produced a labelling $\mathcal{L}_{\mu+1}$ that is easily seen to be admissible. This finishes the proof.
\end{proof}

\begin{remark}
    In the proof above we have produced a labelling $\mathcal{L}_{\mu+1}$ starting from a labelling $\mathcal{L}_\mu$. It is easy to see that to different labelling $\mathcal{L}_\mu^{(1)},\mathcal{L}_\mu^{(2)}$ correspond different labelling $\mathcal{L}_{\mu+1}^{(1)},\mathcal{L}_{\mu+1}^{(2)}$. In particular the multiplicity of $\Sigma_{\mu+1}$ in $\mE nd(\Sigma_{(m+1,t,s,0)}\mQ)$ is at least equal to the multiplicity of $\Sigma_\mu$ in $\mE nd(\Sigma_{(m,t,s,0)}\mQ)$. On the other hand there can be more labelling on $\Sigma_{\mu+1}$ not coming from labelling of $\Sigma_\mu$, as we will see in several situations in the rest of the sections.
\end{remark}

As we have seen in equality (\ref{eqn: + n}) in Section~\ref{section:homogeneous}, there is a natural identification
\[ \Sigma_{\mu+1}=\Sigma_\mu. \]
Therefore Proposition~\ref{prop:+ delta 1 non cambia} implies that we can approach the study of $\mE nd(\Sigma_\lambda\mQ)$ in a inductive way. 
\medskip



{\bf Convention.} The following sections contain a case-by-case proof of Theorem~\ref{thm:garden of happiness}. In order to simplify the notations we will write parametric partitions such as, for example, $\mu=(2m,m,m-1,1\mid m,m)$. Since the values $m=0$ and $m=1$ do not produce partitions, we should specify the range $m\geq2$.

On the other hand, due to the elevated level of difficulty in the notations, we decided to abuse notation and tacitly ignore the values of the parametric $6$-tuples that do not produce a partition.

\newpage
\subsection{The case $\lambda=(m,1,0,0)$}\label{section:LR for (m,1,0,0)}
For any $m\geq 1$ let us put
\[ \mathsf{K}_{(m,1,0,0)}:=\left\{ 
\begin{array}{l}
(2m,m+i,m-i,0\mid m,m) \mbox{ for } i=0,1 \\ 
(2m,m,m-1,1\mid m,m) \\
(2m-1,m+1,m,0\mid m,m) \\
(2m-1,m,m,1\mid m,m)
\end{array}
\right\} \]
where, as by convention, a partition is considered empty if it fails to be non-decreasing. 

\begin{proposition}\label{prop:LR di End(m,1,0,0)}
    For every $m\geq1$, the following recursive Littlewood--Richardson decompositions hold
    \[ \mE nd(\Sigma_{(m,1,0,0)}\mQ)=\mE nd(\Sigma_{(m-1,1,0,0)}\mQ)\oplus\bigoplus_{\mu\in\mathsf{K}_{(m,1,0,0)}}\Sigma_\mu, \]
    where the term $\mE nd(\Sigma_{(m-1,1,0,0)}\mQ)$ does not appear if $m=1$.
\end{proposition}

\begin{example}\label{example:LR di End(m,1,0,0)}
    Let us explicitly write down the first two cases.
    \begin{itemize}
        \item ($m=1$) We have $\Sigma_{(1,1,0,0)}\mQ=\W^2\mQ$ and Proposition~\ref{prop:LR di End(m,1,0,0)} gives
    \[ \mE nd\left(\W^2\mQ\right)=\Sigma_{(2,2,0,0\mid 1,1)} \oplus \Sigma_{(2,1,1,0\mid 1,1)} \oplus \Sigma_{(1,1,1,1\mid 1,1)}, \]
    which can be checked by hand (or one can look at \cite{Fat24}). \\
    \item ($m=2$) Proposition~\ref{prop:LR di End(m,1,0,0)} gives (see also Section~\ref{section:Sigma 2 1})
    \begin{align*} 
    \mE nd(\Sigma_{(2,1,0,0)}\mQ)= & \Sigma_{(4,3,1,0\mid 2,2)} \oplus \Sigma_{(4,2,2,0\mid 2,2)} \oplus \Sigma_{(4,2,1,1\mid 2,2)} \oplus \Sigma_{(3,3,2,0\mid 2,2)} \oplus \\
    & \Sigma_{(3,2,2,1\mid 2,2)}^{\oplus 2} \oplus    \Sigma_{(3,3,1,1\mid 2,2)} \oplus \Sigma_{(2,2,2,2\mid 2,2)} .
    \end{align*}
    \end{itemize}
\end{example}

\begin{proof}
    We proceed by induction on $m$, the base of the induction being Example~\ref{example:LR di End(m,1,0,0)}. If $m\geq2$, then by Proposition~\ref{prop:+ delta 1 non cambia} there is an inclusion $\mE nd(\Sigma_{(m-1,1,0,0)}\mQ)\subset\mE nd(\Sigma_{(m,1,0,0)}\mQ)$ given by mapping each irreducible factor $\Sigma_\mu$ of $\mE nd(\Sigma_{(m-1,1,0,0)}\mQ)$ to the irreducible factor $\Sigma_{\mu+1}$ of $\mE nd(\Sigma_{(m,1,0,0)}\mQ)$. Notice that for $m\geq2$ the cardinality of $\mathsf{K}_m$ is constant.
    It is now a long and boring but straightforward computation (see for example the proof of Proposition~\ref{prop:W3 in Ext1}) to check that 
    \begin{equation}\label{eqn:inclusione m 1 0 0 } 
    \mE nd(\Sigma_{(m,1,0,0)}\mQ)\supset\mE nd(\Sigma_{(m-1,1,0,0)}\mQ)\oplus\bigoplus_{\mu\in\mathsf{K}_{(m,1,0,0)}}\Sigma_\mu. 
    \end{equation}
    On the other hand one can check that the following equality holds
    \[ r^2_{(m,1,0,0)}-r^2_{(m-1,1,0,0)}=\frac{(m+2)^2}{4}(6m^3+11m^2-1), \]
    where $r_\lambda$ is the rank function (\ref{eqn:rank}).

    Similarly one can check that also 
    \[ \sum_{\mu\in\mathsf{K}_{(m,1,0,0)}}r_\mu=\frac{(m+2)^2}{4}(6m^3+11m^2-1), \]
    from which it follows that the inclusion (\ref{eqn:inclusione m 1 0 0 }) is an equality, hence the claim.
\end{proof}

\begin{remark}
    Notice that the summand that contributes to the groups  $\Ext^1(\Sigma_{(m,1,0,0)}\mQ,\Sigma_{(m,1,0,0)}\mQ)$ and $\Ext^3(\Sigma_{(m,1,0,0)}\mQ,\Sigma_{(m,1,0,0)}\mQ)$ is $\Sigma_{(2m,m+1,m-1,0\mid m,m)}$ for $m=1$ (and then it propagates for any $m$ thanks to Proposition~\ref{prop:+ delta 1 non cambia}). Similarly, the summand corresponding to $\mO_X$ is $\Sigma_{(2m-1,m,m,1\mid m,m)}$ for $m=1$.
\end{remark}

In order to prove Theorem~\ref{thm:garden of happiness} in this case, it is enough to compute the $0$-th and $1$-st cohomology groups of the vector bundles $\Sigma_\mu$ for every $\mu\in\mathsf{K}_{(m,1,0,0)}$.

\begin{proposition}
    For any $m\geq1$ and any $\mu\in\mathsf{K}_{(m,1,0,0)}$, the following hold:
    \begin{enumerate}
        \item $\oH^0(X, \Sigma_\mu)=0$ unless $\mu=(1,1,1,1\mid 1,1)$, in which case $\oH^0(X, \Sigma_{(1,1,1,1\mid 1,1)})=\CC$;
        \item $\oH^1(X, \Sigma_\mu)=0$ unless $\mu=(2,2,0,0\mid 1,1)$, in which case $\oH^1(X, \Sigma_{(2,2,0,0\mid 1,1)})=\W^3 V_6^\vee$.
    \end{enumerate}
\end{proposition}
\begin{proof}
    This is a long but easy computation.
\end{proof}

\begin{corollary}
    For every $m\geq1$ we have
    \[ \Hom(\Sigma_{(m,1,0,0)}\mQ,\Sigma_{(m,1,0,0)}\mQ)=\CC\quad\mbox{ and }\quad \Ext^1(\Sigma_{(m,1,0,0)}\mQ,\Sigma_{(m,1,0,0)}\mQ)=\W^3 V_6^\vee. \]
\end{corollary}

\begin{remark}
    Notice that the factor contributing to $\Ext^1(\Sigma_{(m,1,0,0)}\mQ,\Sigma_{(m,1,0,0)}\mQ)$ appears already for $m=1$ (coinciding with the one in Proposition~\ref{prop:W3 in Ext1}) and then it propagates to the rest of the sequence thanks to Proposition~\ref{prop:+ delta 1 non cambia}.
\end{remark}

\newpage
\subsection{The case $\lambda=(m,1,1,0)$}\label{section:LR for (m,1,1,0)}

For $m\geq1$, let us put
\[ \mathsf{K}_{(m,1,1,0)}:=\left\{ 
\begin{array}{l} 
(2m,m,m,0\mid m,m) \\
(2m,m,m-1,1\mid m,m) \\
(2m-1,m+1,m,0\mid m,m) \\
(2m-1,m,m,1\mid m,m) \\
(2m-1,m+1,m-1,1\mid m,m)
\end{array}
\right\}.  \]

\begin{proposition}\label{prop:LR di End(m,1,1,0)}
    For every $m\geq1$, the following recursive Littlewood--Richardson decompositions hold
    \[ \mE nd(\Sigma_{(m,1,1,0)}\mQ)=\mE nd(\Sigma_{(m-1,1,0,0)}\mQ)\oplus\bigoplus_{\mu\in\mathsf{K}_{(m,1,1,0)}}\Sigma_\mu, \]
    where the term $\mE nd(\Sigma_{(m-1,1,1,0)}\mQ)$ does not appear if $m=1$.
\end{proposition}

\begin{example}\label{example: LR di End (m,1,1,0)}
    Let us again write down the first two cases.
    \begin{itemize}
        \item ($m=1$) By Example~\ref{example:duale} we have that $\Sigma_{(1,1,1,0)}\mQ=\mQ^\vee(-1)$ so that $\mE nd(\Sigma_{(1,1,1,0)}\mQ)=\mE nd(\mQ)$ was already computed in Section~\ref{section:Sym Q} (see equality (\ref{eqn:LR for Sym})). In this case the bundle is simple and rigid (see Proposition~\ref{prop:sym simple and rigid}) and Proposition~\ref{prop:LR di End(m,1,1,0)} says that
    \[ \mE nd(\Sigma_{(1,1,1,0))}\mQ)=\mathsf{K}_{(1,1,1,0)}=\Sigma_{(2,1,1,0\mid 1,1)}\oplus\Sigma_{(1,1,1,1\mid 1,1)}. \]
    \item ($m=2$) Proposition~\ref{prop:LR di End(m,1,1,0)} gives
    \begin{align*} 
    \mE nd(\Sigma_{(2,1,1,0))}\mQ)= & \Sigma_{(4,2,2,0\mid 2,2)} \oplus \Sigma_{(4,2,1,1\mid 2,2)} \oplus \Sigma_{(3,3,2,0\mid 2,2)} \oplus \\
    & \Sigma_{(3,2,2,1\mid 2,2)}^{\oplus 2} \oplus \Sigma_{(3,3,1,1\mid 2,2)} \oplus \Sigma_{(2,2,2,2\mid 2,2)}. 
    \end{align*}
    (Notice that the multiplicity of the partition $(3,2,2,1\mid 2,2)$ is $2$ since it appears both as a member of $\mathsf{K}_{(2,1,1,0)}$ and as $(2,1,1,0\mid 1,1)+1$, with $(2,1,1,0\mid 1,1)\in\mathsf{K}_{(1,1,1,0)}$.)
    \end{itemize}
\end{example}

\begin{proof}
    As before, we proceed by induction on $m$, the base of the induction being Example~\ref{example: LR di End (m,1,1,0)}. Notice that for $m\geq2$ the cardinality of $\mathsf{K}_{(m,1,1,0)}$ is constant. Thanks to Proposition~\ref{prop:+ delta 1 non cambia} and a direct check, one can see that
    \begin{equation}\label{eqn:inclusione m 1 1 0 } 
    \mE nd(\Sigma_{(m,1,1,0)}\mQ)\supset\mE nd(\Sigma_{(m-1,1,1,0)}\mQ)\oplus\bigoplus_{\mu\in\mathsf{K}_{(m,1,1,0)}}\Sigma_\mu. 
    \end{equation}
    Now we see that the following equality holds
    \[ r^2_{(m,1,1,0)}-r^2_{(m-1,1,1,0)}=\frac{m^2}{4}(3m+5)(2m^2+5m+1), \]
    where $r_\lambda$ is the rank function (\ref{eqn:rank}), and that similarly the following equality holds 
    \[ \sum_{\mu\in\mathsf{K}_{(m,1,1,0)}}r_\mu=\frac{m^2}{4}(3m+5)(2m^2+5m+1). \]
    Combining the two equality above, it follows that the inclusion (\ref{eqn:inclusione m 1 1 0 }) is an equality, hence the claim.
\end{proof}


\begin{remark}
    Notice that the summand that contributes to the groups  $\Ext^1(\Sigma_{(m,1,1,0)}\mQ,\Sigma_{(m,1,1,0)}\mQ)$ and $\Ext^3(\Sigma_{(m,1,1,0)}\mQ,\Sigma_{(m,1,1,0)}\mQ)$ is $\Sigma_{(2m-1,m+1,m-1,1\mid m,m)}$ for $m=2$ (and then it propagates for any $m$ thanks to Proposition~\ref{prop:+ delta 1 non cambia}) - if $m=1$ then $\Sigma_{(m,1,1,0)}\mQ$ is rigid. Similarly, the summand corresponding to $\mO_X$ is $\Sigma_{(2m-1,m,m,1\mid m,m)}$ for $m=1$.
\end{remark}

\begin{proposition}
    For every $m\geq1$ and every $\mu\in\mathsf{K}_{(m,1,1,0)}$ we have
    \begin{enumerate}
        \item $\oH^0(X, \Sigma_\mu)=0$ unless $\mu=(2,2,2,2\mid2,2)$, in which case $\oH^0(X, \Sigma_{(2,2,2,2\mid 2,2)})=\CC$;
        \item $\oH^1(X, \Sigma_\mu)=0$ unless $\mu=(3,3,1,1\mid 2,2)$, in which case $\oH^1(X, \Sigma_{(3,3,1,1\mid 2,2)})=\W^3 V_6^\vee$.
    \end{enumerate}
\end{proposition}
\begin{proof}
    This is a long but easy computation.
\end{proof}

\begin{corollary}
    For every $m\geq1$ we have
    \[ \Hom(\Sigma_{(m,1,1,0)}\mQ,\Sigma_{(m,1,1,0)}\mQ)=\CC\quad\mbox{ and }\quad \Ext^1(\Sigma_{(m,1,1,0)}\mQ,\Sigma_{(m,1,1,0)}\mQ)=\W^3 V_6^\vee. \]
\end{corollary}

\newpage
\subsection{The case $\lambda=(m,2,0,0)$}\label{section:LR for (m,2,0,0)}

For $m\geq2$, let us put
\[ \mathsf{K}_{(m,2,0,0)}:=\left\{ 
\begin{array}{l} 
(2m,m+i,m-i,0\mid m,m) \mbox{ for } i=0,1,2 \\
(2m,m+i,m-i-1,1\mid m,m) \mbox{ for } i=0,1\, (m\geq3) \\
(2m,m,m-2,2\mid m,m)\, (m\geq4) \\
(2m-1,m+1+i,m-i,0\mid m,m) \mbox{ for } i=0,1 \, (m\geq3) \\
(2m-1,m+i,m-i,1\mid m,m) \mbox{ for } i=0,1 \\
(2m-1,m,m-1,2\mid m,m)\, (m\geq3) \\
(2m-2,m+2-i,m,i\mid m,m) \mbox{ for } i=0,1,2 \, (m\geq 4-i) \\
\end{array}
\right\}.  \]

\begin{proposition}\label{prop:LR di End(m,2,0,0)}
    For every $m\geq1$, the following recursive Littlewood--Richardson decomposition holds,
    \[ \mE nd(\Sigma_{(m,2,0,0)}\mQ)=\mE nd(\Sigma_{(m-1,2,0,0)}\mQ)\oplus\bigoplus_{\mu\in\mathsf{K}_{(m,2,0,0)}}\Sigma_\mu. \]
\end{proposition}

\begin{example}\label{example: LR di End (m,2,0,0)}
    Let us write down the first two cases.
    \begin{itemize}
        \item ($m=2$) 
        \begin{align*} 
        \mE nd(\Sigma_{(2,2,0,0)}\mQ)=\mathsf{K}_{(2,2,0,0)}= &
       \Sigma_{(4,4,0,0\mid 2,2)} \oplus \Sigma_{(4,3,1,0\mid 2,2)} \oplus \Sigma_{(4,2,2,0\mid 2,2)} \oplus \\
        & \Sigma_{(3,3,1,1\mid 2,2)} \oplus 
        \Sigma_{(3,2,2,1\mid 2,2)} \oplus  \Sigma_{(2,2,2,2\mid 2,2)} . 
    \end{align*}
        \item ($m=3$) 
        \begin{align*}
        \mE nd(\Sigma_{(3,2,0,0)}\mQ)= & \Sigma_{(6,5,1,0\mid 3,3)} \oplus \Sigma_{(6,4,2,0\mid 3,3)} \oplus \Sigma_{(6,3,3,0\mid 3,3)} \oplus \Sigma_{(6,4,1,1\mid 3,3)} \oplus \Sigma_{(6,3,2,1\mid 3,3)} \oplus \\
        & \Sigma_{(5,5,2,0\mid 3,3)} \oplus \Sigma_{(5,4,3,0\mid 3,3)} \oplus \Sigma_{(5,4,2,1\mid 3,3)}^{\oplus 2} \oplus 
        \Sigma_{(5,3,3,1\mid 3,3)}^{\oplus2} \oplus \Sigma_{(5,3,2,2\mid 3,3)} \oplus \\
        & \Sigma_{(4,3,3,2\mid 3,3)}^{\oplus 2} \oplus \Sigma_{(4,4,3,1\mid 3,3)} \oplus \Sigma_{(5,5,1,1\mid 3,3)} \oplus \Sigma_{(4,4,2,2\mid 3,3)} \oplus \Sigma_{(3,3,3,3\mid 3,3)} .
        \end{align*}
    \end{itemize}
\end{example}


\begin{proof}
    We proceed by induction on $m$, the base of the induction being Example~\ref{example: LR di End (m,2,0,0)}, from which the claim follows for $m\leq3$. Let us then assume that $m\geq4$. Notice that the cardinality of $\mathsf{K}_{(m,2,0,0)}$ is constant for $m\geq4$. Thanks to Proposition~\ref{prop:+ delta 1 non cambia} and a direct check, one can see that
    \begin{equation}\label{eqn:inclusione m 2 0 0 } 
    \mE nd(\Sigma_{(m,2,0,0)}\mQ)\supset\mE nd(\Sigma_{(m-1,2,0,0)}\mQ)\oplus\bigoplus_{\mu\in\mathsf{K}_{(m,2,0,0)}}\Sigma_\mu. 
    \end{equation}
    Now we see that the following equality holds
    \[ r^2_{(m,2,0,0)}-r^2_{(m-1,2,0,0)}=(m+2)^2(6m^3+m^2-16m+5), \]
    where $r_\lambda$ is the rank function (\ref{eqn:rank}), and that similarly the following equality holds 
    \[ \sum_{\mu\in\mathsf{K}_{(m,2,0,0)}}r_\mu=(m+2)^2(6m^3+m^2-16m+5). \]
    Combining the two equality above, it follows that the inclusion (\ref{eqn:inclusione m 2 0 0 }) is an equality, hence the claim.
\end{proof}

\begin{remark}
    Notice that the summand that contributes to the groups  $\Ext^1(\Sigma_{(m,2,0,0)}\mQ,\Sigma_{(m,2,0,0)}\mQ)$ and $\Ext^3(\Sigma_{(m,2,0,0)}\mQ,\Sigma_{(m,2,0,0)}\mQ)$ is $\Sigma_{(2m-1,m+1,m-1,1\mid m,m)}$ for $m=2$ (and then it propagates for any $m$ thanks to Proposition~\ref{prop:+ delta 1 non cambia}). Similarly, the summand corresponding to $\mO_X$ is $\Sigma_{(2m-2,m,m,2\mid m,m)}$ for $m=2$.
\end{remark}

\begin{proposition}
    For every $m\geq2$ and every $\mu\in\mathsf{K}_{(m,2,0,0)}$ we have
    \begin{enumerate}
        \item $\oH^0(X, \Sigma_\mu)=0$ unless $\mu=(2,2,2,2\mid 2,2)$, in which case $\oH^0(X, \Sigma_{(2,2,2,2\mid 2,2)})=\CC$;
        \item $\oH^1(X, \Sigma_\mu)=0$ unless $\mu=(3,3,1,1\mid 2,2)$, in which case $\oH^1(X, \Sigma_{(3,3,1,1\mid 2,2)})=\W^3 V_6^\vee$.
    \end{enumerate}
\end{proposition}
\begin{proof}
    This is a long but easy computation.
\end{proof}

\begin{corollary}
    For every $m\geq2$ we have
    \[ \Hom(\Sigma_{(m,2,0,0)}\mQ,\Sigma_{(m,2,0,0)}\mQ)=\CC\quad\mbox{ and }\quad \Ext^1(\Sigma_{(m,2,0,0)}\mQ,\Sigma_{(m,2,0,0)}\mQ)=\W^3 V_6^\vee. \]
\end{corollary}

\newpage
\subsection{The case $\lambda=(m,2,1,0)$}\label{section:LR for (m,2,1,0)}

For $m\geq2$, let us put

\[ \mathsf{K}_{(m,2,1,0)}:=\left\{ 
\begin{array}{l}
(2m,m+1,m-1-i,i\mid m,m) \mbox{ for } i=0,1 \\
(2m,m,m,0\,|\,m,m) \\
(2m,m,m-1,1\,|\,m,m)^{\oplus2} \\ 
(2m,m,m-2,2\mid m,m)\, (m\geq4) \\
(2m,m-1,m-1,2\,|\,m,m) \\
(2m-1,m+2,m-1,0\mid m,m) \\
(2m-1,m+2,m-2,1\mid m,m) \\
(2m-1,m+1,m,0\mid m,m)^{\oplus2} \\
(2m-1,m+1,m-1,1\mid m,m)^{\oplus j}\, (j=1\,\mbox{if }m=2\mbox{ and }j=3\,\mbox{if }m\geq3) \\
(2m-1,m+1,m-2,2\mid m,m) \\
(2m-1,m,m,1\mid m,m)^{\oplus2} \\
(2m-1,m,m-1,2\mid m,m)^{\oplus2} \\
(2m-2,m+2,m,0\mid m,m) \\
(2m-2,m+2,m-1,1\mid m,m) \\
(2m-2,m+1,m+1,0\mid m,m) \\
(2m-2,m+1,m,1\mid m,m)^{\oplus2} \\
(2m-2,m+1,m-1,2\mid m,m) \\
(2m-2,m,m,2\mid m,m)
\end{array}
\right\}.  \]

\begin{proposition}\label{prop:LR di End(m,2,1,0)}
    For every $m\geq1$, the following recursive Littlewood--Richardson decomposition holds,
    \[ \mE nd(\Sigma_{(m,2,1,0)}\mQ)=\mE nd(\Sigma_{(m-1,2,1,0)}\mQ)\oplus\bigoplus_{\mu\in\mathsf{K}_{(m,2,1,0)}}\Sigma_\mu. \]
\end{proposition}

\begin{example}\label{example: LR di End (m,2,1,0)}
    Let us write down the first three cases.
    \begin{itemize}
        \item ($m=2$) By Example~\ref{example:duale} we have $\Sigma_{(2,2,1,0)}\mQ=\Sigma_{(2,1,0,0)}\mQ^\vee\otimes\mO_X(-2)$, therefore 
        \[ \mE nd(\Sigma_{(2,2,1,0)}\mQ)=\mE nd(\Sigma_{(2,1,0,0)}\mQ) \] 
        and the decomposition can be read in Section~\ref{section:Sigma 2 1} (see decomposition~(\ref{eqn:dec21})). \\
        \item ($m=3$) The case $\Sigma_{(3,2,1,0)}\mQ$ has been explicitly worked out in Section~\ref{section:Sigma 3 2 1} and the decomposition of the endomorphism bundle is the decomposition~(\ref{eqn:decomp 3210}).
        
    \end{itemize}
\end{example}


\begin{proof}
    We proceed by induction on $m$, the base of the induction being Example~\ref{example: LR di End (m,2,1,0)}, from which the claim follows for $m\leq3$. Let us then assume that $m\geq4$. Notice that the cardinality of $\mathsf{K}_{(m,2,1,0)}$ is constant for $m\geq4$. Thanks to Proposition~\ref{prop:+ delta 1 non cambia} and a direct check, one can see that
    \begin{equation}\label{eqn:inclusione m 2 1 0 } 
    \mE nd(\Sigma_{(m,2,1,0)}\mQ)\supset\mE nd(\Sigma_{(m-1,2,1,0)}\mQ)\oplus\bigoplus_{\mu\in\mathsf{K}_{(m,2,1,0)}}\Sigma_\mu. 
    \end{equation}
    Now we see that the following equality holds
    \[ r^2_{(m,2,1,0)}-r^2_{(m-1,2,1,0)}=\frac{16}{3}(2m+1)(m^2+m-3)(m^2+m-1), \]
    where $r_\lambda$ is the rank function (\ref{eqn:rank}), and that similarly the following equality holds 
    \[ \sum_{\mu\in\mathsf{K}_{(m,2,1,0)}}r_\mu=\frac{16}{3}(2m+1)(m^2+m-3)(m^2+m-1). \]
    Combining the two equality above, it follows that the inclusion (\ref{eqn:inclusione m 2 1 0 }) is an equality, hence the claim.
\end{proof}

\begin{remark}
    The summand that contributes to the groups  $\Ext^{1,3}(\Sigma_{(m,2,1,0)}\mQ,\Sigma_{(m,2,1,0)}\mQ)$ is more subtle here:
    if $m=2$ then it is $\Sigma_{(2m-1,m+1,m-1,1\mid m,m)}$, which appears with multiplicity $1$ in this case; for $m=3$ it is $\Sigma_{(2m-2,m+1,m-1,2\mid m,m)}$.     
    Notice that for $m\geq3$ the factor $\Sigma_{(2m-1,m+1,m-1,1\mid m,m)}$ (appearing with multiplicity 3) does not contribute to the $\Ext^{1,3}$-groups (a direct check). Therefore the two copies of $\W^3 V_6^\vee$ in $\Ext^1(\Sigma_{(m,2,1,0)}\mQ,\Sigma_{(m,2,1,0)}\mQ)$ for $m\geq3$ are the propagations of the two summands above.
    
    Similarly, the summand corresponding to $\mO_X$ is $\Sigma_{(2m-2,m,m,2\mid m,m)}$ for $m=2$.
\end{remark}

\begin{proposition}
    For every $m\geq2$ and every $\mu\in\mathsf{K}_{(m,2,1,0)}$ we have
    \begin{enumerate}
        \item $\oH^0(X, \Sigma_\mu)=0$ unless $\mu=(2,2,2,2\mid 2,2)$, in which case $\oH^0(X, \Sigma_{(2,2,2,2\mid 2,2)})=\CC$;
        \item For every $m\leq4$ we have that $\oH^1(X, \Sigma_\mu)=0$ unless $\mu=(3,3,1,1\mid 2,2)$, in which case $\oH^1(X, \Sigma_{(3,3,1,1\mid 2,2)})=\W^3 V_6^\vee$.
    \end{enumerate}
\end{proposition}
The reason why on item (2) there is the assumption $m\leq4$ is that as soon as $m\geq5$ the Koszul spectral sequence (\ref{eqn:Koszul}) does not degenerate and there may be a priori non-trivial first cohomology groups. This is the first time we encounter this problem -- we remind to Section~\ref{section:example of non-degeneracy} for a case where this is solved by hand, showing also the difficulty to approach this kind of problems in a systematic way.
\begin{proof}
    This is a long but easy computation.
\end{proof}

\begin{corollary}
    For every $m\geq2$ we have $\Hom(\Sigma_{(m,2,1,0)}\mQ,\Sigma_{(m,2,1,0)}\mQ)=\CC$.
    
    Moreover, $\Ext^1(\Sigma_{(2,2,1,0)}\mQ,\Sigma_{(2,2,1,0)}\mQ)=\W^3 V_6^\vee$ and $\Ext^1(\Sigma_{(m,2,1,0)}\mQ,\Sigma_{(m,2,1,0)}\mQ)=(\W^3 V_6^\vee)^{\oplus 2}$ for $m=3,4$.
\end{corollary}


\newpage
\subsection{The case $\lambda=(m,2,2,0)$}\label{section:LR for (m,2,2,0)}

For $m\geq2$, let us put
\[ \mathsf{K}_{(m,2,2,0)}:=\left\{ 
\begin{array}{l} 
(2m,m,m-i,i\,|\,m,m) \mbox{ for } i=0,1,2 \\
(2m-1,m+1,m-i,i\,|\,m,m) \mbox{ for } i=0,1,2\, (\mbox{if } i=1\mbox{ then } m\geq3) \\
(2m-1,m,m-i,i+1\,|\,m,m) \mbox{ for } i=0,1 \\
(2m-2,m+2,m-i,i\,|\,m,m) \mbox{ for } i=0,1,2 \\
(2m-2,m+1,m-i,i+1\,|\,m,m) \mbox{ for } i=0,1 \\
(2m-2,m,m,2\,|\,m,m)
\end{array}
\right\}.  \]


\begin{proposition}\label{prop:LR di End(m,2,2,0)}
    For every $m\geq2$, the following recursive Littlewood--Richardson decomposition holds,
    \[ \mE nd(\Sigma_{(m,2,2,0)}\mQ)=\mE nd(\Sigma_{(m-1,2,2,0)}\mQ)\oplus\bigoplus_{\mu\in\mathsf{K}_{(m,2,2,0)}}\Sigma_\mu. \]
\end{proposition}

\begin{example}\label{example: LR di End (m,2,2,0)}
    Let us write down the first two cases.
    \begin{itemize}
        \item ($m=2$) In this case $\Sigma_{(2,2,2,0)}\mQ=\Sym^2\mQ^\vee(-2)$ and the decomposition is
        \begin{align*} 
        \mE nd(\Sigma_{(2,2,0,0)}\mQ)= &
       \Sigma_{(4,2,2,0\mid 2,2)} \oplus         \Sigma_{(3,2,2,1\mid 2,2)} \oplus  \Sigma_{(2,2,2,2\mid 2,2)}, 
        \end{align*}
        which coincides with the decomposition~(\ref{eqn:LR for Sym}). \\
        \item ($m=3$) In this case $\Sigma_{(3,2,2,0)}\mQ=\Sigma_{(3,1,1,0)}\mQ^\vee(-3)$ and the decomposition agrees with the $m=3$ case of Proposition~\ref{prop:LR di End(m,1,1,0)}.
    \end{itemize}
\end{example}

\begin{proof}
    We proceed by induction on $m$, the base of the induction being Example~\ref{example: LR di End (m,2,2,0)}, from which the claim follows for $m\leq3$. Let us then assume that $m\geq4$. Notice that the cardinality of $\mathsf{K}_{(m,2,2,0)}$ is constant for $m\geq4$. Thanks to Proposition~\ref{prop:+ delta 1 non cambia} and a direct check, one can see that
    \begin{equation}\label{eqn:inclusione m 2 2 0 } 
    \mE nd(\Sigma_{(m,2,2,0)}\mQ)\supset\mE nd(\Sigma_{(m-1,2,2,0)}\mQ)\oplus\bigoplus_{\mu\in\mathsf{K}_{(m,2,2,0)}}\Sigma_\mu. 
    \end{equation}
    Now we see that the following equality holds
    \[ r^2_{(m,2,2,0)}-r^2_{(m-1,2,2,0)}=(m-1)^2(6m^3+17m^2-16), \]
    where $r_\lambda$ is the rank function (\ref{eqn:rank}), and that similarly the following equality holds 
    \[ \sum_{\mu\in\mathsf{K}_{(m,2,2,0)}}r_\mu=(m-1)^2(6m^3+17m^2-16). \]
    Combining the two equality above, it follows that the inclusion (\ref{eqn:inclusione m 2 2 0 }) is an equality, hence the claim.
\end{proof}

\begin{remark}
    The summand that contributes to the groups  $\Ext^{1,3}(\Sigma_{(m,2,2,0)}\mQ,\Sigma_{(m,2,2,0)}\mQ)$ is, for $m=3$, the Schur functor $\Sigma_{(2m-2,m+1,m-1,2\mid m,m)}$. 
    The summand corresponding to $\mO_X$ is $\Sigma_{(2m-2,m,m,2\mid m,m)}$ for $m=2$.
\end{remark}

\begin{proposition}
    For every $m\geq2$ and every $\mu\in\mathsf{K}_{(m,2,2,0)}$ we have
    \begin{enumerate}
        \item $\oH^0(X, \Sigma_\mu)=0$ unless $\mu=(2,2,2,2\mid 2,2)$, in which case $\oH^0(X, \Sigma_{(2,2,2,2\mid 2,2)})=\CC$;
        \item $\oH^1(X, \Sigma_\mu)=0$ unless $\mu=(4,4,2,2\mid 3,3)$, in which case $\oH^1(X, \Sigma_{(4,4,2,2\mid 3,3)})=\W^3 V_6^\vee$.
    \end{enumerate}
\end{proposition}
\begin{proof}
    This is a long but easy computation.
\end{proof}

\begin{corollary}
    For every $m\geq2$ we have $\Hom(\Sigma_{(m,2,2,0)}\mQ,\Sigma_{(m,2,2,0)}\mQ)=\CC$.
    Moreover, for every $m\geq3$ we have $\Ext^1(\Sigma_{(m,2,2,0)}\mQ,\Sigma_{(m,2,2,0)}\mQ)=\W^3 V_6^\vee$.
\end{corollary}
In the statement above we did not include the case $\Ext^1(\Sigma_{(2,2,2,0)}\mQ,\Sigma_{(2,2,2,0)}\mQ)=0$ since it already appeared in Proposition~\ref{prop:sym simple and rigid}.

\newpage
\subsection{The case $\lambda=(m,3,0,0)$}\label{section:LR for (m,3,0,0)}

For $m\geq3$, let us put

\[  
\mathsf{K}_{(m,3,0,0)}:=\left\{ 
\begin{array}{l}
(2m,m+i,m-i,0\mid m,m) \mbox{ for } i=0,1,2,3 \\
(2m,m+i,m-i-1,1\mid m,m) \mbox{ for } i=0,1,2\, (m\geq4) \\
(2m,m+i,m-2-i,2\mid m,m) \mbox{ for } i=0,1\,(m\geq5) \\
(2m,m,m-3,3\mid m,m)\, (m\geq6) \\
(2m-1,m+i+1,m-i,0\mid m,m) \mbox{ for } i=0,1,2\,(m\geq4) \\
(2m-1,m+i,m-i,1\mid m,m) \mbox{ for } i=0,1 \\
(2m-1,m+i,m-i-1,2\mid m,m) \mbox{ for } i=0,1\,(m\geq4) \\
(2m-1,m,m-2,3\mid m,m)\,(m\geq5) \\
(2m-2,m,m,2\mid m,m) \\
(2m-2,m+1,m,1\mid m,m)\,(m\geq4) \\
(2m-2,m+2,m,0\mid m,m)\,(m\geq5) \\
(2m-2,m,m-1,3\mid m,m)\,(m\geq4) \\
(2m-2,m+1,m-1,2\mid m,m) \\
(2m-2,m+2,m-1,1\mid m,m)\,(m\geq4) \\
(2m-2,m+3,m-1,0\mid m,m)\,(m\geq5) \\
(2m-3,m,m,3\mid m,m) \\
(2m-3,m+1,m,2\mid m,m)\,(m\geq4) \\
(2m-3,m+2,m,1\mid m,m)\,(m\geq5) \\
(2m-3,m+3,m,0\mid m,m)\,(m\geq6) \\
\end{array}
\right\}.
\]


\begin{proposition}\label{prop:LR di End(m,3,0,0)}
    For every $m\geq3$, the following recursive Littlewood--Richardson decomposition holds,
    \[ \mE nd(\Sigma_{(m,3,0,0)}\mQ)=\mE nd(\Sigma_{(m-1,3,0,0)}\mQ)\oplus\bigoplus_{\mu\in\mathsf{K}_{(m,3,0,0)}}\Sigma_\mu. \]
\end{proposition}

\begin{example}\label{example: LR di End (m,3,0,0)}
    Let us write down the first case, namely the case $m=3$. Explicitly we have
    \begin{align*}
        \mE nd(\Sigma_{(3,3,0,0)}\mQ)= & \Sigma_{(6,6,0,0\mid 3,3)}\oplus \Sigma_{(6,5,1,0\mid 3,3)} \oplus \Sigma_{(6,4,2,0\mid 3,3)} \oplus \Sigma_{(6,3,3,0\mid 3,3)} \oplus \Sigma_{(5,5,1,1\mid 3,3)}\oplus \\
        & \Sigma_{(5,4,2,1\mid 3,3)} \oplus \Sigma_{(5,3,3,1\mid 3,3)} \oplus \Sigma_{(4,4,2,2\mid 3,3)} \oplus \Sigma_{(4,3,3,2\mid 3,3)} \oplus \Sigma_{(3,3,3,3\mid 3,3)}
    \end{align*}
\end{example}

\begin{proof}
    We proceed by induction on $m$, the base of the induction being Example~\ref{example: LR di End (m,3,0,0)}. Let us then assume that $m\geq4$. Thanks to Proposition~\ref{prop:+ delta 1 non cambia} and a direct check, one can see that
    \begin{equation}\label{eqn:inclusione m 3 0 0 } 
    \mE nd(\Sigma_{(m,3,0,0)}\mQ)\supset\mE nd(\Sigma_{(m-1,3,0,0)}\mQ)\oplus\bigoplus_{\mu\in\mathsf{K}_{(m,3,0,0)}}\Sigma_\mu. 
    \end{equation}
    Now we see that the following equality holds
    \[ r^2_{(m,3,0,0)}-r^2_{(m-1,3,0,0)}=\frac{25}{3}(m-1)(m+2)^2(2m^2-m-9), \]
    where $r_\lambda$ is the rank function (\ref{eqn:rank}), and that similarly the following equality holds 
    \[ \sum_{\mu\in\mathsf{K}_{(m,3,0,0)}}r_\mu=\frac{25}{3}(m-1)(m+2)^2(2m^2-m-9). \]
    Combining the two equality above, it follows that the inclusion (\ref{eqn:inclusione m 3 0 0 }) is an equality, hence the claim.
\end{proof}

\begin{remark}
    The summand that contributes to the groups  $\Ext^{1,3}(\Sigma_{(m,3,0,0)}\mQ,\Sigma_{(m,3,0,0)}\mQ)$ is, for $m=3$, the Schur functor $\Sigma_{(2m-2,m+1,m-1,2\mid m,m)}$. 
    The summand corresponding to $\mO_X$ is $\Sigma_{(2m-3,m,m,3\mid m,m)}$ for $m=3$.
\end{remark}

\begin{proposition}
    For every $m\geq3$ and every $\mu\in\mathsf{K}_{(m,3,0,0)}$ we have
    \begin{enumerate}
        \item $\oH^0(X, \Sigma_\mu)=0$ unless $\mu=(3,3,3,3\mid 3,3)$, in which case $\oH^0(X, \Sigma_{(3,3,3,3\mid 3,3)})=\CC$;
        \item for every $m\leq4$ we have that $\oH^1(X, \Sigma_\mu)=0$ unless $\mu=(4,4,2,2\mid 3,3)$, in which case $\oH^1(X, \Sigma_{(4,4,2,2\mid 3,3)})=\W^3 V_6^\vee$.
    \end{enumerate}
\end{proposition}
\begin{proof}

    This is a long but easy computation.
\end{proof}

\begin{remark}
    Because of the size of $\mathsf{K}_{(m,3,0,0)}$, the proof of this proposition seems an incredibly long computation. On the other hand we wish to remark that many of the partitions appearing in $\mathsf{K}_{(m,3,0,0)}$ already appeared in the previous cases: there remains only 12 ``new" partitions to check.

    This remark is especially important in the next three sections, where the size of $\mathsf{K}_\lambda$ grows even more.
\end{remark}

\begin{corollary}
    For every $m\geq3$ we have $\Hom(\Sigma_{(m,3,0,0)}\mQ,\Sigma_{(m,3,0,0)}\mQ)=\CC$.
    
    Moreover, $\Ext^1(\Sigma_{(3,3,0,0)}\mQ,\Sigma_{(3,3,0,0)}\mQ)=\W^3 V_6^\vee$ and $\Ext^1(\Sigma_{(4,3,0,0)}\mQ,\Sigma_{(4,3,0,0)}\mQ)=\W^3 V_6^\vee$.
\end{corollary}

\newpage
\subsection{The case $\lambda=(m,3,1,0)$}\label{section:LR for (m,3,1,0)}
For $m\geq3$, let us put
\[ \mathsf{K}_{(m,3,1,0)}:=\left\{ 
\begin{array}{l}
  (2m,m+i,m-i,0\mid m,m) \mbox{ for } i=0,1,2 \\
  (2m,m,m-1,1\mid m,m)^{\oplus j} \mbox{ for } j=1,2 \,(m\geq 2+j) \\
  (2m, m+1, m-2,1\mid m,m)^{\oplus j} \mbox{ for } j=1,2 \, (m\geq 2+j) \\
  (2m,m+2,m-3,1\mid m,m) \, (m\geq4) \\
  (2m,m-1,m-1,2\mid m,m) \, (m\geq4) \\
  (2m,m,m-2,2\mid m,m)^{\oplus j} \mbox{ for } j=1,2 \, (m\geq 3+j) \\
  (2m,m+1,m-3,2\mid m,m) \, (m\geq 5) \\
  (2m,m-1,m-2,3\mid m,m) \, (m\geq5) \\
  (2m,m,m-3,3\mid m,m) \, (m\geq6) \\
  (2m-1,m+3,m-2,0\mid m,m)  \, (m\geq4) \\
  (2m-1,m+2,m-1,0\mid m,m)^{\oplus j} \mbox{ for } j=1,2 \, (m\geq 2+j) \\
  (2m-1,m+1,m,0\mid m,m)^{\oplus j} \mbox{ for } j=1,2 \, (m\geq 2+j) \\ 
  (2m-1,m+3,m-3,1 \mid m,m) \, (m\geq4) \\
  (2m-1,m+2,m-2,1 \mid m,m)^{\oplus j} \, (j=1\mbox{ if } m=3\mbox{ and } j=3\mbox{ if } m\geq4) \\
  (2m-1,m+1,m-1,1 \mid m,m)^{\oplus j} \, (j=2\mbox{ if } m=3\mbox{ and } j=4\mbox{ if } m\geq4) \\
  (2m-1,m,m,1\mid m,m)^{\oplus2} \\
  (2m-1,m,m-1,2 \mid m,m)^{\oplus j} \, (j=1\mbox{ if } m=3\mbox{ and } j=3\mbox{ if } m\geq4) \\
  (2m-1,m+1,m-2,2\mid m,m)^{\oplus j} \, (j=2\mbox{ if } m=4\mbox{ and } j=3\mbox{ if } m\geq5) \\
  (2m-1,m-1,m-1,3 \mid m,m) \, (m\geq4) \\
  (2m-1,m+2,m-3,2 \mid m,m) \, (m\geq5) \\
  (2m-1,m,m-2,3 \mid m,m)^{\oplus2} \, (m\geq5) \\
  (2m-1,m+1,m-3,3 \mid m,m) \, (m\geq6) \\
  (2m-2,m+1,m,1 \mid m,m)^{\oplus j}\, (j=1\mbox{ if } m=3\mbox{ and } j=3\mbox{ if } m\geq4) \\
  (2m-2,m+1,m-1,2 \mid m,m)^{\oplus j}\, (j=1\mbox{ if } m=3\mbox{ and } j=3\mbox{ if } m\geq4) \\
  (2m-2,m,m,2\mid m,m)^{\oplus2} \\
  (2m-2,m+2,m,0 \mid m,m)^{\oplus j}\, (j=1\mbox{ if } m=4\mbox{ and } j=2\mbox{ if } m\geq5) \\
  (2m-2,m+2,m-1,1 \mid m,m)^{\oplus j}\, (j=2\mbox{ if } m=4\mbox{ and } j=3\mbox{ if } m\geq5) \\
  (2m-2,m+2,m-2,2 \mid m,m) \, (m\geq4) \\
  (2m-2,m+1,m+1,0\mid m,m) \, (m\geq4) \\
  (2m-2,m,m-1,3 \mid m,m)^{\oplus2} \, (m\geq4) \\
  (2m-2,m+3,m-1,0\mid m,m) \, (m\geq5) \\
  (2m-2,m+3,m-2,1\mid m,m) \, (m\geq5) \\
  (2m-2,m+1,m-2,3\mid m,m) \, (m\geq5) \\
  (2m-3,m,m,3\mid m,m) \\
  (2m-3,m+1,m+1,1\mid m,m) \, (m\geq4) \\
  (2m-3,m+1,m,2\mid m,m)^{\oplus2}\, (m\geq4) \\
  (2m-3,m+1,m-1,3\mid m,m) \, (m\geq4) \\
  (2m-3,m+2,m+1,0\mid m,m) \, (m\geq5) \\
  (2m-3,m+2,m,1\mid m,m)^{\oplus2}\, (m\geq5) \\
  (2m-3,m+2,m-1,2\mid m,m) \, (m\geq5) \\
  (2m-3,m+3,m,0\mid m,m) \, (m\geq6) \\
  (2m-3,m+3,m-1,1\mid m,m) \, (m\geq6)
 \end{array}
\right\}.  \]


\begin{proposition}\label{prop:LR di End(m,3,1,0)}
    For every $m\geq3$, the following recursive Littlewood--Richardson decomposition holds,
    \[ \mE nd(\Sigma_{(m,3,1,0)}\mQ)=\mE nd(\Sigma_{(m-1,3,1,0)}\mQ)\oplus\bigoplus_{\mu\in\mathsf{K}_{(m,3,1,0)}}\Sigma_\mu. \]
\end{proposition}

\begin{example}\label{example: LR di End (m,3,1,0)}
    As usual, since $\Sigma_{(3,3,1,0)}\mQ=\Sigma_{(3,2,0,0)}\mQ^\vee(-3)$, let us write down the first new case, namely the case $m=4$. Explicitly we have
    \begin{align*}
        \mE nd(\Sigma_{(4,3,1,0)}\mQ)= & \Sigma_{(8,6,2,0\mid 4,4)} \oplus \Sigma_{(8,6,1,1\mid 4,4)} \oplus \Sigma_{(8,5,3,0\mid 4,4)} \oplus \Sigma_{(8,5,2,1\mid 4,4)}^{\oplus2} \oplus \Sigma_{(8,4,4,0\mid 4,4)} \oplus \\
        & \Sigma_{(8,4,3,1\mid 4,4)}^{\oplus2} \oplus \Sigma_{(8,4,2,2\mid 4,4)} \oplus \Sigma_{(8,3,3,2\mid 4,4)} \oplus \Sigma_{(7,7,2,0\mid 4,4)} \oplus \Sigma_{(7,7,1,1\mid 4,4)} \oplus \\
        & \Sigma_{(7,6,3,0\mid 4,4)}^{\oplus2} \oplus \Sigma_{(7,6,2,1\mid 4,4)}^{\oplus4} \oplus \Sigma_{(7,5,4,0\mid 4,4)}^{\oplus2} \oplus \Sigma_{(7,5,3,1\mid 4,4)}^{\oplus5} \oplus \Sigma_{(7,5,2,2\mid 4,4)}^{\oplus3} \oplus \\
        & \Sigma_{(7,4,4,1\mid 4,4)}^{\oplus3} \oplus \Sigma_{(7,4,3,2\mid 4,4)}^{\oplus4} \oplus \Sigma_{(7,3,3,3\mid 4,4)} \oplus \Sigma_{(6,6,4,0\mid 4,4)} \oplus \Sigma_{(6,6,3,1\mid 4,4)}^{\oplus3} \oplus \\
        & \Sigma_{(6,6,2,2\mid 4,4)}^{\oplus2} \oplus \Sigma_{(6,5,5,0\mid 4,4)} \oplus \Sigma_{(6,5,4,1\mid 4,4)}^{\oplus4} \oplus \Sigma_{(6,5,3,2\mid 4,4)}^{\oplus5} \oplus \Sigma_{(6,4,4,2\mid 4,4)}^{\oplus4} \oplus \\
        & \Sigma_{(6,4,3,3\mid 4,4)}^{\oplus3} \oplus \Sigma_{(5,5,5,1\mid 4,4)} \oplus \Sigma_{(5,5,4,2\mid 4,4)}^{\oplus3} \oplus \Sigma_{(5,5,3,3\mid 4,4)}^{\oplus2} \oplus \Sigma_{(5,4,4,3\mid 4,4)}^{\oplus3} \oplus \\
        & \Sigma_{(4,4,4,4\mid 4,4)}\,.
    \end{align*}
\end{example}

\begin{proof}
    We proceed by induction on $m$, the base of the induction being Example~\ref{example: LR di End (m,3,1,0)}. Let us then assume that $m\geq4$. Thanks to Proposition~\ref{prop:+ delta 1 non cambia} and a direct check, one can see that
    \begin{equation}\label{eqn:inclusione m 3 1 0 } 
    \mE nd(\Sigma_{(m,3,1,0)}\mQ)\supset\mE nd(\Sigma_{(m-1,3,1,0)}\mQ)\oplus\bigoplus_{\mu\in\mathsf{K}_{(m,3,1,0)}}\Sigma_\mu. 
    \end{equation}
    Now we see that the following equality holds
    \[ r^2_{(m,3,1,0)}-r^2_{(m-1,3,1,0)}=\frac{25}{4}(3m^2+m-6)(2m^3+m^2-11m-6), \]
    where $r_\lambda$ is the rank function (\ref{eqn:rank}), and that similarly the following equality holds 
    \[ \sum_{\mu\in\mathsf{K}_{(m,3,1,0)}}r_\mu=\frac{25}{4}(3m^2+m-6)(2m^3+m^2-11m-6). \]
    Combining the two equality above, it follows that the inclusion (\ref{eqn:inclusione m 3 1 0 }) is an equality, hence the claim.
\end{proof}

\begin{remark}
    The summand that contributes to the groups  $\Ext^{1,3}(\Sigma_{(m,3,1,0)}\mQ,\Sigma_{(m,3,1,0)}\mQ)$ is:
    if $m=3$ then it is $\Sigma_{(2m-2,m+1,m-1,2\mid m,m)}$, which appears with multiplicity $1$ in this case; if $m=4$ then it is $\Sigma_{(2m-3,m+1,m-1,3\mid m,m)}$.     
    The summand corresponding to $\mO_X$ is $\Sigma_{(2m-3,m,m,3\mid m,m)}$ for $m=3$.
\end{remark}

\begin{proposition}
    For every $m\geq3$ and every $\mu\in\mathsf{K}_{(m,3,1,0)}$ we have
    \begin{enumerate}
        \item $\oH^0(X, \Sigma_\mu)=0$ unless $\mu=(3,3,3,3\mid 3,3)$, in which case $\oH^0(X, \Sigma_{(3,3,3,3\mid 3,3)})=\CC$;
        \item for every $m\leq4$ we have that $\oH^1(X, \Sigma_\mu)=0$ unless $\mu=(4,4,2,2\mid 3,3)$, in which case $\oH^1(X, \Sigma_{(4,4,2,2\mid 3,3)})=\W^3 V_6^\vee$.
    \end{enumerate}
\end{proposition}
\begin{proof}
    This is a long but easy computation.
\end{proof}

\begin{corollary}
    For every $m\geq3$ we have $\Hom(\Sigma_{(m,3,1,0)}\mQ,\Sigma_{(m,3,1,0)}\mQ)=\CC$.
    
    Moreover, $\Ext^1(\Sigma_{(3,3,1,0)}\mQ,\Sigma_{(3,3,1,0)}\mQ)=\W^3 V_6^\vee$ and $\Ext^1(\Sigma_{(4,3,1,0)}\mQ,\Sigma_{(4,3,1,0)}\mQ)=(\W^3 V_6^\vee)^{\oplus 2}$.
\end{corollary}

\newpage
\subsection{The case $\lambda=(m,3,2,0)$}\label{section:LR for (m,3,2,0)}

For $m\geq3$, let us put
\[ \mathsf{K}_{(m,3,2,0)}:=\left\{ 
\begin{array}{l} 
(2m,m+1,m-1,0\mid m,m) \\
(2m,m,m,0\mid m,m) \\
(2m,m,m-1,1\mid m,m)^{\oplus j}\, (j=1\mbox{ if } m=3\mbox{ and } j=2\mbox{ if } m\geq4) \\
(2m,m+1,m-2,1\mid m,m)\, (m\geq4) \\
(2m,m,m-2,2\mid m,m)^{\oplus j}\, (j=1\mbox{ if } m=4\mbox{ and } j=2\mbox{ if } m\geq5) \\
(2m,m-1,m-1,2\mid m,m) \, (m\geq4) \\ 
(2m,m+1,m-3,2\mid m,m) \, (m\geq5) \\
(2m,m-1,m-2,3\mid m,m) \, (m\geq5) \\
(2m,m,m-3,3\mid m,m) \, (m\geq6) \\
(2m-1,m+1,m,0\mid m,m)^{\oplus j}\, (j=1\mbox{ if } m=3\mbox{ and } j=2\mbox{ if } m\geq4) \\
(2m-1,m+1,m-1,1\mid m,m)^{\oplus j}\, (j=1\mbox{ if } m=3\mbox{ and } j=3\mbox{ if } m\geq4) \\
(2m-1,m,m,1\mid m,m)^{\oplus2} \\
(2m-1,m,m-1,2\mid m,m)^{\oplus j}\, (j=1\mbox{ if } m=3\mbox{ and } j=3\mbox{ if } m\geq4) \\
(2m-1,m+2,m-1,0\mid m,m) \, (m\geq4) \\
(2m-1,m+2,m-2,1\mid m,m) \, (m\geq4) \\
(2m-1,m+1,m-2,2\mid m,m)^{\oplus j}\, (j=2\mbox{ if } m=4\mbox{ and } j=3\mbox{ if } m\geq5) \\
(2m-1,m-1,m-1,3\mid m,m) \, (m\geq4) \\
(2m-1,m+2,m-3,2\mid m,m) \, (m\geq5) \\
(2m-1,m,m-2,3\mid m,m)^{\oplus2}\, (m\geq5) \\
(2m-1,m+1,m-3,3\mid m,m)\, (m\geq6) \\
(2m-2,m+1,m,1\mid m,m)^{\oplus j}\, (j=1\mbox{ if } m=3\mbox{ and } j=3\mbox{ if } m\geq4) \\
(2m-2,m+1,m-1,2\mid m,m)^{\oplus j}\, (j=1\mbox{ if } m=3\mbox{ and } j=4\mbox{ if } m\geq4) \\
(2m-2,m,m,2\mid m,m)^{\oplus2} \\
(2m-2,m+2,m,0\mid m,m)^{\oplus j}\, (j=1\mbox{ if } m=4\mbox{ and } j=2\mbox{ if } m\geq5) \\
(2m-2,m+2,m-1,1\mid m,m)^{\oplus j}\, (j=2 \mbox{ if } m=4\mbox{ and } j=3\mbox{ if } m\geq5) \\
(2m-2,m+2,m-2,2\mid m,m)^{\oplus j}\, (j=2 \mbox{ if } m=4\mbox{ and } j=3\mbox{ if } m\geq5) \\
(2m-2,m+1,m+1,0\mid m,m)\, (m\geq4) \\
(2m-2,m,m-1,3\mid m,m)^{\oplus2} \\
(2m-2,m+3,m-1,0\mid m,m) \, (m\geq5) \\
(2m-2,m+3,m-2,1\mid m,m) \, (m\geq5) \\
(2m-2,m+3,m-3,2\mid m,m) \, (m\geq5) \\
(2m-2,m+1,m-2,3\mid m,m)^{\oplus2} \, (m\geq5) \\
(2m-2,m+2,m-3,3\mid m,m) \, (m\geq6) \\
(2m-3,m,m,3\mid m,m) \\
(2m-3,m+1,m+1,1\mid m,m) \, (m\geq4) \\
(2m-3,m+1,m,2\mid m,m)^{\oplus2} \, (m\geq4) \\
(2m-3,m+1,m-1,3\mid m,m) \, (m\geq4) \\
(2m-3,m+2,m+1,0\mid m,m) \, (m\geq5) \\
(2m-3,m+2,m,1\mid m,m)^{\oplus2} \, (m\geq5) \\
(2m-3,m+2,m-1,2\mid m,m)^{\oplus2} \, (m\geq5) \\
(2m-3,m+2,m-2,3\mid m,m) \, (m\geq5) \\
(2m-3,m+3,m-2,2\mid m,m) \, (m\geq6) \\
(2m-3,m+3,m-1,1\mid m,m) \, (m\geq6) \\
(2m-3,m+3,m,0\mid m,m) \, (m\geq6)
\end{array}
\right\}. \]


\begin{proposition}\label{prop:LR di End(m,3,2,0)}
    For every $m\geq3$, the following recursive Littlewood--Richardson decomposition holds
    \[ \mE nd(\Sigma_{(m,3,2,0)}\mQ)=\mE nd(\Sigma_{(m-1,3,2,0)}\mQ)\oplus\bigoplus_{\mu\in\mathsf{K}_{(m,3,2,0)}}\Sigma_\mu. \]
\end{proposition}

\begin{example}\label{example: LR di End (m,3,2,0)}
    As usual, since $\Sigma_{(3,3,2,0)}\mQ=\Sigma_{(3,1,0,0)}\mQ^\vee(-3)$ and $\Sigma_{(4,3,2,0)}\mQ=\Sigma_{(4,2,1,0)}\mQ^\vee(-4)$, let us write down the first new case, namely the case $m=5$. Explicitly we have
    \begin{align*}
        \mE nd(\Sigma_{(5,3,2,0)}\mQ)= & \Sigma_{(10,6,4,0\mid 5,5)} \oplus \Sigma_{(10,6,3,1\mid 5,5)} \oplus \Sigma_{(10,6,2,2\mid 5,5)} \oplus \Sigma_{(10,5,5,0\mid 5,5)} \oplus \Sigma_{(10,5,4,1\mid 5,5)}^{\oplus2} \oplus \\
        & \Sigma_{(10,5,3,2\mid 5,5)}^{\oplus2} \oplus \Sigma_{(10,4,4,2\mid 5,5)} \oplus \Sigma_{(10,4,3,3\mid 5,5)} \oplus \Sigma_{(9,7,4,0\mid 5,5)} \oplus \Sigma_{(9,7,3,1\mid 5,5)} \oplus \\
        & \Sigma_{(9,7,2,2\mid 5,5)} \oplus \Sigma_{(9,6,5,0\mid 5,5)}^{\oplus2} \oplus \Sigma_{(9,6,4,1\mid 5,5)}^{\oplus4} \oplus \Sigma_{(9,6,3,2\mid 5,5)}^{\oplus4} \oplus \Sigma_{(9,5,5,1\mid 5,5)}^{\oplus3} \oplus \\
        & \Sigma_{(9,5,4,2\mid 5,5)}^{\oplus5} \oplus \Sigma_{(9,5,3,3\mid 5,5)}^{\oplus3} \oplus \Sigma_{(9,4,4,3\mid 5,5)}^{\oplus2} \oplus \Sigma_{(8,8,4,0\mid 5,5)} \oplus \Sigma_{(8,8,3,1\mid 5,5)} \oplus \\
        & \Sigma_{(8,8,2,2\mid 5,5)} \oplus \Sigma_{(8,7,5,0\mid 5,5)}^{\oplus2} \oplus \Sigma_{(8,7,4,1\mid 5,5)}^{\oplus4} \oplus \Sigma_{(8,7,3,2\mid 5,5)}^{\oplus4} \oplus \Sigma_{(8,6,6,0\mid 5,5)} \oplus \\
        & \Sigma_{(8,6,5,1\mid 5,5)}^{\oplus5} \oplus \Sigma_{(8,6,4,2\mid 5,5)}^{\oplus8} \oplus \Sigma_{(8,6,3,3\mid 5,5)}^{\oplus4} \oplus \Sigma_{(8,5,5,2\mid 5,5)}^{\oplus5} \oplus \Sigma_{(8,5,4,3\mid 5,5)}^{\oplus6} \oplus \\
        & \Sigma_{(8,4,4,4\mid 5,5)} \oplus \Sigma_{(7,7,6,0\mid 5,5)} \oplus \Sigma_{(7,7,5,1\mid 5,5)}^{\oplus3} \oplus \Sigma_{(7,7,4,2\mid 5,5)}^{\oplus4} \oplus \Sigma_{(7,7,3,3\mid 5,5)}^{\oplus2} \oplus \\
        & \Sigma_{(7,6,6,1\mid 5,5)}^{\oplus2} \oplus \Sigma_{(7,6,5,2\mid 5,5)}^{\oplus6} \oplus \Sigma_{(7,6,4,3\mid 5,5)}^{\oplus6} \oplus \Sigma_{(7,5,5,3\mid 5,5)}^{\oplus5} \oplus \Sigma_{(7,5,4,4\mid 5,5)}^{\oplus3} \oplus \\
        & \Sigma_{(6,6,6,2\mid 5,5)} \oplus \Sigma_{(6,6,5,3\mid 5,5)}^{\oplus3} \oplus \Sigma_{(6,6,4,4\mid 5,5)}^{\oplus2} \oplus \Sigma_{(6,5,5,4\mid 5,5)}^{\oplus3} \oplus \Sigma_{(5,5,5,5\mid 5,5)}
    \end{align*}
\end{example}

\begin{proof}
    We proceed by induction on $m$, the base of the induction being Example~\ref{example: LR di End (m,3,2,0)}. Let us then assume that $m\geq5$. Thanks to Proposition~\ref{prop:+ delta 1 non cambia} and a direct check, one can see that
    \begin{equation}\label{eqn:inclusione m 3 2 0 } 
    \mE nd(\Sigma_{(m,3,2,0)}\mQ)\supset\mE nd(\Sigma_{(m-1,3,2,0)}\mQ)\oplus\bigoplus_{\mu\in\mathsf{K}_{(m,3,2,0)}}\Sigma_\mu. 
    \end{equation}
    Now we see that the following equality holds
    \[ r^2_{(m,3,2,0)}-r^2_{(m-1,3,2,0)}=\frac{25}{4}(3m^2-m-6)(2m^3-m^2-11m+6), \]
    where $r_\lambda$ is the rank function (\ref{eqn:rank}), and that similarly the following equality holds 
    \[ \sum_{\mu\in\mathsf{K}_{(m,3,2,0)}}r_\mu=\frac{25}{4}(3m^2-m-6)(2m^3-m^2-11m+6). \]
    Combining the two equality above, it follows that the inclusion (\ref{eqn:inclusione m 3 2 0 }) is an equality, hence the claim.
\end{proof}

\begin{remark}
    The summand that contributes to the groups  $\Ext^{1,3}(\Sigma_{(m,3,2,0)}\mQ,\Sigma_{(m,3,2,0)}\mQ)$ is:
    if $m=3$ then it is $\Sigma_{(2m-2,m+1,m-1,2\mid m,m)}$, which appears with multiplicity $1$ in this case; if $m=4$ then it is $\Sigma_{(2m-3,m+1,m-1,3\mid m,m)}$.     
    The summand corresponding to $\mO_X$ is $\Sigma_{(2m-3,m,m,3\mid m,m)}$ for $m=3$.
\end{remark}

\begin{proposition}
    For every $m\geq3$ and every $\mu\in\mathsf{K}_{(m,3,2,0)}$ we have
     $\oH^0(X, \Sigma_\mu)=0$ unless $\mu=(3,3,3,3\mid 3,3)$, in which case $\oH^0(X, \Sigma_{(3,3,3,3\mid 3,3)})=\CC$.
\end{proposition}
\begin{proof}
    This is a long but easy computation.
\end{proof}

\begin{corollary}
    For every $m\geq3$ we have $\Hom(\Sigma_{(m,3,2,0)}\mQ,\Sigma_{(m,3,2,0)}\mQ)=\CC$.
\end{corollary}

\newpage
\subsection{The case $\lambda=(m,3,3,0)$}\label{section:LR for (m,3,3,0)}

For $m\geq3$, let us put
\[ \mathsf{K}_{(m,3,3,0)}:=\left\{ 
\begin{array}{l} 
(2m,m,m-i,i\mid m,m) \mbox{ for } i=0,1,2,3 \, (m\geq 3+i)\\
(2m-1,m+1,m,0\mid m,m) \,
(m\geq4) \\
 (2m-1,m+1,m-i,i\mid m,m) \mbox{ for } i=1,2,3 \, (m\geq 3+i) \\
 (2m-1,m,m+1-i,i\mid m,m) \mbox{ for } i=1,2,3 \, (m\geq 2+i) \\
 (2m-2,m+2,m-i,i\mid m,m) \mbox{ for } i=0,1,2 \, (m\geq5) \\
 (2m-2,m+2,m-3,3\mid m,m) \, (m\geq6) \\
 
 (2m-2,m+1,m,1\mid m,m) \, (m\geq4) \\
 (2m-2,m+1,m-1-i,i+2\mid m,m) \mbox{ for } i=0,1 \, (m\geq 4+i) \\
 (2m-2,m,m-i,i+2\mid m,m) \mbox{ for } i=0,1 \, (m\geq 3+i) \\
 (2m-3,m+3,m-i,i\mid m,m) \mbox{ for } i=0,1,2,3 \, (m\geq6) \\
 (2m-3,m+2,m-i,i+1\mid m,m) \mbox{ for } i=0,1,2 \, (m\geq5) \\
 (2m-3,m+1,m-i,i+2\mid m,m) \mbox{ for } i=0,1 \, (m\geq4) \\
 (2m-3,m,m,3\mid m,m) 
 \end{array}
\right\}.  \]


\begin{proposition}\label{prop:LR di End(m,3,3,0)}
    For every $m\geq3$, the following recursive Littlewood--Richardson decomposition holds
    \[ \mE nd(\Sigma_{(m,3,3,0)}\mQ)=\mE nd(\Sigma_{(m-1,3,3,0)}\mQ)\oplus\bigoplus_{\mu\in\mathsf{K}_{(m,3,3,0)}}\Sigma_\mu. \]
\end{proposition}

\begin{example}\label{example: LR di End (m,3,3,0)}
    As usual, since $\Sigma_{(3,3,3,0)}\mQ=\Sym^3\mQ^\vee(-3)$, $\Sigma_{(4,3,3,0)}\mQ=\Sigma_{(4,1,1,0)}\mQ^\vee(-4)$ and $\Sigma_{(5,3,3,0)}\mQ=\Sigma_{(5,2,2,0)}\mQ^\vee(-5)$, let us write down the first new case, namely the case $m=6$. Explicitly we have
    \begin{align*}
        \mE nd(\Sigma_{(6,3,3,0)}\mQ)= & \Sigma_{(12,6,6,0 \mid 6,6)} \oplus \Sigma_{(12,6,5,1 \mid 6,6)} \oplus \Sigma_{(12,6,4,2 \mid 6,6)} \oplus \Sigma_{(12,6,3,3 \mid 6,6)} \oplus \Sigma_{(11,7,6,0 \mid 6,6)} \oplus \\
        & \Sigma_{(11,7,5,1 \mid 6,6)} \oplus \Sigma_{(11,7,4,2 \mid 6,6)} \oplus \Sigma_{(11,7,3,3 \mid 6,6)} \oplus \Sigma_{(11,6,6,1 \mid 6,6)}^{\oplus2} \oplus \Sigma_{(11,6,5,2 \mid 6,6)}^{\oplus2} \oplus \\
        & \Sigma_{(11,6,4,3 \mid 6,6)}^{\oplus2} \oplus \Sigma_{(10,8,6,0 \mid 6,6)} \oplus \Sigma_{(10,8,5,1 \mid 6,6)} \oplus \Sigma_{(10,8,4,2 \mid 6,6)} \oplus \Sigma_{(10,8,3,3 \mid 6,6)} \oplus \\
        & \Sigma_{(10,7,6,1 \mid 6,6)}^{\oplus2} \oplus \Sigma_{(10,7,5,2 \mid 6,6)}^{\oplus2} \oplus \Sigma_{(10,7,4,3 \mid 6,6)}^{\oplus2} \oplus \Sigma_{(10,6,6,2 \mid 6,6)}^{\oplus3} \oplus \Sigma_{(10,6,5,3 \mid 6,6)}^{\oplus3} \oplus \\
        & \Sigma_{(10,6,4,4 \mid 6,6)} \oplus \Sigma_{(9,9,6,0 \mid 6,6)} \oplus \Sigma_{(9,9,5,1 \mid 6,6)} \oplus \Sigma_{(9,9,4,2 \mid 6,6)} \oplus \Sigma_{(9,9,3,3 \mid 6,6)} \oplus \\
        & \Sigma_{(9,8,6,1 \mid 6,6)}^{\oplus2} \oplus \Sigma_{(9,8,5,2 \mid 6,6)}^{\oplus2} \oplus \Sigma_{(9,8,4,3 \mid 6,6)}^{\oplus2} \oplus \Sigma_{(9,7,6,2 \mid 6,6)}^{\oplus3} \oplus \Sigma_{(9,7,5,3 \mid 6,6)}^{\oplus3} \oplus \\
        & \Sigma_{(9,7,4,4 \mid 6,6)} \oplus \Sigma_{(9,6,6,3 \mid 6,6)}^{\oplus4} \oplus \Sigma_{(9,6,5,4 \mid 6,6)}^{\oplus2} \oplus \Sigma_{(8,8,6,2 \mid 6,6)} \oplus \Sigma_{(8,8,5,3 \mid 6,6)} \oplus \\
        & \Sigma_{(8,8,4,4 \mid 6,6)} \oplus \Sigma_{(8,7,6,3 \mid 6,6)}^{\oplus2} \oplus \Sigma_{(8,7,5,4 \mid 6,6)}^{\oplus2} \oplus \Sigma_{(8,6,6,4 \mid 6,6)}^{\oplus3} \oplus \Sigma_{(8,6,5,5 \mid 6,6)} \oplus \\
        & \Sigma_{(7,7,6,4 \mid 6,6)} \oplus \Sigma_{(7,7,5,5 \mid 6,6)} \oplus \Sigma_{(7,6,6,5 \mid 6,6)}^{\oplus2} \oplus \Sigma_{(6,6,6,6 \mid 6,6)}.
    \end{align*}
\end{example}

\begin{proof}
    We proceed by induction on $m$, the base of the induction being Example~\ref{example: LR di End (m,3,3,0)}. Let us then assume that $m\geq5$. Thanks to Proposition~\ref{prop:+ delta 1 non cambia} and a direct check, one can see that
    \begin{equation}\label{eqn:inclusione m 3 3 0 } 
    \mE nd(\Sigma_{(m,3,3,0)}\mQ)\supset\mE nd(\Sigma_{(m-1,3,3,0)}\mQ)\oplus\bigoplus_{\mu\in\mathsf{K}_{(m,3,3,0)}}\Sigma_\mu. 
    \end{equation}
    Now we see that the following equality holds
    \[ r^2_{(m,3,3,0)}-r^2_{(m-1,3,3,0)}=\frac{25}{3}(m-2)^2(m+1)(2m^2+m-9), \]
    where $r_\lambda$ is the rank function (\ref{eqn:rank}), and that similarly the following equality holds 
    \[ \sum_{\mu\in\mathsf{K}_{(m,3,3,0)}}r_\mu=\frac{25}{3}(m-2)^2(m+1)(2m^2+m-9). \]
    Combining the two equality above, it follows that the inclusion (\ref{eqn:inclusione m 3 3 0 }) is an equality, hence the claim.
\end{proof}

\begin{remark}
    The summand that contributes to the groups  $\Ext^{1,3}(\Sigma_{(m,3,3,0)}\mQ,\Sigma_{(m,3,3,0)}\mQ)$ is, for $m=43$, the Schur functor $\Sigma_{(2m-3,m+1,m-1,3\mid m,m)}$. 
    The summand corresponding to $\mO_X$ is $\Sigma_{(2m-3,m,m,3\mid m,m)}$ for $m=3$.
\end{remark}

\begin{proposition}
    For every $m\geq3$ and every $\mu\in\mathsf{K}_{(m,3,3,0)}$ we have
     $\oH^0(X, \Sigma_\mu)=0$ unless $\mu=(3,3,3,3\mid 3,3)$, in which case $\oH^0(X, \Sigma_{(3,3,3,3\mid 3,3)})=\CC$.
\end{proposition}
\begin{proof}
    This is a long but easy computation.
\end{proof}

\begin{corollary}
    For every $m\geq3$ we have $\Hom(\Sigma_{(m,3,3,0)}\mQ,\Sigma_{(m,3,3,0)}\mQ)=\CC$.
\end{corollary}


\subsection{An example of resolution of the indeterminacy}\label{section:example of non-degeneracy}


The examples analysed in the sections before have a common feature: for all of them the Koszul spectral sequence determines $\Hom(\Sigma_\lambda\mQ,\Sigma_\lambda\mQ)$ without indeterminacy. This is enough for item $(1)$ of Conjecture~\ref{conjecture:simple+ext 1}. On the other hand in many cases the Koszul spectral sequence has non-trivial differentials contributing to $\Ext^1(\Sigma_\lambda\mQ,\Sigma_\lambda\mQ)$. An example is the series $\lambda=(m,3,0,0)$, which has indeterminacy in degree $1$ from $m\geq5$. 

In this section we resolve by brute force this indeterminacy in the first case, namely for $\lambda=(5,3,0,0)$. This computation also shows the typical problems that one encounters in these cases.


\begin{lemma} \label{lem:521}
Conjecture \ref{conjecture:simple+ext 1} holds for $\Sigma_{(5,3,0,0)}\mQ$.
\end{lemma}
\begin{proof}
By Section~\ref{section:LR for (m,3,0,0)}, $\Sigma_{(5,3,0,0)}\mQ$ is simple, hence stable. Therefore, we only need to check that the first extension space is indeed 20 dimensional. As a matter of fact, $\Sigma_{(5,3,0,0)}\mQ$ corresponds to the value $m=5$ in $\Sigma_{(m,3,0,0)}$, which is the first one in the series for which we find some indeterminacy.
If we decompose $\mE nd(\Sigma_{(5,3,0,0)}\mQ)$ in irreducible factors, we get a direct sum of 60 factors. Of these, only two create problems (and they are dual to each other). These are $\Sigma_{(10,6,2,2)}\mQ \otimes \mO(-5)$ and  $\Sigma_{(8,8,4,0)}\mQ \otimes \mO(-5)$. We analyze here the first one, since the second one will completely analogous.

We can easily compute that
\[
\oH^k\left(\Gr(2,6), \W^4 \Sym^3\widetilde{\mU} \otimes \Sigma_{(10,6,2,2)}\widetilde{\mQ} \otimes \mO(-5)\right)= \begin{cases}
    0 \ \textrm{ if } k \neq 6 \\
    \Sigma_{(6,4,4,4,0,0)}V_6^{\vee} \textrm{ if } k = 6\, ,
\end{cases}
\]

\[
\oH^k\left(\Gr(2,6), \W^3 \Sym^3\widetilde{\mU} \otimes \Sigma_{(10,6,2,2)}\widetilde{\mQ} \otimes \mO(-5)\right)= \begin{cases}
    0 \ \textrm{ if } k \neq 5 \\
    \Sigma_{(6,4,3,2,0,0)}V_6^{\vee} \textrm{ if } k = 5
\end{cases}
\]

and

\[
\oH^k\left(\Gr(2,6), \W^2 \Sym^3\widetilde{\mU} \otimes \Sigma_{(10,6,2,2)}\widetilde{\mQ} \otimes \mO(-5)\right)= \begin{cases}
    0 \ \textrm{ if } k \neq 4,5 \\
    \Sigma_{(6,3,3,0,0,0)}V_6^{\vee} \textrm{ if } k = 5\\
    \Sigma_{(6,2,2,2,0,0)}V_6^{\vee} \textrm{ if } k = 4\, .\\
\end{cases}
\]

Denoting with $C_3$ the last kernel bundle coming from the splitting of the Koszul sequence (and similarly with $C_2$), we therefore obtain a sequence
\[
0 \to \Sigma_{(6,4,3,2,0,0)}V_6^{\vee} \to \oH^5(C_3) \to  \Sigma_{(6,4,4,4,0,0)}V_6^{\vee} \to 0,
\]
and moreover $\oH^i(C_3)=0$ for $i\neq5$.

Putting all these information together, we obtain a sequence in cohomology as follows:
\[
0 \to \Sigma_{(6,2,2,2,0,0)}V_6^{\vee} \to \oH^4(C_2) \to  \Sigma_{(6,4,4,4,0,0)}V_6^{\vee}  \oplus \Sigma_{(6,4,3,2,0,0)}V_6^{\vee} \to  \Sigma_{(6,3,3,0,0,0)}V_6^{\vee} \to \oH^5(C_2) \to 0.
\]
The indeterminacy is therefore resolved (and with it, the Lemma proved) if we are able to show that the map
\[
\Sigma_{(6,4,4,4,0,0)}V_6^{\vee}  \oplus \Sigma_{(6,4,3,2,0,0)}V_6^{\vee} \to  \Sigma_{(6, 3,3,0,0,0)}V_6^{\vee}
\]
is surjective, which is possible if we look at the dimensions ($134750$ vs $28875$). Equivalently, we are looking at the injectivity of the dual map 

\[
\varphi\colon \Sigma_{(6,6,6,3,3,0)}V_6 \to  \Sigma_{(6,6,2,2,2,0)}V_6  \oplus \Sigma_{(6,6,4,3,2,0)}V_6.
\]

We claim that the projection of the above map to the second component is in fact already injective, which will conclude the proof.

We first map $u \in  \Sigma_{(6,6,6,3,3,0)}V_6$
to $ \phi(u)=u\otimes \sigma^t \in \Sigma_{(6,6,6,3,3,0)}V_6 \otimes \Sym^3 V_6$. This map is injective for a general $\sigma^t $, where we are reasoning as in \cite[Appendix B]{kuzmanmar}. Moreover, again by generality of $\sigma$, the map $\phi$ is non-zero on any irreducible component of $\Sigma_{(6,6,6,3,3,0)}V_6 \otimes \Sym^3 V_6$. 

If we decompose $\Sigma_{(6,6,6,3,3,0)}V_6 \otimes \Sym^3 V_6$ in irreducibles factors, then we observe that $\Sigma_{(6,6,4,3,2,0)}V_6$ appears with multiplicity 1. Let us denote by $\pi$ the projection from $\Sigma_{(6,6,6,3,3,0)}V_6 \otimes \Sym^3 V_6$ to $\Sigma_{(6,6,4,3,2,0)}V_6$. Then the map $\varphi$ must be obtained as the composition $\varphi= \pi \circ \phi $, which is injective since $\pi$ is non-zero on the image of $\phi$.

\end{proof}

The proof above solves completely the case $\Sigma_{(5,3,0,0)}\mQ$, corroborating our claim leading to Conjecture~\ref{conjecture:simple+ext 1}. However this proof relies on a sophisticated analysis on the map involved, which cannot be transformed into a general argument for any partition. Moreover, in general the sequences arising from the Koszul complex are longer, and the maps involved are way more complicated to analyze, as already the case $\Sigma_{(4,4,0,0)}\mQ$ shows.


\section{Obstructions and smoothness}\label{section:Kuranishi smooth}
In this final section we speculate on the structure of the Kuranishi space of the vector bundles $\Sigma_\lambda\mQ$.
Let us recall the following result.
\begin{proposition}[\protect{\cite[Theorem~4.8]{MeaOno}}]
    Let $E$ be a stable and modular vector bundle on $X$. Then the DG Lie algebra controlling the infinitesimal deformations of $E$ is formal.

    In particular the Kuranishi space $\operatorname{Def}_E$ of infinitesimal deformations of $E$ is smooth if the Yoneda product 
    \[ \cup\colon\Ext^1(E,E)\times\Ext^1(E,E)\longrightarrow\Ext^2(E,E) \]
    is skew-symmetric.
\end{proposition}
The last claim in the statement above follows directly form the formality claim, see for example the end of \cite[Section~1.2]{MeaOno}.

According to Theorem~\ref{theorem A} and Conjecture~\ref{conjecture:simple+ext 1}, the hypothesis of the proposition are always satisfied for $E=\Sigma_\lambda\mQ$. Moreover, according to Proposition~\ref{prop:W3 in Ext1} we always have an inclusion
\[ \W^3 V_6^\vee\subset\Ext^1(\Sigma_\lambda\mQ,\Sigma_\lambda\mQ). \]

Our first remark is the following.

\begin{lemma}\label{lemma:W2 di W3V6 in Ext2}
    Let $\lambda=(m,t,s,0)$ be a partition. In the Littlewood--Richardson decomposition of $\mE nd(\Sigma_\lambda\mQ)$ we always find the following factors:
    \begin{enumerate}
        \item for any $m\geq2$, $\Sigma_{(m+2,m,m,m-2\mid m,m)}$ with multiplicity at least $1$;
        \item for any $m\geq0$, $\Sigma_{(m,m,m,m\mid m,m)}$ with multiplicity $1$.
    \end{enumerate}
\end{lemma}
\begin{proof}
The proof is similar to the proof of Proposition~\ref{prop:W3 in Ext1}.



\end{proof}


The cohomology of $\Sigma_{(2m,m,m,0\mid m,m)}$ has been computed in Lemma~\ref{lemma: (2n n n 0 | n n)}; on the other hand the cohomology of $\Sigma_{(m,m,m,m\mid m,m)}$ is just the cohomology of $\mO_X$. Let us then complete the picture by computing the cohomology of $\Sigma_{(m+2,m,m,m-2\mid m,m)}$.

\begin{lemma}\label{lemma: m+2 m m m-2}
    For every $m\geq2$ we have
    \[ \oH^j(X,\Sigma_{(m+2,m,m,m-2\mid m,m)})=\left\{
    \begin{array}{ll}
      0   & \mbox{if } j\neq2\,. \\
      \Sigma_{(2,2,1,1,0,0)}V_6^\vee   & \mbox{if } j=2
    \end{array} \right. \]
\end{lemma}
\begin{proof}
    Easy application of the Borel--Weil--Bott theorem.
\end{proof}

As a corollary, we get the following result.

\begin{proposition}\label{prop:W^2V_6 in Ext^2}
    For every partition $\lambda=(\lambda_1,\lambda_2,\lambda_3,\lambda_4)$, with $\lambda_1\geq2$, we have that 
    \[ \W^2\left(\W^3 V_6^\vee\right)\subset\Ext^2(\Sigma_\lambda\mQ,\Sigma_\lambda\mQ) \]
    with multiplicity $1$.
\end{proposition}
\begin{proof}
    First of all, let us recall that 
    \[ \W^2\left(\W^3 V_6^{\vee}\right) = \Sigma_{(2,2,1,1,0,0)} V_6^{\vee} \oplus \CC \]
    and that, by Lemma~\ref{lemma: m+2 m m m-2}, $\oH^2(\Sigma_{(m+2,m,m,m-2\mid m,m)})=\Sigma_{(2,2,1,1,0,0)}V_6^\vee$. 

    By Lemma~\ref{lemma:W2 di W3V6 in Ext2} we eventually have
    \[ \oH^2(\Sigma_{(m+2,m,m,m-2\mid m,m)})\oplus\oH^2(\Sigma_{(m,m,m,m\mid m,m)})\subset\oH^2(\mE nd(\Sigma_\lambda\mQ,\Sigma_\lambda\mQ)) \]
    with multiplicity $1$, from which the claim follows.
\end{proof}
Notice that if $\lambda_1=1$, then there are only four cases:
\begin{itemize}
    \item if $\lambda=(1,0,0,0)$, then $\Sigma_\lambda\mQ=\mQ$ is rigid and unobstructed;
    \item if $\lambda=(1,1,0,0)$, then $\Sigma_\lambda\mQ=\W^2\mQ$ has $20$ moduli and $\Ext^2(\W^2\mQ,\W^2\mQ)_0\cong\CC$, so that $\W^2\mQ$ is unobstructed (\cite[Theorem~1.1]{Fat24});
    \item if $\lambda=(1,1,1,0)$, then $\Sigma_\lambda\mQ=\mQ(1)$ is rigid and unobstructed;
    \item if $\lambda=(1,1,1,1)$, then $\Sigma_\lambda\mQ=\mO_X$ is rigid and unobstructed.
\end{itemize}

It is very tempting to speculate that the image of the Yoneda morphism is exactly the factor $\W^2\left(\W^3 V_6^\vee\right)\subset\Ext^2(\Sigma_\lambda\mQ,\Sigma_\lambda\mQ)$, and we do so by setting the following conjecture.

\begin{conjB}
    Let $X\subset\Gr(2,V_6)$ be the Fano variety of lines of a smooth cubic fourfold. For any partition $\lambda=(\lambda_1,\lambda_2,\lambda_3,\lambda_4)$, let us consider the associated vector bundle $\Sigma_\lambda\mQ$ on $X$. Then the Kuranishi space 
    \[ \operatorname{Def}(\Sigma_\lambda\mQ) \]
    of infinitesimal deformations of $\Sigma_\lambda\mQ$ is smooth.
\end{conjB}

\appendix

\section{Vector bundles of the form $\Sigma_\mu\mU$}\label{section:U}

In the main part of the manuscript we only considered vector bundles of the form $\Sigma_\lambda\mQ$, but the reader could ask the reason for this choice. This section provide a simple answer to this question.

\begin{proposition}\label{prop:U not modular}
    Let $\lambda=(\lambda_1,\lambda_2,\lambda_3,\lambda_4)$ and $\mu=(\mu_1,\mu_2)$ be two partitions. Then the vector bundle
    \[ \Sigma_\lambda\mQ\otimes\Sigma_\mu\mU \]
    on the Fano variety of lines $X$ is modular if and only if $\mu_1=\mu_2$.
\end{proposition}
When $\mu_1=\mu_2=k$ we have $\Sigma_\mu\mU=\mO_X(k)$ and the modularity follows from Theorem~\ref{theorem A}. Therefore the proposition is saying that the bundles considered in Theorem~\ref{theorem A} are essentially the only modular vector bundles (among the possible plethysms on a general $X$).

\proof[Proof of Proposition~\ref{prop:U not modular}]

Let us proceed in steps. First of all, we look at the case when $\lambda=(0,0,0,0)$.

\begin{lemma}
    If $\mu=(\mu_1,\mu_2)$ is a partition, then
    \[ \Delta(\Sigma_\mu\mU)=\frac{\rho(\mu)}{2}\rk(\Sigma_\mu\mU)^2\Delta(\mU), \]
    where
    \[ \rho(\mu)=\frac{\mu_1^2+\mu_2^2-2\mu_1\mu_2+2\mu_1-2\mu_2}{6}. \]
    In particular, $\Sigma_\mu\mU$ is not modular.
\end{lemma}
\begin{proof}
    Using the same reduction trick as in Section~\ref{section:reduction} we have
    \[ \Sigma_\mu\mU=\Sym^{\mu_1-\mu_2}\mU\otimes\mO_X(\mu_2). \]
    Since $\Delta(\Sym^{\mu_1-\mu_2}\mU\otimes\mO_X(\mu_2))=\Delta(\Sym^{\mu_1-\mu_2}\mU)$, we can reduce the proof to the symmetric case.

    As in Section~\ref{section:Sym Q}, we can use \cite[Theorem~4.8]{Dragutin} to compute the first two Chern classes of $\Sym^m\mU$:
    \begin{align*}\label{eqn:ch di Sym U}
        c_1(\Sym^m\mU)= & \frac{m}{2}(m+1)c_1(\mU) \\
        \ch_2(\Sym^m\mU)= & \frac{m(m-1)}{12}(m+1)c_1(\mU)^2+\frac{m(m+2)}{6}(m+1)\ch_2(\mU).
    \end{align*}
    It is then easy to see that
    \begin{equation*}
        \Delta(\Sym^m\mU)=\frac{m(m+2)}{12}(m+1)^2\Delta(\mU).
    \end{equation*}
    Putting 
    \begin{equation*} 
    \rho(\mu):=\frac{(\mu_1-\mu_2)(\mu_1-\mu_2+2)}{6} 
    \end{equation*}
    proves the first statement of the proof.
    
    Finally, as it is easy to check, one has $\Delta(\mU)=\frac{3}{2}c_1(\mQ)^2-\frac{1}{2}\mathsf{c}_2(X)$, from which it follows that $\Sigma_{\mu}\mU$ is not modular, unless $\mu=(\mu_1,\mu_1)$.
\end{proof}

To conclude the proof, we use the following known result.

\begin{lemma}
    Let $F$ and $G$ be two torsion free sheaves on a smooth and projective variety $X$. Then
    \[ \Delta(E\otimes F)=\rk(F)^2\Delta(E)+\rk(E)^2\Delta(F). \]
\end{lemma}
\begin{proof}
    This is an easy computation.
\end{proof}

Putting all together, we eventually get
\begin{align*}
    \Delta(\Sigma_\lambda\mQ\otimes\Sigma_\mu\mU) & =r_\mu r_\lambda\left( \frac{\delta(\lambda)}{4}\mathsf{c}_2(X)+\frac{\rho(\mu)}{2}\left(\frac{3}{2}c_1(\mQ)^2-\frac{1}{2}\mathsf{c}_2(X)\right) \right) \\
    & = r_\mu r_\lambda\frac{3\rho(\mu)}{2}c_1(Q)^2 + \frac{r_\mu r_\lambda}{4}\left( \delta(\lambda)+\rho(\mu)\right) \mathsf{c}_2(X)
\end{align*}
which is not modular unless $\mu=(\mu_1,\mu_1)$.

\endproof

\section{Combinatorial remarks and questions}\label{section:congetture verdure}

In this section we work on the Grassmannian $\Gr(2,6)$ and we consider the Schur functors $\Sigma_\lambda\tmQ$ of the tautological quotient bundle. We are interested in uderstanding the Littlewood--Richardson decomposition of
\[ \mE nd(\Sigma_\lambda\tmQ). \]
In particular our starting point is Proposition~\ref{prop:+ delta 1 non cambia}, which claims that there is an abstract inclusion
\[ \mE nd(\Sigma_\lambda\tmQ)\subset \mE nd(\Sigma_{\lambda+\delta_1}\tmQ), \]
where $\delta_1=(1,0,0,0)$ is the first fundamental weight. In particular there exists a finite set of partitions $\mathsf{K}_\lambda$ such that 
\begin{equation} \label{eqn:defK} 
\mE nd(\Sigma_{\lambda+\delta_1}\tmQ)=\mE nd(\Sigma_\lambda\tmQ)\oplus\bigoplus_{\mu\in\mathsf{K}_\lambda}\Sigma_\mu. 
\end{equation}
Moreover, in Section~\ref{section:Evidences} we noticed that the cardinality of $\mathsf{K}_\lambda$ becomes constant as the first entry $\lambda_1$ of $\lambda$ becomes bigger.

The aim of this appendix is to formalise this remark through a series a conjectures.
\medskip

From now on, as usual, without loss of generality we suppose that $\lambda=(m,t,s,0)$.

Starting from $\lambda$ we can repeatedly apply Proposition~\ref{prop:+ delta 1 non cambia} and hence consider the sequence of partitions $\{\lambda+n\delta_1\}_{n\in\mathbb{N}}$. To avoid repetitions, when considering a sequence we always think it as starting from $(t,t,s,0)$. This is implicitly what we have done in Section~\ref{section:Sym Q} and Section~\ref{section:Evidences} when considering at once partitions of the form $(m,0,0,0)$ or $(m,2,1,0)$ etc...
\medskip

It turns out that to better express the phenomena we observed, it is more convenient to write a partition $\lambda$ as a combination of the fundamental weights $\delta_1=(1,0,0,0)$, $\delta_2=(1,1,0,0)$ and $\delta_3=(1,1,1,0)$. (We do not consider $\delta_4$ since we only consider partitions whose last entry vanishes.) In other words, we write
\[
\lambda=a \delta_1+b \delta_2+c \delta_3,
\]
where \[
m=a+b+c,\qquad t=b+c\qquad\mbox{ and }\qquad s=c.
\]

With this notations it is easier, for example, to express the sequence of partitions associated to Proposition~\ref{prop:+ delta 1 non cambia}: once two positive integers $b$ and $c$ are fixed, we start with the partition $\lambda=b\delta_2+c\delta_3$ and then we repeatedly add $\delta_1$ to it.
\medskip

Our first remark is the following conjecture about the cardinality (counted with multiplicity) of the set $\mathsf{K}_\lambda$ defined in (\ref{eqn:defK}).

\begin{conjecture}\label{conjecture:1B}
    Fix two positive integers $b$ and $c$. Then the cardinality (with multiplicity) of $\mathsf{K}_{a\delta_1+b\delta_2+c\delta_3}$ is constant for every $a\geq b+c-1$.
\end{conjecture}

This conjecture determines the threshold after which the number of ``new" partitions needed in the Littlewood--Richardson decomposition of $\mE nd(\Sigma_\lambda\tmQ)$ is constant.

This motivates the following definition.

\begin{definition}
    For any positive integers $b$ and $c$, we define the natural number $\mathsf{k}_{b,c}$ as the cardinality of $\mathsf{K}_{b+c-1\delta_1+b\delta_2+c\delta_3}$,  i.e.
    \[ \mathsf{k}_{b,c}=|\mathsf{K}_{b+c-1\delta_1+b\delta_2+c\delta_3}|. \]
\end{definition}

Before continuing, let us pause on three examples.

\begin{example}
    Let us point out three explicit examples.
    \begin{enumerate}
        \item Let $\lambda=(m,0,0,0)=m\delta_1$. We saw in Section~\ref{section:Sym Q} that for any $m\geq1$,
        \[ \mE nd(\Sym^{m+1}\widetilde{\mQ})=\mE nd(\Sym^{m}\widetilde{\mQ})\oplus\Sigma_{(2m,m,m,0\,| m,m)}. \]
        Since in this case we have $b=c=0$, Conjecture~\ref{conjecture:1B} holds and $\mathsf{k}_{0,0}=1$.\\
        
        \item Let $\lambda=(m,1,0)=(m-1)\delta_1+\delta_2$, so that $b=1$ and $c=0$. 
        The case $m=1$ was described in Example~\ref{example:LR di End(m,1,0,0)} and we have
        \[ \mE nd\left(\Sigma_{(1,1,0,0)}\mQ\right)=\Sigma_{(2,2,0,0\mid 1,1)} \oplus \Sigma_{(2,1,1,0\mid 1,1)} \oplus \Sigma_{(1,1,1,1\mid 1,1)}. \]
    
        Moreover, we saw in Section~\ref{section:Sigma 2 1} that for $m=2$ we have

        \begin{align*}
        \mE nd(\Sigma_{(2,1,0,0)}\mQ)=& \mE nd\left(\Sigma_{(1,1,0,0)}\mQ\right) \oplus \Sigma_{(4,3,1,0\mid 2,2)} \oplus \Sigma_{(4,2,2,0\mid 2,2)}\oplus \\  &\Sigma_{(4,2,1,1\mid 2,2)}\oplus\Sigma_{(3,3,2,0\mid 2,2)} \oplus \Sigma_{(3,2,2,1\mid 2,2)}
      \end{align*}

       We can easily reproduce similar computations for $m=3$, which gives
\begin{align*}
        \mE nd(\Sigma_{(3,1,0,0)}\mQ)=& \mE nd\left(\Sigma_{(2,1,0,0)}\mQ\right) \oplus  \Sigma_{(6,4,2,0\mid 3,3)} \oplus \Sigma_{(6,3,3,0\mid 3,3)}\oplus \\  &\Sigma_{(6,3,2,1\mid 3,3)}\oplus\Sigma_{(5,4,3,0\mid 3,3)} \oplus \Sigma_{(5,3,3,1\mid 3,3)}
      \end{align*}

In each case we add $5$ more partitions. In fact, Conjecture~\ref{conjecture:1B} holds by Proposition~\ref{prop:LR di End(m,1,0,0)} and $\mathsf{k}_{1,0}=5$.


\item In both the two previous examples, we had $b+c-1\leq 0$. To have a more interesting example, let us consider the sequence $(m,2,0,0)=(m-2)\delta_1+2\delta_2$. In Example~\ref{example: LR di End (m,2,0,0)} we saw that $\mE nd(\Sigma_{(2,2,0,0)}\mQ)$ has 6 irreducible factors, while $\mE nd(\Sigma_{(3,2,0,0)}\mQ)$ has 18 irreducible factors: we need to add $12$ partitions. Similarly we can see that $\mE nd(\Sigma_{(4,2,0,0)}\mQ)$ has 32 components and $\mE nd(\Sigma_{(5,2,0,0)}\mQ)$ has 46 components: in both these cases we need to add $14$ partitions. In fact  Conjecture~\ref{conjecture:1B} holds by Proposition~\ref{prop:LR di End(m,2,0,0)} and $\mathsf{k}_{2,0}=14$.

    \end{enumerate}
\end{example}

In Section~\ref{section:Sym Q} and Sections from \ref{section:LR for (m,1,0,0)} to \ref{section:LR for (m,3,3,0)} we verify Conjecture~\ref{conjecture:1B}, and compute the numbers $\mathsf{k}_{b,c}$, in several cases that we list in Table~\ref{table:values for k}.

\FloatBarrier
\begin{table}[h!bt]
\caption{Values of $\mathsf{k}_{b,c}$}
\label{table:values for k}
\begin{tabular}{ccc} 
\toprule
$\lambda$ & $(b,c)$ & $\mathsf{k}_{b,c}$   \\
					\midrule
   $(m,0,0,0)$ & $(0,0)$ & 1 \\  
   $(m,1,0,0)$ & $(1,0)$ & 5 \\	
   $(m,1,1,0)$ & $(0,1)$ & 5 \\
   $(m,2,0,0)$ & $(2,0)$ & 14 \\
   $(m,2,1,0)$ & $(1,1)$ & 26 \\
   $(m,2,2,0)$ & $(0,2)$ & 14 \\
   $(m,3,0,0)$ & $(3,0)$ & 30 \\
   $(m,3,1,0)$ & $(2,1)$ & 71 \\
   $(m,3,2,0)$ & $(1,2)$ & 71 \\
   $(m,3,3,0)$ & $(0,3)$ & 30 \\
					\bottomrule
				\end{tabular}
\end{table}

One peculiar feature appearing from Table~\ref{table:values for k} is that the numbers $\mathsf{k}_{b,c}$ are symmetric in $b$ and $c$. We observed the same symmetry in many more cases, which lead us to state the following conjecture.

\begin{conjecture}\label{conjecture:symmetry}
    For any positive integer $b$ and $c$ we have $\mathsf{k}_{b,c}=\mathsf{k}_{c,b}$, i.e.\ the numbers $\mathsf{k}_{b,c}$ are symmetric in $b$ and $c$.
\end{conjecture}

With the help of computer algebra we computed the numbers $\mathsf{k}_{b,c}$ for approximately one hundred values of $b$ and $c$. What we observed is that (beside confirming Conjectures~\ref{conjecture:1B} and~\ref{conjecture:symmetry}) these values are not casual, but seem to be perfectly interpolated by a polynomial of degree $5$. 

We state in the following final conjecture the shape of this polynomial.

\begin{conjecture}
    We have
    \[ \mathsf{k}_{b,c}=\left\{
    \begin{array}{ll}
    f(b,c) & \mbox{if } b\geq c \\
    f(c,b) & \mbox{if } b< c
    \end{array}\right.
    \]
    where
    \begin{align*}
        f(b,c)= & \frac{1}{60}\left( 60+134c+90c^2+15c^3+c^5 \right)+\frac{1}{12}\left( 26+62c+45c^2+8c^3-c^4 \right)b\, +\\
        & \frac{1}{2}\left( 3+7c+5c^2+c^3 \right)b^2 +\frac{(c+1)^2}{3}b^3.
    \end{align*}
\end{conjecture}

\begin{remark}
    Proposition~\ref{prop:+ delta 1 non cambia} seems to hold more generally if instead of $\delta_1$ we add to $\lambda$ the weights $\delta_2$ and $\delta_3$.
    Working out explicit cases with the help of computer algebra, it seems that an analog of Conjecture~\ref{conjecture:1B} does not hold for $\delta_2$. On the other hand, it seems that a $\delta_3$-version of Conjecture~\ref{conjecture:1B} should hold.
\end{remark}

\begin{remark}
    So far we worked in the grassmannian $\Gr(2,6)$. On the other hand there is nothing special about this choice. One can naturally wonder whether Proposition~\ref{prop:+ delta 1 non cambia} holds more generally for any grassmannian $\Gr(k,n)$. Working out some small values of $k$ and $n$, it seems that the answer is affirmative, hence it may be worth to investigate the features outlined in this appendix in bigger generality.
\end{remark}

\bibliographystyle{alpha}

\end{document}